\documentclass[a4paper]{article}
 
\usepackage{graphicx}
\usepackage{amsthm,amsmath}

\usepackage{amsmath}
\usepackage{amssymb}
\usepackage{latexsym}
\usepackage[all]{xy}
\usepackage{makeidx}
\usepackage{mathrsfs}

 \usepackage{hyperref}[colorlinks=true,linkcolor=purple]
 \usepackage{stmaryrd}

\makeindex

\usepackage{tikz-cd}
\usetikzlibrary{cd}

\usepackage{tikz}
\usetikzlibrary{matrix,arrows,decorations.pathmorphing}

\theoremstyle{cupthm}
\newtheorem{theorem}{Theorem}[section]
\newtheorem{proposition}[theorem]{Proposition}
\newtheorem{lemma}[theorem]{Lemma}
\newtheorem{corollary}[theorem]{Corollary}
\newtheorem{convention}[theorem]{Convention}
\newtheorem{claim}[theorem]{Claim}
\newtheorem*{theorem*}{Theorem}
\theoremstyle{definition}
\newtheorem{definition}[theorem]{Definition}
\newtheorem{example}[theorem]{Example}
\newtheorem{remark}[theorem]{Remark}

\usepackage{xcolor}

\definecolor{darkred}{rgb}{0,0,0} 
\definecolor{darkgreen}{rgb}{0,0,0}
\definecolor{darkblue}{rgb}{0,0,0}

\usepackage{geometry}
\usepackage{marginnote}

\begin{document}

\title{Irreducible generating tuples of Fuchsian groups}

\author{Ederson Dutra{\thanks{The first author was supported by FAPESP, S\~{a}o Paulo Research Foundation, grants  2018/08187-6 and 2021/12276-7.}} \  and Richard Weidmann}

\date{\today}

\maketitle

\begin{abstract}
L. Louder showed in \cite{Lou} that any  generating tuple of a surface group is Nielsen equivalent to a stabilized standard generating tuple i.e.~$(a_1,\ldots ,a_k,1\ldots, 1)$ where $(a_1,\ldots ,a_k)$ is the standard generating tuple. This implies in particular that irreducible generating tuples, i.e. tuples that  are not Nielsen equivalent to a  tuple of the form $(g_1,\ldots ,g_k,1)$, are minimal. 
In \cite{Dut} the first author generalized Louder's ideas and showed that all irreducible and non-standard generating tuples of sufficiently large Fuchsian groups can be represented by so-called almost orbifold covers endowed with a rigid generating tuple.

In the present paper a variation of the ideas from \cite{W2} is used to show that this almost orbifold cover with a rigid generating tuple is unique up to the appropriate equivalence. It is moreover shown that  any such generating tuple  is  irreducible. This provides a way to exhibit many Nielsen classes of non-minimal irreducible generating tuples for Fuchsian groups.

As an application we show that generating tuples of fundamental groups of  Haken Seifert manifolds corresponding to irreducible horizontal Heegaard splittings are irreducible.
\end{abstract}

%-----------------------------------------------------------

\section{Introduction}
Studying Nielsen equivalence classes of generating tuples of surface groups and more generally Fuchsian groups has a long history, starting with the work of Zieschang \cite{Zie} on fundamental groups of orientable surfaces. He showed that any minimal generating tuple of a surface group is Nielsen equivalent to the standard generating tuple. Variations of the cancellation methods developed by Zieschang were then successfully employed by Rosenberger to study Nielsen classes of minimal generating tuples for many Fuchsian groups \cite{R1,R2,R3}.   Nielsen classes of minimal generating tuples were  then studied by Lustig and Moriah using innovative algebraic ideas \cite{L,LM1,LM2}; recently this lead to a classification in all but a few exceptional cases \cite{LM3}.

Non-minimal generating tuples where first studied in the groundbreaking work of Louder \cite{Lou} who proved that any generating tuple of a surface group is Nielsen equivalent to a stabilized standard tuple, i.e.~a tuple of the form $(a_1,\ldots ,a_k,1\ldots, 1)$ where $(a_1,\ldots ,a_k)$ is the standard generating tuple. In particular any two generating tuples of the same size are Nielsen equivalent, thus a true analogue of Nielsen's theorem for free groups holds for surface groups. Louder's proof can be thought of as a folding argument in an appropriate category of square complexes and he shows that any square complex representing a generating tuple can be folded and unfolded onto a square complex representing a stabilized standard generating tuple.

In the case of Fuchsian groups the situation is more subtle and more interesting. In \cite{Dut}  the first author generalized the ideas of Louder to the context of sufficiently large Fuchsian groups, i.e.~Fuchsian groups that are not triangle groups. Dutra proved that any non-standard irreducible generating tuple can be represented by a so called almost orbifold cover with a rigid generating tuple, in the case that all elliptic elements  are of order 2 this implies a direct generalization of Louder's result. Recall that a tuple is reducible if it is equivalent to a tuple of the form $(g_1,\ldots, g_k,1)$ and irreducible otherwise. Clearly minimal generating tuples are irreducible, the converse does not hold in general.

Almost orbifold covers are branched maps that are close to being orbifold covers.

\begin{definition}
Let $\mathcal O$ be a closed cone-type 2-orbifold and $\mathcal O'$ a compact cone-type 2-orbifold with a single boundary component. A map  $\eta:\mathcal O'\to\mathcal O$ is called an \emph{almost orbifold cover}  if exists a  point $x\in \mathcal O$ and a  closed disk $D\subset \mathcal O$ containing $x$  such that $D\setminus\{x\}$ contains no cone points such the following hold:
\begin{enumerate}
\item $\eta^{-1}(D)=D_1\cup\ldots\cup D_m\cup S$ where $D_i$ is a disk for all $1\le i\le m$ and $S=\partial \mathcal O'$.

\item $\eta|_{\mathcal O'\setminus \eta^{-1}(D)}: \mathcal O'\setminus \eta^{-1}(D)\to\mathcal O\setminus D$ is an orbifold cover.

\item $\eta|_{D_1\cup\ldots\cup D_m}:D_1\cup\ldots\cup D_m\to D$ is an orbifold cover.
\item $\eta|_S:S\to \partial D$ is a cover.
\end{enumerate}

The \emph{degree} of $\eta$ is defined as the degree of the cover  $\eta|_{\mathcal O'\setminus \eta^{-1}(D)}$. We say that $\eta$ is a \emph{special  almost orbifold cover}  if $\deg(\eta|_S)<m$ and $\deg(\eta|_S)$ does not divide $m$,  where $m$ is the order of $x$.
\end{definition}

\begin{remark}
Throughout this paper we are mostly interested in the case where the  point $x\in \mathcal O$ is a cone point. This is particular always the case if $\eta$ is  a special almost orbifold covering.
\end{remark}

\begin{remark}
It is easily verified that  $\eta$ can be extended to an orbifold cover by gluing in a disk with cone point of order $\frac{m}{\deg(\eta|_S)}$ iff $\deg(\eta|_S)$ divides $m$.
\end{remark}

\begin{example}{\label{ex:01}}
Let $\mathcal O=T^2(15, 14)$ and  $\mathcal O'=F(15, 14 , 7)$  where $F$ is a once punctured  orientable surface of  genus two.  Consider  the  map $\eta:\mathcal O'\to\mathcal O$ described in Fig.~\ref{ex:almost} where the effect of $\eta$ on the component containing the cone point of order $7$ is described in Fig.~\ref{ex:almost0}.  Let   $x \in \mathcal O$ be the cone point of order $15$ and $D$  be the disk depicted in Fig.~\ref{ex:almost}. Then  $\eta^{-1}(D)=D_1\cup S$ and  $\eta$ defines  an orbifold cover $ \mathcal O'\setminus D_1\cup S\to \mathcal O\setminus D$  of degree three.  Thus $\eta$ is an almost orbifold cover of degree $3$. As  $\eta|_S:S\to \partial D$ is of degree two we  conclude that $\eta$ is special. 
\begin{figure}[h!] 
\begin{center}
\includegraphics[scale=1]{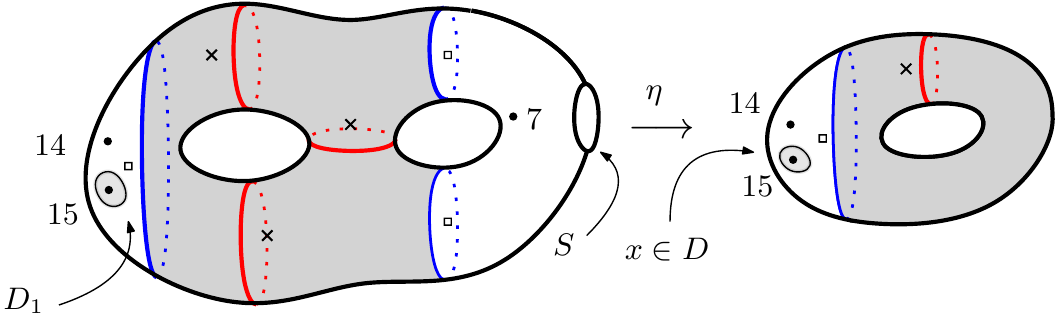}
\caption{$\eta$ is special of degree $3$.}{\label{ex:almost}}
\end{center}
\end{figure}
\begin{figure}[h!] 
\begin{center}
\includegraphics[scale=1]{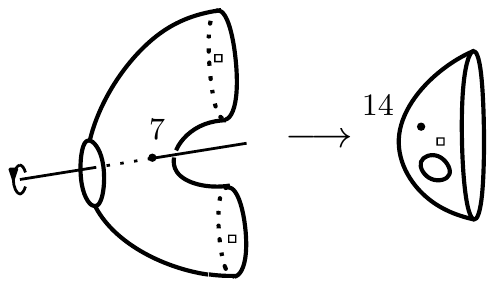}
\caption{A half turn rotation. }{\label{ex:almost0}}
\end{center}
\end{figure}
\end{example}

\begin{definition} 
A marking of $\mathcal O$ is a pair  
 $(\eta :\mathcal O'\rightarrow \mathcal O, [T'])$ 
where $\eta$ is an almost  orbifold cover and $[T']$ is a Nielsen equivalence class of generating tuples of $\pi_1^o(\mathcal O')$. We say that  the Nielsen class $[\eta_{\ast}(T')]$   is represented by  the marking $(\eta:\mathcal O'\rightarrow \mathcal O, [T'])$.
\end{definition}

We say that a generating tuple $T$ of the fundamental group of  a  compact cone-type $2$-orbifold with $q\geq 1$ boundary components is \emph{rigid}  if $T$ is not Nielsen-equivalent to a tuple  $(g_1,\ldots g_l, \gamma_1,\ldots ,\gamma_q)$  where $\gamma_1,\ldots , \gamma_q$ correspond  to the boundary components.

\begin{definition}
Let  $(\eta :\mathcal O'\rightarrow \mathcal O, [T'])$ be a marking of the orbifold $\mathcal O$.
\begin{enumerate}
\item  We say  that  $(\eta :\mathcal O'\rightarrow \mathcal O, [T'])$ is \emph{standard} if the following hold:
\begin{enumerate}
\item $ \eta $ has degree one. 

\item if $\mathcal O$ is not a surface, then the exceptional point of $\eta $ is of order $\ge2$.

\item some  (and therefore any) tuple in $[T']$ is minimal.
\end{enumerate}

\item We say that $(\eta :\mathcal O'\rightarrow \mathcal O, [T'])$  is \emph{special} if $\eta $ is  special and $[T']$ consists of rigid generating tuples of $\pi_1^o(\mathcal O')$.
\end{enumerate}
\end{definition}

\begin{definition}
We say that a generating tuple $T$ of $\pi_1^o(\mathcal O)$ is \emph{standard} if there is a standard marking $(\eta :\mathcal O'\rightarrow \mathcal O, [T'])$ of $\mathcal O$  such  that  $[\eta_\ast (T')]=[T]$.
\end{definition}

In \cite{Dut} it is shown that any non-standard  irreducible generating tuple of a sufficiently large 2-orbifold is represented by a special marking. The main purpose of this paper is to establish the uniqueness of this marking and the fact that generating tuples represented by special markings are irreducible:

\begin{theorem}{\label{thm01}}
Let $\mathcal O$ be a sufficiently large cone type 2-orbifold and let $T$ be a non-standard irreducible generating tuple of $\pi_1^o(\mathcal O)$.  
Then there is a unique special marking  $(\eta:\mathcal O'\rightarrow \mathcal O, [T'])$  such that $[T]=[\eta_{\ast}(T')].$ 
\end{theorem}

\begin{theorem}{\label{thm02}}
Let $\mathcal O$ be a sufficiently large cone type 2-orbifold. If $$(\eta :\mathcal O'\rightarrow \mathcal O, [T'])$$ is a special marking  such that  $\eta_{\ast}'$ is  surjective, then $\eta_{\ast}(T')$ is irreducible.
\end{theorem}

\begin{example}
Let $\eta:\mathcal O'\to \mathcal O$ be the special almost orbifold cover given in Example~\ref{ex:01}.   A presentation for $\pi_1^o(\mathcal O)$ is given by $$ \langle a_1, b_1, s_1, s_2 \ | \ s_1^{15}, s_2^{14}, s_1s_2=[a_1, b_1]\rangle.$$   
Consider the generating tuple  $$T'=(\sigma_1, \sigma_2 ,   \alpha_1, \alpha_2, \beta_1, \gamma_1, \sigma_3)$$ of $\pi_1^o(\mathcal O')$ as described in Fig.~\ref{ex:almost1}.
 \begin{figure}[h!]
\begin{center}
\includegraphics[scale=1]{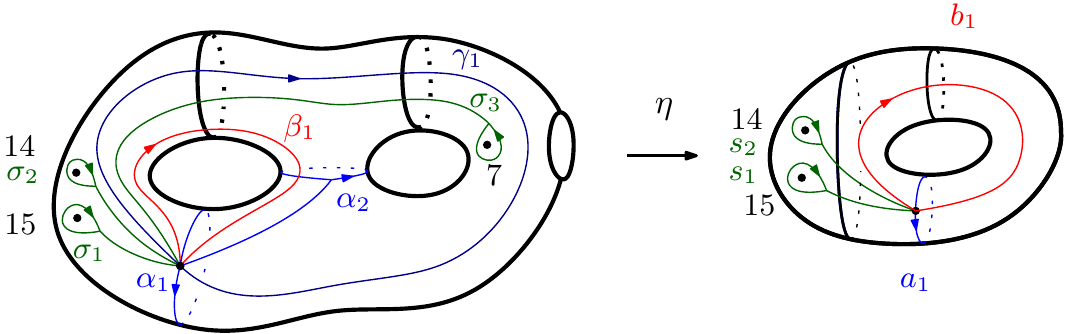}
\caption{The tuple $T'$.}{\label{ex:almost1}}
\end{center} 
\end{figure}  
Then  
$$\eta_{\ast}(T')=(s_1, s_2 , a_1,  b_1^{-1} a_1b_1 , b_1^3,  b_1s_1b_1, b_1 s_2^2 b_1^{-1}).$$ 
Using the relation $s_1=[a_1,b_1]s_2^{-1}$ in $\pi_1^o(\mathcal O)$ we can rewrite $\eta_{\ast}(T') $ as    
$$ (a_1b_1a_1^{-1}b_1^{-1}s_2^{-1}, s_2, a_1, b_1^{-1}a_1b_1, b_1^3, b_1a_1b_1a_1^{-1}b_1^{-1}s_2^{-1}b_1, b_1s_2^2b_1^{-1}).$$
Thus $\eta_{\ast}(T')$ is Nielsen equivalent to
 $$(b_1a_1^{-1}b_1^{-1}, s_2, a_1, b_1^{-1}a_1b_1, b_1^3, b_1, b_1s_2^2b_1^{-1})$$ 
which can be shown to be  Nielsen equivalent to $(a_1, b_1, s_2, 1, 1, 1, 1)$. Therefore,  $\eta_{\ast}(T')$ is reducible. This   shows that the restriction to  rigid generating tuples  of $\pi_1^o(\mathcal O')$ is necessary to guarantee that the corresponding tuple in $\pi_1^o(\mathcal O)$ is irreducible.  On the other hand, according to Lemma~\ref{allprimitive},    
 $$T':=  (\sigma_1^{2}, \sigma_2^{3},   \alpha_1, \alpha_2, \beta_1, \gamma_1, \sigma_3^{2 })$$ 
is  a  rigid generating tuple of $\pi_1^o(\mathcal O')$. The previous Theorem  therefore implies that  
 $$\eta_{\ast}(T')=(s_1^2, s_2^3, a_1,  b_1^{-1} a_1b_1 , b_1^3,  b_1s_1b_1, b_1 s_2^4 b_1^{-1}).$$
is irreducible. This is an example of a non-minimal irreducible generating tuple of $\pi_1^o(\mathcal O)$. 
\end{example}

We expect that our approach can also be used to classify standard generating tuples up to Nielsen equivalence, completing the work of Moriah and Lustig. We plan to address this question in a future paper. While the underlying ideas we employ should also be able to study generating tuples of triangle groups, it is clear that new language needs to be developed to carry out this approach in  these remaining cases.

\smallskip

There is an interesting relationship between almost orbifold covers and horizontal Heegaard splittings of Seifert 3-manifolds. If $M$ is a Seifert 3-manifold and $T$ is generating tuple of $\pi_1(M)$ corresponding to a horizontal Heegaard splitting then the image $T'$ of $T$ in the fundamental group of the base orbifold is naturally represented by an almost orbifold cover which is induced by the Heegaard splitting. Theorem~\ref{thm02} gives us therefore a way to establish the irreducibility of $T'$ and therefore of $T$. We obtain the following:

\begin{theorem}{\label{thm03}} Let $M$ be a Haken orientable Seifert 3-manifold with orientable base space and $T$ be a tuple corresponding to a horizontal Heegaard splitting. Then  $T$ is irreducible if and only if the Heegaard splitting is irreducible.
\end{theorem} 

Note that we need to exclude the small Seifert manifolds as their base orbifolds are not sufficiently large. We actually give a new proof of Theorem 8.1 of \cite{Se} in this setting.

%-------------------------------------------------------------

\section{Strategy of proof}

In this section we briefly sketch the strategy of the proof of the main theorems. At this point we cannot introduce the subtle notions needed to be precise. Thus this section remains vague, maybe even pointless.

Nielsen's theorem states that any two generating $m$-tuples of the free groups $F_n$ are Nielsen equivalent. The folding proof of Nielsen's theorem (in the case $m=n$) goes as follows): Identify the free group with $\pi_1(R_n)$ where $R_n$ is the rose with $n$ petals. Any generating $n$-tuple $T$ can be represented by  a tuple $(\Gamma,u_0,Y,E,f)$ where $\Gamma$ is a graph with base point $u_0\in V\Gamma$, $Y\subset \Gamma$ is a maximal tree, $E=(e_1,\ldots ,e_n)$ is an orientation of $E\Gamma\setminus EY$  and $f:\Gamma\to R_n$ is a $\pi_1$-surjective morphism. The represented tuple is $(f_*(s_1),\ldots ,f_*(s_n))$ where $s_i\in\pi_1(\Gamma,u_0)$ is represented by the closed path whose only edge not contained in $Y$ is $e_i$. Now the graph can be folded onto $R_n$ using Stallings folds \cite{St} yielding a sequence $$\Gamma=\Gamma_0,\Gamma_1,\ldots ,\Gamma_{k-1},\Gamma_k=R_n$$ such that $\Gamma_{i+1}$ is obtained from $\Gamma_i$ by a single Stalling fold $p_i$. One observes that we can define tuples $(\Gamma_i,u_0^i,Y_i,E_i,f_i)$ with $f_{i+1}\circ p_i=f_i$ for all $i$  such that all tuples represent the same Nielsen class. As the initial tuple defines $T$ and the terminal tuples defines the  standard basis, the theorem follows. 
Note that it is crucial that any folding sequence yields essentially the same (folded) object. Thus there is only one Nielsen-class and therefore there  is no need to distinguish distinct classes.

\smallskip

 A similar argument can be used to prove Grushko's theorem \cite{Gr} which states that any generating tuple of a free product $A*B$ is Nielsen equivalent to a tuple of the form $(a_1,\ldots ,a_k,b_1,\ldots ,b_l)$. This argument relies on folding sequences in the category of (marked) graphs of groups with trivial edge groups, see \cite{Du} and \cite{BF} where folding sequences for graphs of groups were studied. 
In this case one obtains a sequence of graphs of groups with trivial edge groups where each vertex is marked with a generating tuple of its vertex group and where the terminal object is the 1-edge graph of groups corresponding to the free product $A*B$ and the vertices are marked with generating tuples of $(a_1,\ldots ,a_k)$ of $A$ and $(b_1,\ldots ,b_l)$ of $B$, respectively. Unlike in the case of free groups, the terminal object is not unique, however in the case of irreducible generating tuples it was shown in \cite{W2} that any terminal object that occurs is marked by generating tuples $(a_1',\ldots ,a_k')$ and $(b_1',\ldots ,b_l')$ that are Nielsen-equivalent to $(a_1,\ldots ,a_k)$  and $(b_1,\ldots ,b_l)$, respectively. This implies the strongest possible uniqueness result for the output of Grushko's theorem. The simple but somewhat subtle idea underlying the proof in \cite{W2} relies on considering appropriate equivalence classes of marked graphs or groups (or more precisely $\mathbb A$-graphs) to be vertices where vertices are connected by edges if one can be obtained by the other by a  fold. The  result then follows by observing connectivity of the graph and establishing that there is unique vertex of minimal complexity.

\smallskip

The proofs of Louder \cite{Lou} and Dutra \cite{Dut} can be thought of as a variation of the proof of Grushko, but in a different category. It is shown that any Nielsen equivalence class of irreducible generating tuples is represented by a folded object which represents a standard generating tuple in the case of surfaces (Louder) or is a special marking in the case of sufficiently large cone type 2-orbifolds. In the first case the uniqueness is then immediate and in the second case it follows if all cone points are of order 2. The main objective of this paper is to show that in the second case the folded object is unique up to the appropriate equivalence. To do so we broadly follow the strategy of \cite{W2}, however both the language developed and the arguments applied are significantly more involved.

\smallskip

We will use the language of $\mathbb A$-graph as developed in \cite{KMW} and rely on a number of results of \cite{Dut}. Moreover a  certain degree of familiarity with both of these papers is assumed. Following Louder and Dutra we decompose the fundamental group of the orbifold $\mathcal O$ under consideration as the fundamental group of a graph of groups $\mathbb A$ corresponding to a decomposition of $\mathcal O$ along essential simple closed curves. The paper has two main chapters, Chapter~\ref{sectionlocalcase} and Chapter~\ref{section_global_picture}:
\begin{enumerate}
\item In Chapter~\ref{sectionlocalcase} we study generating tuples of  fundamental groups of orbifolds with boundary, those are the groups that occur as vertex groups of $\mathbb A$. In Section~\ref{gen_tuples_orbifold_with_boundary} we collect basic facts about Nielsen classes generating tuples, in particular we characterize rigid generating tuples. In the subsequent sections we study  finitely generated subgroups of the vertex groups of $\mathbb A$ where some of the generators are assumed to be peripheral, i.e. conjugate to elements of the adjacent edge groups. We study these so-called partitioned tuples up to the natural equivalence. As the groups under consideration are fundamental groups of graphs of groups with trivial edge groups we can rely on a variation of the arguments from \cite{W2} to prove Proposition~\ref{prop:01}, i.e. to show that any equivalence class of partitioned tuples corresponds to a unique geometric situation.
\item In Chapter~\ref{section_global_picture} we encode (Nielsen equivalence classes of) irreducible non-standard generating tuples by marked $\mathbb A$-graphs up to some natural equivalence. We then show that any such Nielsen class is encoded by a unique object of minimal complexity, which corresponds to a special almost orbifold cover with a rigid generating tuple. This essentially proves Theorem~\ref{thm01} and Theorem~\ref{thm02}. In the argument the subtle results from Chapter~\ref{sectionlocalcase} are crucial.
\end{enumerate}

In Chapter~\ref{section_heegaard} we give a proof of Theorem~\ref{thm03} which is essentially a corollary of Theorem~\ref{thm02}.

%------------------------------------------------------------------

\section{The local case}\label{sectionlocalcase}

In this chapter we discuss  (partitioned) generating tuples  of fundamental groups of compact orbifolds with boundary. These groups are free products of cyclic groups.

\smallskip

The orbifold $\mathcal O$ with underlying surface $F$ and  cone points of orders $m_1, \ldots , m_r$ is denoted  by $F(m_1, \ldots, m_r)$. 
If $\mathcal O$ has $q\geq 1$ boundary components, then a  presentation (also called the standard presentation) for $G:=\pi_1^o(\mathcal O)$ is  given by:
\begin{equation}{\label{present:G}}
\langle a_1,\ldots ,a_h, c_1,\ldots ,c_q , s_1,\ldots , s_r  \ | \   s_1^{m_1},\ldots ,s_r^{m_r},  R\cdot s_1\cdot \ldots \cdot s_r =c_1\cdot \ldots \cdot c_q \rangle \tag{*} 
\end{equation}   
where  $c_1, \ldots, c_q$ correspond to the boundary components and   $s_1, \ldots, s_r$  correspond to the  cone points  of $\mathcal O$. The word $R$ and $h$ are  given by:
\begin{enumerate}
\item $h$ is a non-negative and even integer  and  $R =[a_1,a_2][a_3, a_4]\cdot \ldots \cdot [a_{h-1},a_h]$  if $F$ is orientable of genus $h/2$;

\item   $h\ge   1$ and  $R =a_1^2\cdot\ldots \cdot a_h^2$  
if $F$ is non-orientable of genus $h$.
\end{enumerate}
From   (\ref{present:G}) we readily see that for any $1\leq k\leq q$ we have    
\begin{equation}{\label{split:G}}
G=\langle a_1, \ldots, a_h, c_1, \ldots , c_{k-1}, c_{k+1}, \ldots, c_q \ | \ - \ \rangle \ast \langle s_1 \ | \  s_1^{m_1}\rangle\ast \ldots \ast \langle s_r\ | \  s_r^{m_r}\rangle \tag{**} 
\end{equation}

Let $g\in G$ be   non-trivial element of  finite order.  Thus $g=us_i^ku^{-1}$ for some $u\in G$,   $i\in \{1, \ldots, r\}$ and $k\in \{1, \ldots, m_i-1\}$. We say that $g$  is \emph{angle-minimal} if $k=\pm |\langle s_i\rangle :\langle s_i^k\rangle|$.    
A generating tuple $T$ of a non-trivial subgroup of $u\langle s_i\rangle u^{-1}$ is said to be \emph{angle-minimal}  if $T$ consist of a single angle-minimal element.%is Nielsen equivalent to $(us_i^{k}u^{-1})$ such that $us_i^{k}u^{-1}$ is angle-minimal. 

%-------------------------------------------------------------

\subsection{Generating tuples of fundamental groups of orbifolds with boundary}\label{gen_tuples_orbifold_with_boundary}

The classification of Nielsen equivalence classes of generating tuples of $G$  is therefore   given by the following theorem where the first part is a consequence of Grushko's theorem and the second part is due to Lustig \cite{L}, see also \cite{W2}.

\begin{theorem}[Grushko, Lustig]{\label{grushkolustig}} 
Let $\mathcal Q$ and $G$ as above. Then any irreducible generating tuple $T$ of $G$ is Nielsen equivalent to  
$$(a_1,\ldots ,a_h,\break s_1^{\nu_1},\ldots,s_r^{\nu_r},c_1,\ldots , c_{q-1})$$ 
with $0<\nu_i\le  \frac{m_i}{2}$ and  $gcd(\nu_i,m_i)=1$ for $1\le  i\le  r$. The $\nu_i$ are uniquely determined by the Nielsen equivalence class of $T$.
\end{theorem}

The previous theorem  enables us to characterize rigid generating tuples of $G$. Recall that these are generating tuples of $G$   that are not Nielsen equivalent to a generating tuple containing simultaneously one element corresponding to each boundary component. 
\begin{lemma}{\label{allprimitive}}  
Let $\mathcal O$ and $G$ be as above.  Let  
 $T=(a_1,\ldots ,a_h, s_1^{\nu_1},\ldots, s_r^{\nu_r},c_1,\ldots ,c_{q-1})$ 
be an irreducible  generating tuple of $G$ with $0<\nu_i\le   \frac{m_i}{2}$ and $gcd(\nu_i,m_i)=1$ for $1\le  i\le  r$.  
Then the following are equivalent:
\begin{enumerate}
\item $\nu_i =1$ for some $i\in\{1,\ldots,  r\}$.

\item $T$ is Nielsen equivalent to a tuple containing    $c_1,\ldots , c_q$.

\item There exist $g_1,\ldots ,g_q$ in $ G$ such that $T$ is Nielsen equivalent to a tuple containing $g_1c_1g_1^{-1},\ldots ,g_qc_qg_q^{-1}$.
\end{enumerate}
\end{lemma}

Before we give a proof  we need one simple lemma, it might be well-known but the authors are not aware of any reference.

\begin{lemma}\label{length2} Let $G=A*B$ be a free product and $  T=(g_1, \ldots ,g_k)$ be a generating tuple for $G$  such that $g_1=ab$ with $a\in A$ and $b\in B$. Then $ T$ is Nielsen equivalent to  $(g_1',\ldots ,g_k')$  with $g_i'\in A\cup B$ for $1\le  i\le  k$ such that either $g_1'=a$ or $g_1'=b$. We can furthermore assume that $g_j=g_j'$ for all $j$ with $g_j\in A\cup B$.
\end{lemma}
\begin{proof}[Sketch of proof]
Let $\mathbb A$ be the graph of groups with a single edge pair $\{e,e^{-1}\}$ with trivial edge group and vertices $v_0=\alpha(e)$ and $v_1=\omega(e)$ such that 
 $G_{v_0}=A$  and $G_{v_1}=B.$  
Then $\pi_1(\mathbb A, v_0)\cong G$. Clearly there exists a marked $\mathbb A$-graph $(\mathcal B, u_0)$ (see \cite{W2}) such that the following hold (see Figure~\ref{fig:lemma23}):
\begin{enumerate}
\item[(a)] $T_{u_0}$ consists precisely of the elements of $T$ that lie in $A$.

\item[(b)] There exists an edge $f\in EB$ with label $(1,e,1)$ and $\alpha(f)=u_0$ such that the tuple $T_{u_1}$  at   $u_1=\omega(f)$ consists of the elements of $T$ that lie in $B$.

\item[(c)] There exits an edge $f'\in EB$ with   label  $(a, e, b)$ such that $\alpha(f')=u_0$ and $\omega(f')=u_1$.   The  closed loop $f', f^{-1}$ therefore   corresponds to $g_1=ab$.

\item[(d)]  The elements of $\{g_2,\ldots,g_k\}\setminus(A\cup B)$ are represented by  non-degenerate  loops based at $u_0$.
\end{enumerate}
\begin{figure}[h!]
\begin{center}
\includegraphics[scale=1]{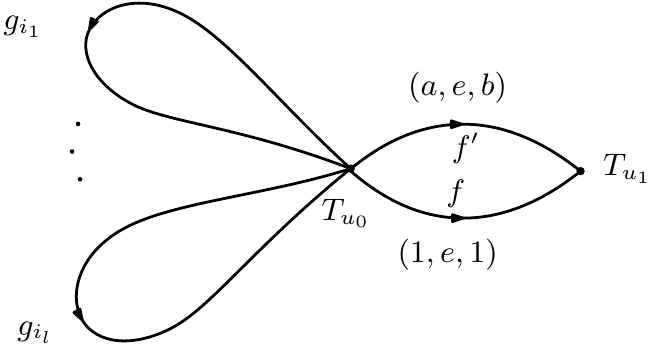}
\end{center}
\caption{\{$g_{i_1},\ldots, g_{i_l}\}=\{g_2,\ldots,g_k\}\setminus(A\cup B)$.}{\label{fig:lemma23}}
\end{figure}
 
The folding proof of Grushko's theorem transforms this wedge of circles into an edge with the elements sitting in the vertex groups. It now suffices to observe that we can choose the folding sequence  in such a way that  the label of $f$ is unchanged and the label of $f'$ is unchanged unless it is modified by an auxiliary move of type A2 just before identifying $f$ and $f'$ with an elementary fold of type  IIIA. After this auxiliary move the label of $f'$ is either $(1,e,b)$ or $(a,e,1)$.  Thus the subsequent fold either adds the element $a$ to $ T_{u_0}$ or the element $b$ to $ T_{u_1}$. The assertion follows. 
\end{proof}

\begin{proof}[Proof of Lemma~\ref{allprimitive}]
The result is  easily verified if $\mathcal O=D^2(m_1)$ as in this case $G\cong \mathbb Z_{m_1}$. Thus we may assume that  $\mathcal O\neq D^2(m_1)$.

\smallskip

\noindent(1)$\implies$(2) Suppose first that $\nu_i =1$ for some $i\in\{1,\ldots , r\}$. After appropriate left and right multiplications of $s_i=s_i^{\nu_i}$ with powers of the other elements of the tuple we  can replace $s_i$ by $c_q$. Thus $T$ is  equivalent to  $(a_1, \ldots, a_p, s_1^{\nu_1}, \ldots, s_{i-1}^{\nu_{i-1}}, s_{i+1}^{\nu_{i+1}}, \ldots, s_r^{\nu_r}, c_1, \ldots, c_q)$  which  proves the assertion.

\smallskip

\noindent(2) $\implies$ (3) is trivial.
 
\smallskip
 
\noindent(3) $\implies$ (1) The proof is by induction on $r$.  The case $r=0$ is trivial as $\mathcal O$  is  in this case   a surface and a simple homology argument shows that for fundamental groups of surfaces with boundary  there is no irreducible generating  tuple  containing elements corresponding to all boundary curves.  

Suppose  that $r\ge 1$ and there exist $g_1,\ldots ,g_q$  in $G$ such that  $T$ is Nielsen equivalent to a tuple $T'$ containing $g_1c_1g_1^{-1},\ldots ,g_qc_qg_q^{-1}$. We have to show that $\nu_i=1$ for some $i\in\{1,\ldots ,r\}$. After a global conjugation (which preserves the Nielsen class of $T$ since it generates $G$) we may assume that $g_q=1$. As we have excluded the case $\mathcal O=D^2(m_1)$, we can write 
$$G=A*B=\langle a_1,\ldots ,a_h,c_1,\ldots ,c_{q-1},s_1, \ldots , s_{r-1}\rangle \ast \langle s_r \rangle.$$ 
Recall that $c_q=as_r$ with $a\in A$ corresponding to a boundary component of the orbifold $\mathcal O'$ with the cone point of order $m_r$ turned into an ordinary point. It now follows from Lemma~\ref{length2} that $T'$ (and therefore $  T$) is Nielsen equivalent to a tuple $ T''$ consisting entirely of  elements from $A\cup B$ such that one of the following holds:
\begin{enumerate}
\item $T''$ contains $s_r$.
\item $T''$ contains $a$.
\end{enumerate}
In the first case the conclusion is immediate as the remaining element of $T''$ must be equivalent to a standard generating tuple of $A$ by Theorem~\ref{grushkolustig}.

In the second case $T''$ contains a generating set of $A$ containing conjugates of all boundary components (as the folding proof of Grushko's theorem preserves the conjugacy class of elliptic elements and therefore of the conjugates of $c_1,\ldots ,c_{q-1}$). The conclusion now follows by induction. 
\end{proof}

Lemma~\ref{addinggeneratormakesreducible} below is a variation of the well-known fact that $rank(U_1)\le rank(U_2)$  if  $U_1$ and $U_2$ are  finite index subgroups of a finitely generated free groups $F$ such that $U_2\le U_1$. This almost remains true in  a free product of cyclic groups. As in the free group case this follows from a simple Euler characteristic argument. We first record the following well-known (trivial) fact.

\begin{lemma}{\label{lemma:euler}} 
Let $G=\mathbb Z_{m_1}*\ldots \ast \mathbb Z_{m_r}$ with $2 \le  m_i \le   \infty$. Then  
$$1-r\le  \chi(G):=1-r+\sum_{i=1}^r\limits 1/m_i\le  1-r/2,$$ 
and therefore $$1-\chi(G)\le  rank(G) \le  2-2\chi(G).$$ 
Furthermore 
$$rank(G)=1-\chi(G) \ \text{ iff } \ m_i=\infty \  \text { for all } \   1\le  i \le  r$$ 
and 
$$rank(G)=2-2\chi(G) \     \text{ iff  } \  \ m_i=2 \ \text{  for all   }  \  1\le  i \le  r.$$
\end{lemma}
\begin{lemma}\label{addinggeneratormakesreducible} 
Let $G=\mathbb Z_{m_1}*\ldots  \ast \mathbb Z_{m_r}$ with $2 \le  m_i \le \infty$ and $ T=(g_1,\ldots ,g_k)$ be a minimal generating tuple of a finite index subgroup $H\le G$. Let $g\in G$. Then one of the following holds:
\begin{enumerate}
\item $(g_1,\ldots ,g_k,g)$ is reducible.

\item $H\cong F_k$, $\langle H,g\rangle\cong \mathbb Z_2*\ldots \ast \mathbb Z_2$ ($k+1$ free factors) and $|\langle H,g\rangle:H|=2$.
\end{enumerate}
\end{lemma}
\begin{proof} 
We can assume that $g $ does not lie in  $H=\langle g_1,\ldots ,  g_k\rangle$ as $(g_1,\ldots ,g_k,g)$ is otherwise Nielsen equivalent to $(g_1,\ldots ,g_k,1)$ and therefore reducible. We can further assume that $G$ is not cyclic as the assertion is trivial in this case. Thus   $\chi(G)\le  0$ and therefore $\chi(H)\le  0$.

Put $U:=\langle H,g\rangle$. Clearly $|G:U|\le  \frac{1}{2}|G:H|$ and therefore $\chi(H)\le  2\chi(U)$. By Kurosh's  subgroup theorem both $H$ and $U$ are free products of cyclics themselves, thus by Lemma~\ref{lemma:euler} we have 
$$rank(U)\le  2-2\chi(U)\le  2-\chi (H)=1+(1-\chi(H))\le  1+ rank(H)$$
 with $rank(U)=1+ rank(H)$ iff  $U\cong F_k$ and $H\cong\mathbb Z_2*\ldots*\mathbb Z_2$ (with $k+1$ free factors) and $|U:H|=2$. This last case puts us into case $(2)$, while the case $ rank(U)\le  rank(H)$ clearly puts us into case $(1)$ by Lemma~\ref{grushkolustig}.
\end{proof}

\begin{corollary}\label{adding2generatorsmakesreducible} 
Let $G$, $T$ and $H$ be as in Lemma~\ref{addinggeneratormakesreducible} and $g,h\in G$.  Then  $(g_1,\ldots,g_k,g,h)$  is reducible.
\end{corollary}

%---------------------------------------------------------------
 
\subsection{Partitioned tuples}\label{section_local_picture}

In this section we define the notion of partitioned tuples and equivalence of partitioned tuples in the fundamental group of compact orbifolds with boundary. 

If $\mathcal O$ is a compact orbifold with boundary then we call $g\in G=\pi_1(\mathcal O)$ \emph{peripheral} if $g$ is conjugate to an element of a subgroup corresponding to a boundary component. If $G$ is given by the standard presentation
$$
\langle a_1,\ldots ,a_h, c_1,\ldots ,c_q , s_1,\ldots , s_r  \ | \   s_1^{m_1},\ldots ,s_r^{m_r},  R\cdot s_1\cdot \ldots \cdot s_r =c_1\cdot \ldots \cdot c_q \rangle
$$ 
then $g$ is peripheral if and only if $g$ is conjugate into $\langle c_i\rangle$ for some $1\le i\le q$. A \emph{label} of a peripheral element $g$ is a pair $(o, i) \in G\times \{1, \ldots, q\}$   
 such that $g \in o \langle  c_i \rangle  o^{-1}.$  A \emph{labeled peripheral element} is a peripheral element endowed with  a label.

\begin{remark}{\label{rem:uniquelabel}}
If $\mathcal O$ is   hyperbolic then  labels are unique in the following sense: if $(o, i)$ and $(o', i')$ are labels of the same peripheral element $g\in G$, then $i=i'$ and $o'=oc_i^k$ for some integer $k$. This follows from the  following facts:
\begin{enumerate}
\item[(i)]   the subgroups  $\langle c_1\rangle, \ldots,  \langle c_q\rangle$ are malnormal, i.e.~$g\in G$ lies in the peripheral subgroup $\langle c_i\rangle$  if and only if  $g\langle c_i\rangle g^{-1}\cap \langle c_i\rangle \neq 1.$ 

\item[(ii)] the subgroups  $\langle c_1\rangle, \ldots,  \langle c_q\rangle$  are  conjugacy separable, i.e.~when $i\neq j$ then $g\langle c_i\rangle g^{-1} \cap \langle c_j\rangle =1 $ for any $g\in G$.	   
\end{enumerate}
\end{remark}

Let $G$  be a (arbitrary) group. A \emph{partitioned tuple} in $G$ is a pair $\mathcal P=(T, P)$  
consisting of two tuples 
$T=(g_1, \ldots, g_m)$ and $ P=(\gamma_1, \ldots, \gamma_n)$  
of elements of $G$.  We call $P$  the \emph{peripheral tuple}  and $T\oplus P$ the  \emph{underlying tuple}  of $\mathcal P$. 

If $G$ is the fundamental group  of a compact orbifold with boundary, then we assume that   $P$ consists only of labeled peripheral elements.

\smallskip

\emph{An elementary transformation} applied to $\mathcal P=(T, P)$ is a move of one of the following types:
\begin{enumerate}
\item[(1)]  replace $T=(g_1, \ldots, g_m)$ by a Nielsen equivalent tuple  $T'=(g_1', \ldots, g_m').$

\item[(2)]  replace $P=(\gamma_1, \ldots, \gamma_n)$  by  
 $P'=(\gamma_{\sigma(1)}^{\varepsilon_1}, \ldots, \gamma_{\sigma(n)}^{\varepsilon_n})$
for some $\sigma\in S_n$ and $\varepsilon_1, \ldots, \varepsilon_n\in \{\pm 1\}$.

\item[(3)] for some $1\le k\le n$ replace   $\gamma_k$    by $h^{\varepsilon}\gamma_kh^{-\varepsilon}$ where $\varepsilon\in\{-1,1\}$ and $$h\in\{g_1,\ldots ,g_m, \gamma_1,\ldots,\gamma_{k-1},\gamma_{k+1},\ldots  ,\gamma_n\}.$$ 
If $G$ is the fundamental group of a compact orbifold we assume that the label of $\gamma_k$  changes in the obvious way, i.e.~if $\gamma_k$ has label $(o_k, i_k)$, then $h\gamma_kh^{-1}$ has label $(ho_k, i_k)$.

\item[(4)] for some $1\le l\le m$ replace  $g_l$   by $h_1^{\varepsilon_1}g_lh_2^{\varepsilon_2}$ where $\varepsilon_1,\varepsilon_2\in\{-1,0,1\}$ and   $h_1,h_2\in\{g_1,\ldots ,g_{l-1},g_{l+1},\ldots g_m, \gamma_1,\ldots ,\gamma_n\}.$   
\end{enumerate}
We say that two partitioned tuples $\mathcal P$ and $\mathcal P'$ are \emph{equivalent}, and write $\mathcal P\sim \mathcal P'$, if there are partitioned tuples  
 $\mathcal P=\mathcal P_0, \mathcal P_1, \ldots,\mathcal P_{k-1},  \mathcal P_k=\mathcal P'$ 
such that $\mathcal P_i$ is obtained from $\mathcal P_{i-1}$ by an elementary equivalence  for $1\le  i\le  k$.   The   equivalence class of $\mathcal P$ is denoted by $[\mathcal P]$.

\begin{lemma}{\label{lemma:equivalencesimpletype}}
Let $F_n=F(x_1,\ldots ,x_n)$ and $G=H\ast F_n$ for some group $H$. Let $0\le r\le n$.  Suppose that 
$$(T_1\oplus (g_{r+1}x_{r+1}g_{r+1}^{-1} ,\ldots  ,g_nx_ng_n^{-1}) ,   (g_1x_1g_1^{-1},\ldots ,g_rx_rg_r^{-1}))$$ 
and 
$$(T_2\oplus (h_{r+1}x_{r+1}h_{r+1}^{-1},\ldots ,h_{n}x_{n}h_{n}^{-1}) ,   (h_1x_1h_1^{-1},\ldots , h_rx_r h_r^{-1}))$$ are equivalent and minimal partitioned generating tuples.
Then  
$$\mathcal P_1 :=(T_1, (g_1x_1g_1^{-1}, \ldots ,  g_{n}x_{n}g_{n}^{-1})) \ \ \text{ and } \  \ \mathcal P_2 :=(T_2, (h_1x_1 h_1^{-1},  \ldots ,  h_{n}x_{n}h_{n}^{-1}))$$ 
are  equivalent and are (both) equivalent to $(T,(x_1,\ldots, x_r, \ldots ,x_n))$ for some $T$.
\end{lemma}
\begin{proof} Consider the free product decomposition $ G=H\ast \langle x_1\rangle \ast \ldots \ast  \langle x_{n}\rangle  .$ 
The folding proof of Grushko's theorem (which preserves the element $g_ix_ig_i^{-1}$ and $h_ix_ih_i^{-1}$ up to conjugacy as these elements are elliptic in the above splitting of $G$) shows that $\mathcal P_1$ is  equivalent to 
$(T_3, (x_1,\ldots, x_{n}))$   
and that $\mathcal P_2$ is equivalent to 
 $(T_4,(x_1,\ldots, x_{n}))$  
where $T_3$ and $T_4$ are generating tuples of $H$. As the underlying tuples are moreover Nielsen equivalent and irreducible it follows from the main result of \cite{W2} that $T_3$ and $T_4$ are Nielsen equivalent which implies that  the partitioned tuples  $(T_3, (x_1,\ldots, x_{n}))$ and $(T_4, (x_1,\ldots ,  x_{n}))$ are equivalent. Thus $\mathcal P_1$ and $\mathcal P_2$ are   equivalent and equivalent to $(T, (x_1, \ldots, x_n))$.  
\end{proof}
 
%---------------------------------------------------------------
 
From now  on we only consider partitioned tuples of element of the fundamental group of a compact orbifold   with boundary. Most of the following definitions are motivated by Louder's notion of coverlike morphisms between square complexes, see Definition~7.1 from \cite{Lou}. The only  exception is the notion of partitioned tuples of almost   covering type as they cannot occur in the surface case.

\smallskip

\noindent{\textbf{(1)} {\em Partitioned tuples of simple type.}  We say that a partitioned tuple $\mathcal P$ \emph{is of simple type} if  $\mathcal P$ is equivalent to $ (T , (\gamma_1,\ldots , \gamma_n ))$ such that   $U=\langle T \rangle \ast \langle \gamma_1\rangle \ast \ldots \ast \langle \gamma_n\rangle$ and $  rank(U)=size(T)+n$    where $U\le G$ is generated by the underlying tuple of  $\mathcal P$.

\smallskip

Partitioned tuples of orbifold covering type where defined in \cite{Dut} using almost orbifold coverings  induced by orbifold covers. Here we give a slightly different (but equivalent) definition.  
 
\smallskip

\noindent{\em\textbf{(2)} Partitioned tuples of orbifold covering type.}  A marking  of \emph{orbifold covering type} is a pair  $(\eta':\mathcal O'\rightarrow \mathcal O, [ \mathcal P'])$  where the following hold:
\begin{enumerate}
\item[(a)] $\eta'$ is an orbifold covering  of finite degree.

\item[(b)] $[\mathcal P']$ is the equivalence class of a partitioned tuple  $\mathcal P'=(T', P')$  in $G':=\pi_1^o(\mathcal O')$  such that  the following hold:
\begin{enumerate}
\item[(i)]  there is a one-to-one correspondence between the elements of $P'$ and the boundary components of $\mathcal O'$.
 
\item[(ii)] if $\mathcal O'$  is a surface (that is, $\mathcal O'$  has no cone points) then  $(T'\oplus P')\setminus\{\gamma'\}$  is a  minimal generating tuple of $G'$ for any $\gamma'\in P'$.

\item[(ii')] if $\mathcal O'$ is not a surface (that is, $\mathcal O'$ has at least one cone point), then $T'\oplus P'$   is a  minimal generating tuple of $G'$.
\end{enumerate}
\end{enumerate}
We say that a partitioned tuple $\mathcal P$ is of \emph{orbifold covering type} if there is a marking  $$(\eta':\mathcal O'\rightarrow \mathcal O , [\mathcal P'])$$  of orbifold covering type such that $\mathcal P$ is equivalent to $ \eta_{\ast}'(\mathcal P'):=(\eta_{\ast}'(T'), \eta_{\ast}'(P')).$  When $\mathcal O'$  has no cone points we will also say that $\mathcal P$  is of \emph{surface  covering type}.  
 
\smallskip

\noindent{\textbf{(3)} {\em Partitioned tuples of almost orbifold  covering type.}  A marking  of \emph{almost orbifold covering type}    is a pair  $(\eta':\mathcal O'\rightarrow \mathcal O , [ \mathcal P'])$  
where the following hold: 
\begin{enumerate}
\item[(a)] $\eta'$ is a special almost orbifold covering.

\item [(b)] $[\mathcal P']$ is the equivalence class of a partitioned tuple $\mathcal P' = (T', P')$ in  $G':=\pi_1^o(\mathcal O')$ such that the following hold:
\begin{enumerate} \item[(i)] there is a one-to-one correspondence between the elements of $P'$ and the  non-exceptional boundary components of $\mathcal O'$.
 \item[(ii)] $T'\oplus P'$  is a rigid (and therefore minimal) generating tuple of $G'$. 
\end{enumerate}
\end{enumerate}
We say that a partitioned tuple $\mathcal P$ is of \emph{almost  orbifold covering type} if there is a marking   
$ (\eta':\mathcal O'\rightarrow \mathcal O ,  [\mathcal P'])$    
of almost orbifold covering type  such that $\mathcal P$ is equivalent to $\eta_{\ast}'(\mathcal P')$.

\smallskip

\noindent{\textbf{(R)} {\em Reducible partitioned tuples.}  We say that a partitioned tuple   $\mathcal P$ is \emph{reducible} if $\mathcal P$ is equivalent to $ (T'\oplus (1), P')$.

\smallskip

\noindent{\textbf{(FPE)}  {\em Partitioned tuples that fold peripheral elements.}  We say that a partitioned tuple  $\mathcal P$ \emph{folds peripheral elements} if $\mathcal P$ is equivalent to $(T,(\gamma_1, \gamma_2)\oplus P)$ such that the labels of $\gamma_1$ and $\gamma_2$ satisfy $i:=i'=i''$ and $o''=o'c_i^z$ for some $z\in \mathbb Z$.

\begin{remark}{\label{remark:foldperipheral}}
It follows from  the equations $i:=i' =i''$ and   $o''=o' c_i^z$  that $\langle\gamma', \gamma''\rangle\leq G$ is cyclic.  The converse does not hold. For instance, assume that $\mathcal O=D^2(2, 2)$. Hence 
$$G= \langle s_1, s_2, c_1 \  | \ s_1^2,  s_2^2 , c_1^{-1}s_1s_2\rangle.$$ 
Consider the partitioned tuple  $\mathcal P=(\emptyset, (\gamma_1, \gamma_2))$  where  $\gamma_1=c_1=s_1s_2$ is  labeled  $(1, 1)$ and $\gamma_2=s_1c_1s_1^{-1}$ is labeled  $(s_1, 1)$. Then $\langle \gamma_1, \gamma_2\rangle$ is cyclic  but $\mathcal P$ does not fold peripheral elements. One can easily check that $\mathcal P$ is of surface covering type with corresponding orbifold an annulus. 
\end{remark}

\noindent{\textbf{(OR)} {\em Partitioned tuples with an obvious relation.} We say that a partitioned tuple $\mathcal P$ \emph{has an obvious relation} if  $\mathcal P$ is equivalent to $(T, (\gamma)\oplus P)$ such that 
 $$|o\langle c_i\rangle o^{-1}: U \cap o \langle c_i\rangle o^{-1}|< | o \langle c_i\rangle o^{-1}:  \langle \gamma \rangle|$$ 
where $U:=\langle T\oplus P \rangle\leq G$ and   $(o, i)$ is the label of $\gamma$. 
 
\smallskip

Observe that a partitioned tuple   can be simultaneously of type (R), (FPE) and (OR). For example, let $\mathcal P=( (\gamma_1) , (\gamma_1, \gamma_1^2))$   where $G$ and $\gamma_1$ are as in  remark~\ref{remark:foldperipheral}. {\it 

\smallskip

\noindent \textbf{(4)} We say that a partitioned tuple is of type \textbf{(4)} if it is  of one of the types {(R)}, {(FPE)} or  {(OR)}. 
}
 
\smallskip

Small orbifolds were defined in \cite{Dut}. In the present work we use a slightly more restrictive definition. In particular all results from \cite{Dut} hold in the current setting.

\begin{definition}
We say that an orbifold is \emph{small} if it is isomorphic to one of the following types of  orbifolds:
\begin{enumerate}
\item[(1)] A Möbius band  with no cone points. 
 
\item[(2)] A disk with two cone points, i.e.~$D^2(m_1, m_2)$.
 
\item[(3)] A annulus with one cone point, i.e.~$A(m_1)$.

\item[(4)] A pair of pants  without cone points.
\end{enumerate}
\begin{figure}[h!]
\begin{center}
\includegraphics[scale=0.9]{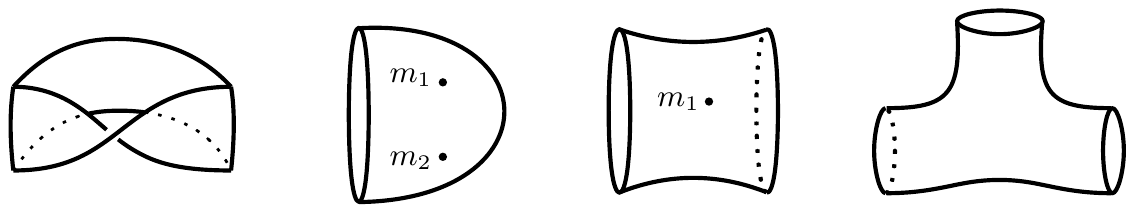}
\end{center}
\caption{Small orbifolds.}
\end{figure}
\end{definition}
 
The main purpose of Chapter~\ref{sectionlocalcase}  is to establish Proposition~\ref{prop:01} below which is a  strengthening of Proposition~2.9 of \cite{Dut}.

\begin{proposition}{\label{prop:01}}
Let $\mathcal O$ be a small orbifold,  and let $\mathcal P$ be an arbitrary partitioned tuple in  $G=\pi_1^o(\mathcal O)$. Then $\mathcal P$ is of one of the type (\textbf{1})-(\textbf{4}).

Moreover  \textbf{(1)},  \textbf{(2)},  \textbf{(3)}  and  \textbf{(4)}    are mutually exclusive and if case   \textbf{(2)}  or  \textbf{(3)}  occurs, then the marking $(\eta':\mathcal O'\rightarrow \mathcal O, [\mathcal P'])$ representing $[\mathcal P]$ is unique. 
\end{proposition} 

That any partitioned tuples falls  into one of the cases  \textbf{(1)} - \textbf{(4)}  is precisely Proposition~2.9 of \cite{Dut}, thus we need to prove the second part. 

\begin{remark} 
The uniqueness of $\eta'$ in Proposition~\ref{prop:01} is trivial if    \textbf{(2)}  occurs  as $\eta'$ is just the covering map corresponding to the subgroup generated by the underlying tuple of $\mathcal P$ and $\mathcal P'$ is any lifting of $\mathcal P$. 
\end{remark}

%--------------------------------------------------------------

\subsection{Decorated $\mathbb A$-graphs} 

In the following we represent partitioned tuples in the fundamental group of an orientable small orbifold   by  so-called decorated $\mathbb A$-graphs.   Here $\mathbb A$ is the splitting of the fundamental group of the (orientable small) orbifold $\mathcal O$ given by:
\begin{enumerate}
\item The underlying graph $A$ of $\mathbb A$ has vertex set $VA=\{v_1, v_2\}$   and edge  set $$EA=\{e_i^{\pm1} \ | \ 1\le  i\le q\}.$$  The edges $e_1, \ldots, e_q$ have initial vertex $ v_1$ and terminal vertex  $  v_2$.

\item All edge groups in  $\mathbb A$ are trivial. The vertex  groups are defined as $$A_{v_i}=\langle s_{v_i} \ | \ s_{v_i}^{m_i}\rangle\cong \mathbb{Z}_{m_i} \ \text{ for } \ i=1,2.$$
\end{enumerate}   
The base vertex of $\mathbb A$  is set to be $v_1$.    Observe that  the valence of $v_1$ and $v_2$  equals  the number of boundary components of $\mathcal O$, that is, $q=val(v_1, A)=val(v_2, A)$.

Some definitions  from Section~3 of  \cite{Dut} will also be given here, but even so we assume that the reader is familiar with the language developed there.

For each $1\le   i\le   q$,  let $c_i:=s_{v_1}^{\varepsilon_i},e_i,s_{v_2}^{\varepsilon_i}, e_{i+1}^{-1},1$   where $\varepsilon_1=1$ and $\varepsilon_i= 0$ for $i\ge 2$.  The element $[c_i]\in \pi_1(\mathbb A, v_1)\cong G$ represented by $c_i$  will also be denoted by $c_i$. Note that   $c_1, \ldots , c_n$ correspond to the boundary components of $\mathcal O$.   
  
We need to extend the notion of  peripheral elements to the fundamental group of $\mathbb A$-graphs. Let $\mathcal B$  be an $\mathbb A$-graph with associated graph of groups $\mathbb B$.  A closed path $p$  in   $\mathbb B$   is called a  \emph{peripheral path of $\mathcal B$} if 
 $\mu_\mathcal B(p)=a \cdot c_i^z \cdot a^{-1}$ for some $i\in \{1,\ldots, q\}$, some $a\in  A_{v_1}$, and some positive integer $z$.  Note that  $a$ and $ i$ are uniquely determined by $p$.  Assume  that $u_1$ is a vertex of $\mathcal B$ of type $v_1$.  For any  $\mathbb B$-path $q$ from $u_1$  to  $\alpha(p)$,  we call     
$$h_{p,q}:=[q\cdot p\cdot q^{-1}]\in \pi_1(\mathbb B, u_1)$$ 
a \emph{peripheral element of $\pi_1(\mathbb B, u_1)$ associated to $p$}, or simply \emph{a peripheral element}. The corresponding peripheral element 
 $$\gamma:=\phi_{\mathcal{B}}(h_{p,q})=[\mu_{\mathcal B}(q\cdot p\cdot q^{-1})]$$  
of $\pi_1(\mathbb A, v_1)$ has a natural label defined  in terms of $p$ and $q$, namely   
 $$(o_{\gamma}, i_{\gamma}):= ([\mu_\mathcal B(q)\cdot a], i)\in \pi_1(\mathbb A, v_1)\times \{1, \ldots, q\}.$$

\begin{definition}
A \emph{decorated $\mathbb A$-graph} is a tuple
 $\mathfrak B=(\mathcal B, u_1, (T_u)_{u\in VB}, (p_j)_{1\le  j\le  n})$ 
where the following hold: 
\begin{enumerate}
\item $\mathcal B$ is a finite $\mathbb A$-graph. 

\item $u_1$ is a vertex of $\mathcal B$ such that $[u_1]=v_1$, the base vertex of $\mathbb A$.

\item  for each $u\in VB$,  $T_u$ is a generating tuple of   $B_u\le A_{[u]}$.

\item $(p_j)_{1\le  j\le  n}$ is a tuple of   peripheral paths of $\mathcal B$.
\end{enumerate}  
The  $\mathbb A$-graph $\mathcal B$ (resp.~the underlying graph  $B$ of  $\mathcal B$) is called the underlying  $\mathbb A$-graph (resp.~underlying graph) of $\mathfrak B$.  
\end{definition}

\begin{remark}
Observe that a decorated $\mathbb A$-graph  in which the tuple of peripheral paths is empty can also  be seen as a marked $\mathbb A$-graph as defined in  \cite{W2}.  
\end{remark}

\noindent\textbf{Induced decoration.} We will need the notion of  decorated sub-$\mathbb A$-graphs. Let $$\mathfrak B=(\mathcal B, u_1, (T_u)_{u\in VB}, (p_j)_{1\le j\le  n})$$   be a decorated $\mathbb A$-graph.  For any  sub-$\mathbb A$-graph $\mathcal B'\subseteq \mathcal B$  and any vertex $u_1'\in VB'$  such that $[u_1']=v_1$ we define 
$$\mathfrak B'=(\mathcal B', u_1' ,  (T_{u'})_{u'\in VB} , (p_j')_{1\le j\le  n'})$$    where $T_u'=T_u$ for all $u\in VB'$ and where $(p_j')_{1\le  j\le  n'}$ consists of those peripheral paths in  $(p_j)_{1\le  j\le  n }$  that lie in the graph of groups $\mathbb B'$ associated to $\mathcal B'$. We say that  $\mathfrak B'$ is a decorated sub-$\mathbb A$-graph of $\mathfrak B$ carried by $\mathcal B'$.

\begin{convention}
Whenever $u_1$ belongs to $\mathcal B'\subseteq \mathcal B$, for instance, if $\mathcal B'=core(\mathcal B, u_1)$,   we will assume $u_1'=u_1$ so that $\mathfrak B'$ is completely determined by  $\mathcal B'$.  

Moreover, any decorated sub-$\mathbb A$-graph of $\mathfrak B$  carried by $core(\mathcal B)$ will be denoted by $core(\mathfrak B)$ and \emph{the} decorated sub-$\mathbb A$-graph carried by   $core(\mathcal B, u_1)$ will be denoted by $core(\mathfrak B, u_1)$.     
\end{convention}

\begin{remark}{\label{remark:subdecorated}}
Observe that the peripheral paths $p_1, \ldots, p_n$ in the defining data of $\mathfrak B$  are reducible  closed  $\mathbb B$-paths and   are therefore contained in $core(\mathbb B)$ and in $core(\mathbb B, u_1')$ for any $u_1'\in VB$.  Thus   $(p_j)_{1\le j\leq n} $ is the tuple of peripheral paths in $core(\mathfrak B)$, regardless of the  base  vertex,  and in $core(\mathfrak B , u_1)$.
\end{remark}

%-----------------------------------------------------------------------

\noindent{\textbf{Associated partitioned tuple.}} Let   $\mathfrak B =(\mathcal B , u_1 , (T_u )_{u\in VB}, (p_j)_{1\le j\le  n})$     be a  decorated $\mathbb A$-graph.  We say that a set $\{f_1, \ldots, f_n\}\subseteq EB $  containing $n$ distinct edges of $B$  is a  \emph{set of collapsing edges}  of $\mathfrak B$  if   (1) $f_i\neq f_j^{-1}$ for all $1\le  i\neq j\le  n$ and (2) there is   $\sigma \in S_n$ such that  for any  $k\in \{1, \ldots, n\}$  the path $p_{\sigma(k)}$ decomposes as   
$$p_{\sigma(k)}=p_{\sigma(k)}' \cdot (1,f_k,1)\cdot  p_{\sigma(k)}''$$ 
with $p_{\sigma(k)}'$ and $p_{\sigma(k)}''$ contained in  $\mathbb B_{f_1, \ldots, f_k}$. We say that $\mathfrak B$ is \emph{collapsible} if $\mathfrak B$ admits  a set of collapsing edges.

\smallskip

Any  collapsible decorated  $\mathbb A$-graph   $\mathfrak B=(\mathcal B, u_1, (T_u)_{u\in VB}, (p_j)_{1\le  j\le n})$ 
defines an equivalence class of partitioned tuples in $\pi_1(\mathbb B, u_1)$ as follows.   Assume
\begin{enumerate}
\item[(i)] $\{f_1, \ldots, f_n \}\in EB$ is  a set of collapsing edges of $\mathfrak B$, see Definition~3.9 of \cite{Dut}, 

\item[(ii)]  $Y$ is a maximal subtree of $B':=B_{f_1, \ldots, f_n}$, 

\item[(iii)]  $E$ is an orientation of  $EB'-EY$. 
\end{enumerate}
For any $u\in VB$ let $[u_1, u]_Y:=e_{u,1},\ldots, e_{u,l_u}$  be the unique reduced path in $Y$ from $u_1$ to $u$ and let   $q_u:=1,e_{u,1}, 1 ,\ldots, e_{u,l_u}, 1$     
be the corresponding $\mathbb B$-path. We define  $(T_\mathfrak B , P_\mathfrak B)$ as  follows:
\begin{enumerate}
\item   $T_\mathfrak B$ consists of the following elements:  
 $g_e:=[ q_{\alpha(e)} \cdot (1,e,1) \cdot  q_{\omega(e)}^{-1}]  $  for each   $e\in E $ and 
 $g_{u, b}:=[q_u\cdot  b\cdot q_u^{-1}] $  for each $ b\in T_u. $

\item   $P_\mathfrak B$ consists of the  peripheral elements $\gamma_j =[q_{\alpha(p_j)} \cdot p_j \cdot q_{\alpha(p_j)}^{-1}]$  for each $j=1 , \ldots,  n.$  
\end{enumerate}

Note that $(T_\mathfrak B , P_\mathfrak B)$ depends on various choices. However,   Lemma~3.12 from \cite{Dut}   shows that  the equivalence class  of  $(T_{\mathfrak B}, P_{\mathfrak B})$ does not depend on the choice of the collapsing edges $f_1,\ldots ,f_n$ and the choice of the maximal subtree $Y\subseteq B'$. It therefore makes sense to  define $[\mathcal P_\mathfrak B]$ as the equivalence class determined by 
$\phi_\mathcal B(T_\mathfrak B, P_\mathfrak B)=(\phi_\mathcal B(T_\mathfrak B), \phi_\mathcal B(P_\mathfrak B)) $  
where $(T_\mathfrak B, P_\mathfrak B)$  corresponds to an arbitrary set of collapsing edges and arbitrary maximal subtree. We say that the collapsible  decorated $\mathbb A$-graph $ \mathfrak B$ \emph{represents} the class $[\mathcal P_\mathfrak B]$.
 
\smallskip

%----------------------------------------------------------
 
\noindent\textbf{Equivalence of decorated $\mathbb A$-graphs.} We define equivalence of decorated $\mathbb A$-graphs  in terms of auxiliary moves and Nielsen equivalences of the vertex tuples. The effect of auxiliary moves  and elementary folds on a decorated $\mathbb A$-graph is described in  Section 3.2 of \cite{Dut} which in turn is strongly based on the definitions given in \cite{W2}. 

We say that two decorated $\mathbb A$-graphs  $\mathfrak B$ and $\mathfrak B'$ are \emph{equivalent} if there exists a finite sequence 
$\mathfrak B= \mathfrak B_0,  \mathfrak B_1,  \ldots, \mathfrak B_k= \mathfrak B'$ of decorated $\mathbb A$-graphs such that for each $1\le i \le k$ one of the following holds:
\begin{enumerate}
\item[(1)]  $\mathfrak B_i$ is obtained from $\mathfrak B_{i-1}$ by an auxiliary move (in the case of an A0 move we assume it is admissible with respect to the base vertex $u_1^{i-1}$ of $\mathfrak B_{i-1}$).

\item[(2)]   $\mathfrak B_i$ is obtained from $ \mathfrak B_{i-1}$ by replacing  the generating tuple of some vertex tuple by a Nielsen equivalent  generating  tuple.

\item[(3)]   $\mathfrak B_i$ is obtained from $ \mathfrak B_{i-1}$ by removing  a trivial element from a vertex tuple and adding it to another vertex tuple.
\end{enumerate}
We denote the equivalence class of a decorated $\mathbb A$-graph  $\mathfrak B$ by  $\mathfrak b=[\mathfrak B]$.

\smallskip
  
In \cite[Lemma~3.14]{Dut} it is shown  that  two collapsible decorated $\mathbb A$-graphs that are related by auxiliary moves determine   equivalent partitioned tuples in $\pi_1(\mathbb A, v_1)$. For elementary equivalences that change the vertex tuples this is easily verified.  Therefore,  for any equivalence class $\mathfrak b=[\mathfrak B]$ of collapsible decorated $\mathbb A$-graphs we define $[\mathcal P_{\mathfrak b}]:=[\mathcal P_\mathfrak B]$ 
for some(and therefore any) representative $\mathfrak B$ of $\mathfrak b=[\mathfrak B]$.

\begin{lemma}{\label{lemma:foldsameedges}}
Let $\mathfrak B_1$ and $\mathfrak B_2$ be equivalent decorated $\mathbb A$-graphs with underlying graph $B$. Suppose that both admit an  elementary fold that identifies the  edges $f_1$ and $f_2$  based on $x:=\alpha(f_1)=\alpha(f_2)$.  Let $\mathfrak B_1'$ and $\mathfrak B_2'$ be the resulting decorated $\mathbb A$-graphs.  Then $\mathfrak B_1'$ and $\mathfrak B_2'$ are equivalent.
\end{lemma} 
\begin{proof}
This is a consequence of the proof of Lemma~6 in \cite{W2} and the way the decoration is modified when auxiliary moves are applied to decorated $\mathbb A$-graphs.
\end{proof}

%-----------------------------------------------------------------
 
\noindent{\textbf{Special types of decorated $\mathbb A$-graphs.}} We will be mainly interested in tame decorated $\mathbb A$-graphs and decorated $\mathbb A$-graphs of (almost) orbifold covering type.  We observe that the definition of decorated $\mathbb A$-graphs  of (almost) orbifold covering type  we are going to give here is slightly different from that given  in \cite[Section~3.7]{Dut} in the sense that we do not require the underlying $\mathbb A$-graph to be minimal.  Before giving the definitions we  recall some notions introduced in Subsection~3.6 from \cite{Dut}.

\smallskip

Let $\mathfrak B=(\mathcal B , u_1, (T_u)_{u\in VB} , (p_j)_{1\le j\le n})$ be a decorated $\mathbb A$-graph.   We say that $\mathfrak B$   \emph{self-folds} if  there are $f\in EB$ and  $k\in \{1, \ldots, n\}$ such that $p_k=p_k'\cdot  (1,f,1)\cdot   p_k'' \cdot (1,f,1)\cdot  p_k'''.$ 
  
We say that $\mathfrak B$ \emph{folds peripheral paths}  if   if there are $f\in EB$ and   $j\neq k \in \{1,\ldots, n\}$ such that    $i_j=i_k\in \{1,\ldots, q\}$  and that  $p_{j}=p_{j}' \cdot (1,f,1) \cdot  p_j'' $ and  $p_k=p_k' \cdot (1,f,1)\cdot  p_{k}''.$  
  
We  say that $\mathfrak B$ \emph{folds squares} if  either  $\mathfrak B$ folds peripheral paths or $\mathfrak B$ self-folds. 

\smallskip

A \emph{directed graph} $\Gamma$ is a pair $(V\Gamma, E\Gamma)$ consisting of a vertex set $V\Gamma$ and an edge set  $E\Gamma \subseteq V\Gamma\times V\Gamma$. An edge   $(v,w)$   is usually denoted by   $v\mapsto w$. A \emph{circuit}  is  a directed graph isomorphic to  
$$\mathcal{C}_n=(\{1,\ldots, n\}, \{(1,2), \ldots, (n-1, n), (n, 1)\})$$    
for some $n\ge  1$. An \emph{interval} is a directed graph isomorphic  to 
$$\mathcal{I}_n=(\{1,\ldots, n\}, \{(1,2), \ldots, (n-1, n) \})$$  
for some $n\ge1$. Note that  $\mathcal{I}_1$  is a  degenerate interval as it consists of a single vertex. 

The length of  $\Gamma=(V\Gamma, E\Gamma)$ is defined as  $length( \Gamma) := |E\Gamma|$. 

\smallskip

Assume that $\mathfrak B$ does not fold squares. The definition of peripheral paths says that, for each $1\le j\le n$,  there are $a_j\in A_v$ and  $i_j\in \{1,\ldots, q\}$  such that  $\mu_\mathcal B(p_j)=a_j \cdot   c_{i_j}^{z_j}\cdot    a_j^{-1}$ for some positive integer $z_j$.  
 
Let $u\in VB$. The  local graph $\Gamma_{ \mathfrak B}( u)$   of $\mathfrak B$  at $u$ is defined as the directed graph   having vertex set  $V\Gamma_{\mathfrak{B}}( u)= Star (u,B):=\{f\in EB \ | \ \alpha(f)=u\}$   and  edge set   consisting  of all pairs of vertices  $(f, g)$ for which  there is $j\in \{1, \ldots, n\}$ such that,  up to a cyclic permutation, $p_j= p_j' \cdot   (1,f^{-1}, b, g, 1)\cdot   p_j''.$   The label of  the edge  $f\mapsto g$ is   defined as 
\begin{equation} {\label{eq:label}}
label(f\mapsto g):= (j, b)\in \{1, \ldots, n\}\times B_u\tag{*}.
\end{equation}

\begin{remark} 
The restriction to decorated $\mathbb A$-graphs that do not fold  squares guarantees that (\ref{eq:label}) is well-defined.  This follows from the fact  that, up to a cyclic permutation,  the edges $f^{-1}$ and $g$ occur  in the underlying path of a unique peripheral path  of $\mathfrak B$.  This also implies that   the components of $\Gamma_{\mathfrak B}(u)$ are   (possibly degenerate) intervals and  circuits. 
\end{remark}

\begin{remark}{\label{rem:local}}
 Let $\mathfrak B$ be a decorated $\mathbb A$-graph that does not fold squares and let $w\in VB$. Put $v:=[w]\in VA$.   Assume that  $\Gamma_{\mathfrak B}(w)$ is a circuit. 
It is shown in \cite[Section~3.6]{Dut} that  the following hold:
\begin{enumerate}
\item $val( v , A)$ divides $length (\Gamma_{\mathfrak B}(w))$. 

\item  $length (\Gamma_{\mathfrak B}( w))\ge val(v, A)\cdot |A_{[w]}:B_w|$.

\item $\mathcal B$ is locally surjective at $w$.

\item $\mathcal B$ is folded  at   $w\in VB$ iff $
length ( \Gamma_{\mathfrak B}(w)) = val(v, A)\cdot |A_v :B_w|.$
\end{enumerate}
\end{remark}

\begin{definition}{\label{def:types}}
Let $\mathfrak B=(\mathcal{B}, u_1, (T_u)_{u\in VB}, (p_j)_{1\le j\le n}) $  be a decorated $\mathbb A$-graph.   
 
\noindent(1) We say that $\mathfrak B$ is \emph{tame} if $\mathfrak B$ is collapsible, does not fold squares and  all vertex tuples are minimal.
 
\noindent(2) We say that  $\mathfrak B$  is an \emph{orbifold cover} if  the following hold:
\begin{enumerate}
\item[(a)]  if $u\notin core( \mathcal B)$ then $T_u=\emptyset$.  

\item[(b)] the decorated sub-$\mathbb A$-graph $core(\mathfrak B)$ is a decorated $\mathbb A$-graph  of  orbifold covering type as defined in  \cite[Definition~3.34]{Dut}, that is,   $core(\mathfrak B)$ satisfies the following conditions:   
\begin{enumerate}
\item[\textbf{(C1)}] $core(\mathfrak B)$  is folded and does not fold squares.

\item[\textbf{(C2)}]  for each   $w\in VB_{core}$ the local graph $\Gamma_{core(\mathfrak B )}( w)$ is a circuit such that  
$$length ( \Gamma_{core(\mathfrak B )} (w)) \le  val(v, A)\cdot |A_{v}|.$$ 
where $v=[w]\in VA$.

\item[\textbf{(C3)}] for each $w\in VB_{core}$ the vertex tuple $T_w$  has size $\le  1$. 

\item[\textbf{(C4)}]  there is a vertex $y \in VB_{core}$  such that one of the following holds:
\begin{enumerate}
\item[(i)] if $B_y=1$, then $T_y=(1)$ and  $T_x=\emptyset$ for all $x\neq y$.  In this case we also say that $\mathfrak B$ is of surface covering type.

\item[(ii)] if $B_y\neq 1$, then $T_y$ is angle-minimal  and $T_x$ is minimal for all $x\neq y$.
\end{enumerate}
\end{enumerate} 
\end{enumerate}

\noindent(3)  We say that  $\mathfrak B$  is an \emph{almost orbifold cover} if the following hold: 
\begin{enumerate}
\item[(a)]  if $u\notin core( \mathcal B)$ then $T_u=\emptyset$. 

\item[(b)] the decorated sub-$\mathbb A$-graph   $core(\mathfrak B)\subset \mathfrak B$ is a decorated $\mathbb A$-graph  of  almost orbifold covering type as defined in  \cite[Definition~3.35]{Dut}, that is, $core(\mathfrak B)$ satisfies the following conditions:
\begin{enumerate}
\item[\textbf{(C1')}] $core(\mathfrak B)$  does not fold squares.

\item[\textbf{(C2')}] for each $w\in VB_{core}$ the local graph $\Gamma_{core(\mathfrak B)}(w)$ is a circuit such that  
$$length(\Gamma_{core(\mathfrak B)}( w)) \le  val(v, A)\cdot |A_{v}| $$
where $v:=[w]\in VA$.

\item[\textbf{(C3')}] for each $w\in VB_{core}$ the vertex tuple $T_w$  has size $\le  1$. 

\item[\textbf{(C4')}]  there is a vertex $y\in VB_{core}$ such that  the following hold: 
\begin{enumerate}
\item[(i)] if $x\in VB_{core}$ is distinct from $y$, then 
{\begin{enumerate}
\item[(a)]  $core(\mathcal B)$ is folded at $x$.

\item[(b)] $T_x$ is minimal and if $B_x\neq 1$, then $T_x$ is  not  angle-minimal.
\end{enumerate}}
\item[(ii)] at $y$ the following hold: 
{\begin{enumerate}
\item[(a)] the vertex group $B_y$ is non-trivial.

\item[(b)]  $core(\mathcal B)$ is not folded at $y$.

\item[(c)] $T_y=(s_{v}^{d_y})$ where    $d_y=\frac{ length ( \Gamma_{core(\mathfrak B)}( y))}{val(v, A)}$ and $v=[y] \in VA$.    
\end{enumerate}}
\end{enumerate}
\end{enumerate}
\end{enumerate}

\noindent(4)  We say that $\mathfrak B$ is \emph{bad} if $\mathfrak B$  is neither tame nor an (almost) orbifold cover. 
\end{definition}

\begin{remark}
In an (almost) orbifold cover  $\mathfrak B$  for each $x\in VB_{core}$ we have 
$$val(v, A)\cdot |A_v:B_x|\le lenght(\Gamma_{core(\mathfrak B)} (x)) \le val(v, A)\cdot |A_v|$$ 
where $v=[x]\in VA$. 
\end{remark}

\begin{lemma}{\label{remark:prealmost}}
Let $\mathfrak B$ be minimal  tame decorated $\mathbb A$-graph. Assume that a fold $F$  based on $f_1$ and $f_2\in EB$ with $x=\alpha(f_1)=\alpha(f_2)\in VB$  yields a  minimal (almost) orbifold cover $\mathfrak C$.  Then the following hold:
\begin{enumerate}
\item $\omega(f_1)=\omega(f_2)$, that is, $F$ is of type IIIA.

\item if $f\in EB\setminus\{f_1^{\pm 1}, f_2^{\pm 1}\}$  then $f$   occurs in the underlying path  of exactly one  peripheral path  of $\mathfrak B$. In other words, the edge pair $\{f, f^{-1}\}$ is crossed  twice (in opposite directions) by   the  peripheral paths of $\mathfrak B$. 

\item there is $\varepsilon \in \{\pm 1\}$ such that  $f_1^{\varepsilon}$ and $f_2^{-\varepsilon}$ occur in  the underlying path of exactly  one   peripheral path of $\mathfrak B$ and $f_1^{-\varepsilon}$ and $f_2^{\varepsilon}$ do not occur at all.

\item  If $\mathfrak C$ is of surface covering type, then a fold that identifies $f_1^{-1}$ and $f_2^{-1}$ can also be applied to $\mathfrak B$. In any case no other  fold can be applied to $\mathfrak B$.  
\end{enumerate}
\end{lemma}
\begin{proof}
Since $\mathfrak C$ is not collapsible $F$ does not preserve collapsibility. Lemma~3.20 of \cite{Dut} implies  that $F$ is of type IIIA. This shows that (1) holds. 

Put $y:=\omega(f_1)=\omega(f_2)$. Since  $F$ only affects $\Gamma_{\mathfrak B}(x)$ and $\Gamma_{\mathfrak B}(y)$  it follows that $\Gamma_{\mathfrak B}(z)$ is a circuit for all $z$ distinct from $x$ and $y$. At $x$ and $y$ one of the following hold: 
\begin{enumerate}
\item[(a)] both $\Gamma_{\mathfrak B}(x)$ and $\Gamma_{\mathfrak B}(y)$ contain a circuit  and there is $i$ such that $f_i$ and $f_i^{-1}$ are not crossed by the peripheral paths in $\mathfrak B$. 

\item[(b)] both $\Gamma_{\mathfrak B}(x)$ and $\Gamma_{\mathfrak B}(y)$ are intervals with initial vertex $f_1 $ and $f_2 $  and $f_1^{-1}$ and $f_2^{-1}$ respectively.
\end{enumerate}
Observe that (a) cannot occur since this  would imply that  $\mathfrak B$ is not collapsible. Therefore (b)  occurs  which is exactly what (2) and (3) claim.

If $\mathfrak C$ is a surface cover then $F$ adds the trivial element to  $y$. This meas that a fold that identifies $f_1^{-1}$ and $f_2^{-1}$ (and adds the trivial element to $x$) can also be applied to $\mathfrak B$.
\end{proof}

\begin{lemma}{\label{lemma:foldalmost}}
Let $\mathfrak B$ be a  decorated  $\mathbb A$-graph  such that the following hold:
\begin{enumerate}
\item $\mathfrak B$ is not collapsible and does not fold squares. 

\item  there is a vertex $u $ in $\mathcal B$   such that $T_u\neq\emptyset$. 
   
\item $\mathfrak B$ folds onto an (almost) orbifold cover $\mathfrak C$ with a single elementary fold.
   
\end{enumerate} 
Then $\mathfrak B$  is an (almost) orbifold cover.
\end{lemma}
\begin{proof}
Assume that $$\mathfrak B= (\mathcal B, u_1, (T_u)_{u\in VB}, (p_j)_{1\le  j \le  n}) \  \ \text{ and } \ \ \mathfrak C=(\mathcal C, w_1, (S_w)_{w\in VC}, (q_j)_{1\le j \le  n}).$$ 
 Let $F$ be the fold   that turns $\mathfrak B$ into $\mathfrak C$. By definition, $F(u_1)=w_1$,  $q_j=F(p_j) $ for all $1\le j\le n$ and $T_{u}$ is a subtuple of $S_{F(u)}$  for all $u\in VB$.

 As $\mathfrak B$ does not fold squares and is not collapsible,   there is a decorated  sub-$\mathbb A$-graph  
 $$\mathfrak B'=(\mathcal B', u_1', (T_u)_{u \in VB'}, (p_{j_k})_{1\le  k\le  n'})$$  of $\mathfrak B$ such that any  pair  of edges  in  $B'$ (the  underlying graph of $\mathcal B'$) is crossed twice  (in opposite directions)  by the peripheral paths $p_{j_1}, \ldots, p_{j_{n'}}$ and where $u_1'\in VB'$ such that  $[u_1']=v_1\in VA$.

\begin{claim}
$F$  maps $B'$  isomorphically onto $C_{core}$. In particular, $\mathfrak B'$ contains all peripheral paths  $p_1, \ldots, p_n$.
\end{claim} 
\begin{proof}[proof of claim]
We  compare the local graphs  in $\mathfrak B'$ with the corresponding local graphs  in $\mathfrak C$.   We observe the following:
\begin{enumerate}
\item[(a)]  As $\mathfrak B'$ does not fold squares and every edge is crossed twice by $p_{j_1}, \ldots, p_{j_k}$ it follows that,   for each $u'\in VB'$,  the local  graph $\Gamma_{\mathfrak B'}( u')$ is a non-empty union  of circuits.

\item[(b)] The definition of (almost) orbifold covers  implies that  $w\in VC_{core}$ iff $\Gamma_{\mathfrak C}(w)$  contains a circuit.  Moreover,  for any $w\in VC_{core}$,  the   components of  $\Gamma_{\mathfrak C}(w) \setminus\Gamma_{core(\mathfrak C)}(w)$  are degenerate segments (i.e.~consist of single vertices).
\end{enumerate}

Let  $u'\in VB'$. As $\Gamma_{\mathfrak B'}(u')$ contains a circuit  we see that the corresponding local graph  $\Gamma_{\mathfrak C}( F(u'))$ contains  a circuit.  It follows from (b) that    $F( u')\in VC_{core}$. As $\mathfrak C$ does not fold squares, $F$ cannot identify edges in $Star(u', B')$. Thus $F$ induces a bijection 
$$ V\Gamma_{\mathfrak B'}( u') =Star(u ', B') \to Star(F(u'), C_{core}) =V\Gamma_{core(\mathfrak C)}(F(u')).$$ 
Thus $F|_{B'}:B'\to C_{core}$ is a covering.  As edges from $ B'$ cannot be identified (since otherwise $\mathfrak C$ would fold squares) we see that $F|_{B'}$ has degree one. Therefore 
$F|_{B'}$ is an isomorphism. 
\end{proof}

%Assume first that $F$ is of type IIIA. Thus two edges $f_1$ and $f_2$   with $x:=\alpha(f_1)=\alpha(f_2)$ and $y:=\omega(f_1)=\omega(f_2)$ are identified.   The effect of the fold is to replace $T_y$ by $S_{F(y)}=T_y\oplus (g)$ for some $f\in A_{[y]}$.  Since by definition the size of the tuples   in $\mathfrak C$ is at most one, it follows that $T_y$ is empty.   As $S_{F(y)}$ is non-empty we conclude that $F(y)$ lies in $C_{core}$  and therefore $y$ lies in $B'$.  As $T_y=\emptyset$ it follows that $B_y=1$ and so $length(\Gamma_{\mathfrak B'}(y)= val([y], A)\cdot |A_{[y]}|$.Using the fact that $F$ induces an isomrophism we coclude that $length(\Gamma_{core(\mathfrak C}(F(y)))=val([y], A)|A_{[y]}|$ and so $core(\mathfrak C)$ is folded at $F(y)$ because $F(y)$ cannot be the exceptional vertex of $\mathfrak C$.     But this implies that $S_{F(y)}=(1)$ and so $\mathfrak C$ is a surface cover. Consequently all vertex tuples in $\mathcal B$ are empty, a contradiction.

%Assume now that $F$ is of type IA based on $f_1$ and $f_2$ with $y_i=\omega(f_i)$.   We can assume that at least one of the tuples $T_{y_1}$ or $T_{y_2}$ is empty.  
\begin{claim}
$F$ does not affect any vertex  tuple of $\mathfrak B'$.   
\end{claim}
\begin{proof}
Suppose that a vertex  tuple of $\mathfrak B'$ is affected by $F$, that is, $T_{x'} \neq S_{F(x')}$ for some $x'\in VB'$.  The vertex tuple $T_{x'}$  must be empty as all vertex tuples in $\mathfrak C$ are of size at most one. In particular,  $B_{x'}=1$. Condition (C2),  or condition  (C2') in the almost orbifold cover case,  combined with remark~\ref{rem:local}(4) imply  that  $\mathfrak B'$ is folded at $x'$. Since $F|_{B'}$ is an isomorphism it follows that $core(\mathfrak C)$ is also folded at $x$. It follows from condition (C2) or (C2') depending on $\mathfrak C$  combined with remark~\ref{rem:local}(4)   that $C_{F(x')}=1$. Thus $S_{F(x)}=(1)$ and so $\mathfrak C$ is a surface cover. This implies that all vertex tuples in $\mathfrak B $ are empty, a contradiction.    
\end{proof}

The previous claim implies  that the fold cannot be of type IIIA because the affected vertex would lie in $C_{core}$ and therefore in $B'$.  It also implies that if a vertex does not lie in $B'$ then the associated tuple is empty.   In fact, if $y$ has non-empty tuple then $F(y)$ lies necessarily in $C_{core}$.  Thus there is $y'$ in $B'$ such that $F(y')=F(y)$ and so $T_y=\emptyset$ and $S_{F(y')}=S_{F(y)}=T_{y} \neq T_{y'}$.   

Therefore $B'= B_{core}$ and all vertices  not in $B_{core}$ are equipped with the empty tuple. This completes the proof that $\mathfrak B$ is an (almost) orbifold cover.  
\end{proof}

%------------------------------------------------------------

\subsection{The directed graph $\Omega_{[\mathcal P]}$}
In \cite{W2} a graph of equivalence classes of marked $\mathbb A$-graphs representing a fixed Nielsen equivalence class of generating tuples of a free product is defined. In this section we define a similar graph for a fixed   equivalence class of partitioned tuples in $\pi_1(\mathbb A, v_1)\cong \pi_1^o(\mathcal O)$.  
Let $\mathcal P$ be a partitioned tuple in $\pi_1(\mathbb A , v_1)$. We define  $\Omega_{[\mathcal P]}$ as follows:
\begin{enumerate}
\item The vertices of $\Omega_{[\mathcal P]}$ are the equivalence classes of minimal (with respect to the base vertex) tame decorated $\mathbb A$-graphs  that represent $[\mathcal P]$. 

\item Two vertices $\mathfrak b$  and $\mathfrak b'$ of $\Omega_{[\mathcal P]}$ are connected by a directed edge $\mathfrak b\mapsto  \mathfrak b'$ if there are representatives $\mathfrak B$ of $\mathfrak b$ and $\mathfrak B'$ of $\mathfrak b'$ such that:
\begin{enumerate}
\item There exists a tame decorated $\mathbb A$-graph $\bar{\mathfrak B}$ that is  obtained from $\mathfrak B$ by an elementary fold.
 
\item  $\mathfrak B'=  core(\bar{\mathfrak B}, \bar{u}_1)$ where $\bar{u}_1$ corresponds to the base vertex $u_1$ of $\mathfrak B$. 
\end{enumerate}
\end{enumerate}
 
We say that a vertex $\mathfrak b$ \emph{projects} onto a vertex $\mathfrak b'$ if there is an oriented path in  $\Omega_{[\mathcal P]}$  from $\mathfrak b$ to $ \mathfrak b'$. A vertex that does not project onto any vertex is called a \emph{root}.

The \emph{height} of a vertex $\mathfrak b$ of $\Omega_{[\mathcal P]}$ is defined by  $h(\mathfrak b):=\frac{1}{2}|EB|$  where $B$ is the underlying graph of some and therefore any  representative of $\mathfrak b$. It follows immediately from the definition   that  $h(\mathfrak b')\le  h(\mathfrak b)-1$ if $\mathfrak b\mapsto\mathfrak b'$.

\begin{lemma}{\label{lemma:connectedness}}
The graph $\Omega_{[\mathcal P]}$ is connected.
\end{lemma}
\begin{proof} 
The proof is a variation of the proof of Lemma~7 of \cite{W2}. Let $\mathfrak b$ be a vertex of $\Omega_{[\mathcal P]}$. We first observe that there exists a vertex $\mathfrak b'$ of very specific type that projects onto $\mathfrak b$. 

Let $\mathfrak B$  be an arbitrary representative of $\mathfrak b$. Choose  a set of collapsing edges $\{f_1,\ldots ,f_n\}$ of $\mathfrak B$ and a maximal subtree $Y$ of $B_{f_1,\ldots ,f_n}$.  Let 
$$\mathcal P_{f_1,\ldots ,f_n}^{Y}=(T_{f_1,\ldots ,f_n}^{Y}, P_{f_1,\ldots ,f_n}^{Y})$$ 
be the associated partitioned tuple in $\pi_1(\mathbb A, v_1)$. By definition  $\mathcal P^Y_{f_1,\ldots ,f_n}$ is equivalent to $ \mathcal P$. We  construct a tame decorated marked $\mathbb A$-graph  $\mathfrak B'$, which will be called \emph{normal},   that represents $[\mathcal P]$  and such that the following hold:
\begin{enumerate}
\item the underlying graph $B'$ of $\mathfrak B'$  is topologically the one point union of circles and lollipops, see Figure~\ref{normaldecorated}.

\item  Any element of $T^Y_{f_1,\ldots ,f_n}=(g_1,\ldots ,g_m)$ is represented by a  possibly degenerate  loop based at the base vertex $u_1'$  and any peripheral element of $P^Y_{f_1,\ldots ,f_n}=(\gamma_1,\ldots ,\gamma_n)$ is represented  by a lollipop  with possibly degenerate stick; the decoration is by the simple closed path corresponding to the  candy part of the  lollipops. 

\item $\mathfrak b'$ projects onto $\mathfrak b$, in particular $\mathfrak b'$  and $\mathfrak b$ lie in the same component of $\Omega_{[\mathcal P]}$. 
\end{enumerate}

Note that any normal decorated $\mathbb A$-graph  is tame and minimal (with respect to the base vertex). 
 
The construction has two steps. In the first step  we unfold the peripheral elements into lollipops in successive order using auxiliary moves an inverse folds of type IA. Note that this is possible as $f_1,\ldots ,f_n$ are collapsing edges meaning that we can unfold the $k$-th lollipop starting at both ends of $f_k$ only affecting edges of $B_{f_1,\ldots ,f_k}$. 
\begin{figure}[h!]
\begin{center}
\includegraphics[scale=1]{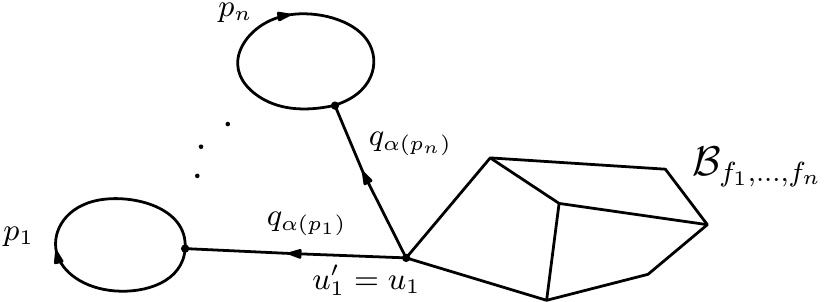}
\end{center}
\caption{The lollipops correspond to peripheral elements.}
\end{figure}
Let $\mathcal B_{f_1,\ldots ,f_n}$ be the sub-$\mathbb A$-graph of $\mathcal B$ corresponding to the sub-graph  $B_{f_1,\ldots ,f_n}$ of $B$. As  in the unfolding process  the edges $f_1, \ldots, f_n$ are removed, no peripheral path from $\mathfrak B$  lies in  $\mathbb B_{f_1, \ldots, f_n}$. Hence the decorated $\mathbb A$-graph $\mathfrak B_{f_1, \ldots, f_n}$ carried by  $\mathcal B_{f_1, \ldots, f_n}$ has no  peripheral paths and therefore  $\mathfrak B_{f_1, \ldots, f_n}$ is a marked $\mathbb A$-graph in the sense of~\cite{W2}.

In the second step we unfold $\mathfrak B_{f_1,\ldots ,f_n}$ into a wedge of circles corresponding to the elements of $T^Y_{f_1,\ldots ,f_n}$ as in the proof of Lemma~7 of \cite{W2}. Note that the elements represented by elements of $T_{u_1}$ are not unfolded, i.e. are represented by degenerate loops.

\begin{figure}[h!]
\begin{center}
\includegraphics[scale=0.9]{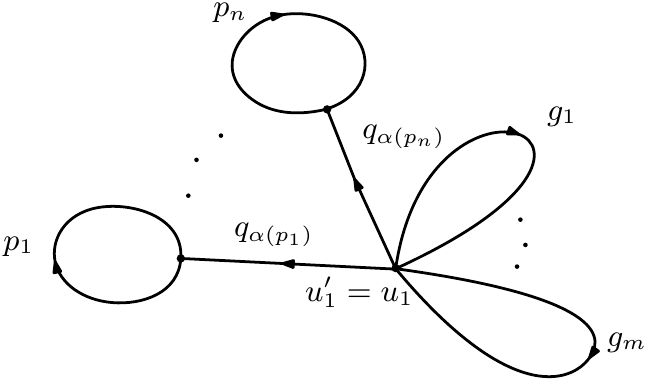}
\end{center}
\caption{The loops correspond to the non-peripheral elements.}{\label{normaldecorated}}
\end{figure}

The obtained normal decorated $\mathbb A$-graph $\mathfrak B'$ is clearly tame and so are all intermediate decorated $\mathbb A$-graphs in the folding sequence  that carries $\mathfrak B'$ into $\mathfrak B$. Thus the vertex  $\mathfrak b'$ represented by $\mathfrak B'$ and the vertex  $\mathfrak b$ represented by  $\mathfrak B$    lie in the same component of $\Omega_{[\mathcal P]}$. 

\smallskip
 
Note that vertices  of $\Omega_{[\mathcal P]}$ that are represented by normal tame decorated  $\mathbb A$-graphs   lie in the same component of $\Omega_{[\mathcal P]}$. Indeed,  this follows from the fact that any  partitioned tuple is represented by a unique equivalence class  $[\bar{\mathfrak B}]$  of normal decorated $\mathbb A$-graphs that have the property that  any loop and any lollipop is mapped to a reduced $\mathbb A$-path. Moreover, any  other normal tame decorated  $\mathbb A$-graph folds onto  a normal decorated $\mathbb A$-graph equivalent to $ \bar{\mathfrak B} $.

\smallskip

It therefore suffices to show that for any normal decorated $\mathbb A$-graph $\mathfrak B$ with corresponding  partitioned tuple $\mathcal P_1$ and any partitioned tuple $\mathcal P_2$  equivalent to $\mathcal P_1$, there exists a normal decorated marked $\mathbb A$ graph $\mathfrak B_2$   with corresponding partitioned tuple  $\mathcal P_2$ such that $\mathfrak b_1$ and $\mathfrak b_2$ lie in the same component of $\Omega_{[\mathcal P]}$. It clearly suffices to consider the case that $\mathcal P_1$ and $\mathcal P_2$ are elementary equivalent. 
 
The argument is now again similar to that in the proof of Lemma~7 in \cite{W2}. If a non-peripheral element $g_i$ is left or right multiplied with some element $h$ represented by some lollipop or some loop distinct from the loop representing $g_i$, then the corresponding  loop is unfolded along this lollipop or loop, see Figure~\ref{unfoldNielsen1}, and if some peripheral element is conjugated by some $h$ then the lollipop is unfolded along the corresponding  lollipop or loop, see Figure~\ref{unfoldNielsen2}. 
\begin{figure}[h!]
\begin{center}
\includegraphics[scale=1]{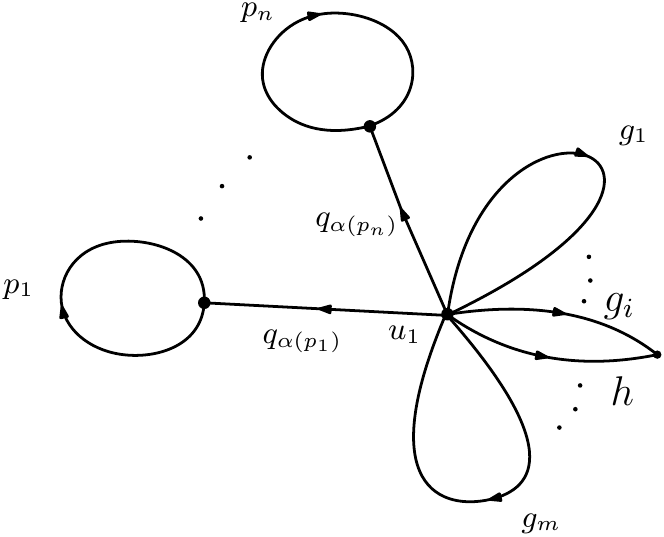}
\end{center}
\caption{Some non-peripheral element $g_i$ is right multiplied with $h$.}\label{unfoldNielsen1}
\end{figure}

\begin{figure}[h!]
\begin{center}
\includegraphics[scale=1]{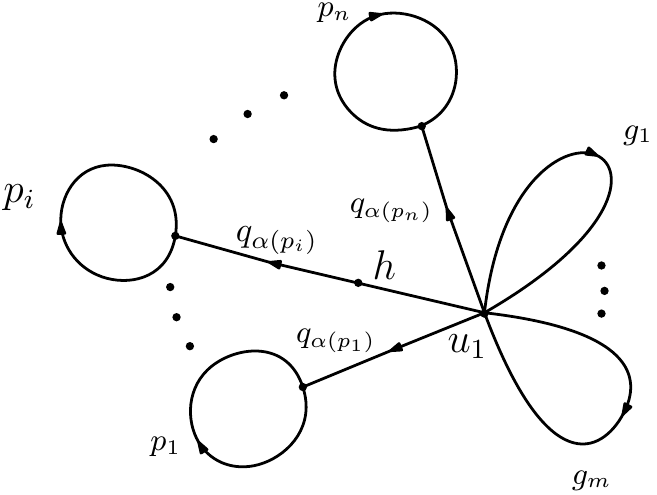}
\end{center}
\caption{Some peripheral element is conjugated by $h$.}\label{unfoldNielsen2}
\end{figure}
\end{proof}

\begin{definition}
Let $\mathfrak b$ be a vertex of $ \Omega_{[\mathcal P]}$.  We say that $\mathfrak b$ is: 
\begin{enumerate}
\item  a pre-(almost) orbifold cover if $\mathfrak b$ has a representative  that folds onto an (almost) orbifold cover with a single fold.

\item  pre-bad if $\mathfrak b$ has a representative that folds  onto a bad decorated $\mathbb A$-graph with a single fold.  
\end{enumerate}
\end{definition}

\begin{lemma}{\label{lemma:orbifoldcov}}
If $\mathfrak b$ is a pre-(almost) orbifold cover, then $\mathfrak b$ is not pre-bad.
\end{lemma}
\begin{proof}
The proof follows immediately    from Lemma~\ref{remark:prealmost}.
\end{proof}

\begin{lemma}{\label{lemma:prebad}}
There is a pre-bad vertex in $\Omega_{[\mathcal P]}$ iff   $\mathcal P$ is of type \textbf{(4)}.
\end{lemma}
\begin{proof}
Assume that $\mathfrak b$ is a pre-bad vertex of $\Omega_{[\mathcal P]}$. By definition $\mathfrak b$ has   a representative $\mathfrak B$ that folds onto a bad decorated $\mathbb A$-graph $\mathfrak B'$.  Thus we are in one of the following cases: 
\begin{enumerate}
\item $\mathfrak B'$ folds squares.

\item $\mathfrak B'$  does not fold squares  and one of the following occurs:
\begin{enumerate}
\item $\mathfrak B'$ is collapsible.

\item $\mathfrak B'$ is not collapsible. 
\end{enumerate}
\end{enumerate}

If (1) occurs, then   \cite[Lemma 3.18]{Dut}  implies that $\mathcal P$  fols peripheral elements or has an obvious relation; hence $\mathcal P$ is of type \textbf{(4)}.

If  (2.a) occurs, then  some vertex tuple in   $\mathfrak B'$  is non-minimal  (and therefore reducible). In this case,  as the corresponding partitioned tuple  $\mathcal P_{\mathfrak B'}$ contains a copy of each vertex tuple of $\mathfrak B'$,  we conclude that $\mathcal P_{ \mathfrak B'}$ and therefore  $\mathcal P$  is   reducible. Hence $\mathcal P$ is of type \textbf{(4)}.

If (2.b) occurs,  then  the proof of \cite[Proposition 3.38]{Dut}  shows that $\mathcal P$   is reducible, or folds peripheral elements,  or has an obvious relation. Hence $\mathcal P$ is of type \textbf{(4)}.
   
\smallskip  

To see that the converse holds we give a pictorial  description of a tame decorated $\mathbb A$-graph that represents $[\mathcal P]$ and folds onto a bad decorated $\mathbb A$-graph.

Assume that $\mathcal P$ is reducible. By definition  $\mathcal P$ is equivalent to  $\mathcal P'=(T, P)$ such that $T=(1, g_2, \ldots, g_m )$. Consider the decorated $\mathbb A$-graph $ \mathfrak B $ shown in Fig.~\ref{fig:reducible1}  where   $f_1, f_2\in EB$ are both labeled $(1, e, 1)$.  The fold that identifies $f_1$ and $f_2$ yields  a bad decorated $\mathbb A$-graph   as it replaces  $T_{u'}=\emptyset$ by $T_{u'}'=(1)$.  Therefore   $\mathfrak b:=[\mathfrak B] \in V \Omega_{[\mathcal P]}$ is pre-bad.
\begin{figure}[h!]
\begin{center}
\includegraphics[scale=1]{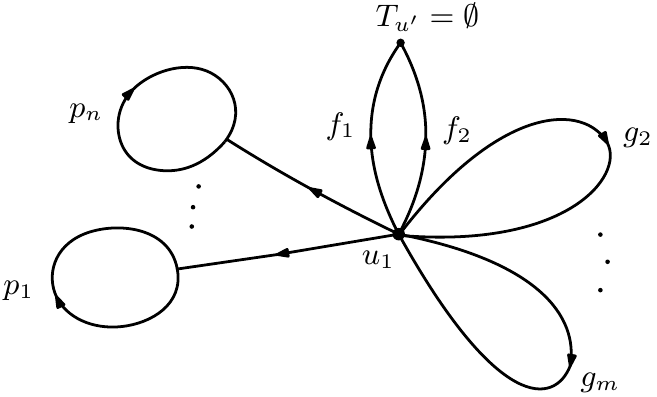}
\end{center}
\caption{The pre-bad decorated $\mathbb A$-graph $\mathfrak B$ in the case $\mathcal P$ is reducible.}{\label{fig:reducible1}}
\end{figure}

Assume that  $\mathcal P$ folds peripheral elements.  By definition,  $\mathcal P$ is equivalent to  $\mathcal P'=((g_1, \ldots, g_m), (  \gamma_1, \ldots, \gamma_n))$ such that  $i:=i_1=i_2$ and $o_2=o_1 c_i^z$ for some integer $z$ where $(o_k, i_k)$ is the label of $\gamma_k$ for $k=1, \ldots , n$. Consider the  normal   decorated $\mathbb A$-graph $\mathfrak B$ shown in Fig.~\ref{fig:reducible2}, where $q'$ and $q''$ are $\mathbb{B}$-paths such that   $\mu_\mathcal B(q')=\mu_\mathcal B(q)$ and $\mu_\mathcal B( q'')=c_i^z.$  After folding $q'q''$ into $qp_1$  we obtain a  tame decorated $\mathbb A$-graph  as shown in Fig.~\ref{fig:reducible2a}. The resulting decorated $\mathbb A$-graph clearly folds onto a   decorated $\mathbb A$-graph that folds squares. Consequently  $\Omega_{[\mathcal P]}$ contains a pre-bad vertex. 
\begin{figure}[h!]
\begin{center}
\includegraphics[scale=1]{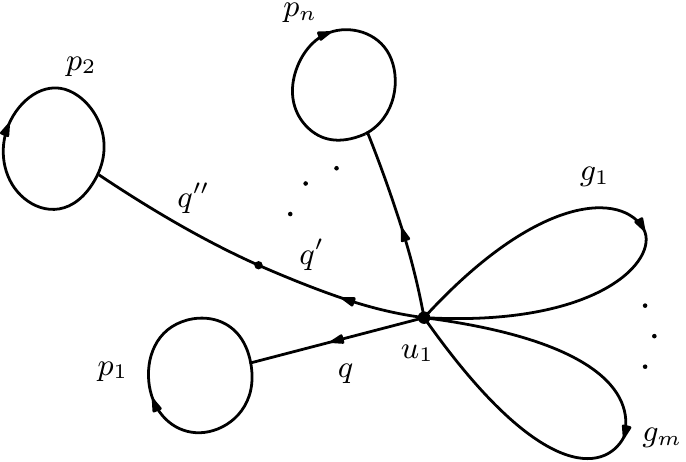}
\end{center}
\caption{The  decorated $\mathbb A$-graph $ \mathfrak B $ in the case $\mathcal P$ folds peripheral elements.}{\label{fig:reducible2}}
\end{figure}

\begin{figure}[h!]
\begin{center}
\includegraphics[scale=1]{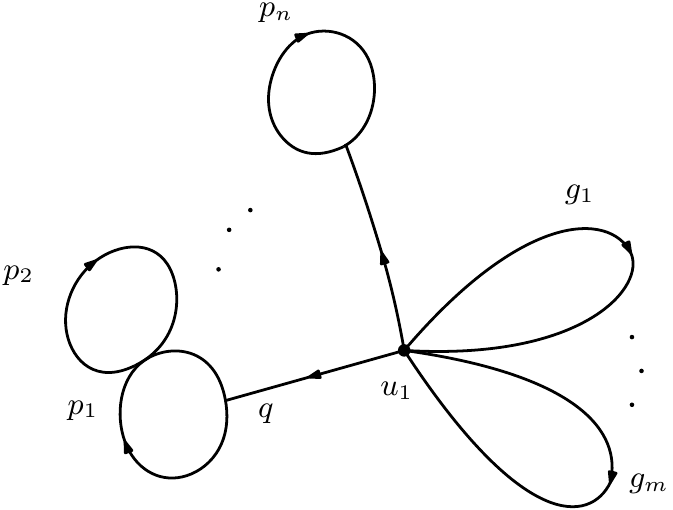}
\end{center}
\caption{Folds onto a bad decorated $\mathbb A$-graph.}{\label{fig:reducible2a}}
\end{figure}
Finally, assume  $\mathcal P$ has an obvious relation.  By definition,  $\mathcal P$ is equivalent to 
$$((g_1, \ldots, g_m), ( \gamma_1,\ldots , \gamma_n))$$  such that 
\begin{equation}{\label{lemma:obviousrelation}}
 w_n:=|o_n\langle c_{i_n}\rangle o_n^{-1}: U \cap o_n\langle c_{i_n}\rangle o_n^{-1}|< | o_n\langle c_{i_n}\rangle o_n^{-1}:  \langle \gamma_n\rangle|=z_n\tag{*}
\end{equation}
where $U\leq \pi_1^o(\mathcal O)$ is generated by $\{g_1, \ldots, g_m, \gamma_1, \ldots, \gamma_{n-1}\}$ and $(o_k, i_k)$ is the label of $\gamma_k$ for $k=1, \ldots , n$. We may assume $o_n=1$,  so that $\gamma_n=c_{i_n}^{z_n}$. 
We consider three cases depending on $\mathcal P':=((g_1, \ldots, g_m),(\gamma_1, \ldots, \gamma_{n-1})).$  

\noindent\textit{Case 1.}  $\mathcal P'$ is of simple type. Thus there is a folded tame decorated $\mathbb A$-graph $\mathfrak B'$ that represents $[\mathcal P']$. Let $\mathfrak B$  be the tame decorated $\mathbb A$-graph  described in Fig.~\ref{fig:reducible3}, that is, a  circuit $p=p'p''$ is glued to the base vertex $u_1'$ of $\mathfrak{B}'$  such that $\mu_{\mathcal B}(p)=c_{i_n}^{z_n}$ and $\mu_\mathcal B(p')=c_{i_n}^{w_n}.$   
It follows from (\ref{lemma:obviousrelation}) that  $p'$  can be folded  with a closed path that lies  entirely in $\mathbb B'$.  The resulting decorated $\mathbb A$-graph   folds onto a decorated $\mathbb A$-graph that folds peripheral paths because the first edge of $p'$  can be folded with the first edge of $p''$, see Fig.~\ref{fig:reducible3}. Hence  the corresponding edge is crossed twice by the peripheral paths in the resulting marked $\mathbb{A}$-graph.   
\begin{figure}[h!]
\begin{center}
\includegraphics[scale=1]{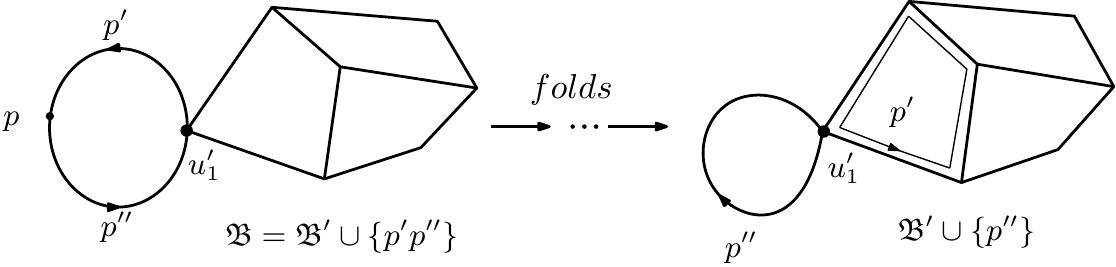}
\end{center}
\caption{The  decorated $\mathbb A$-graph $\mathfrak B=\mathfrak B'\cup \{p'p''\}$.}{\label{fig:reducible3}}
\end{figure}

\noindent{\textit{Case 2.}}  $\mathcal P'$ is of (almost) orbifold covering type. Let $\mathfrak B'$ be a minimal  tame decorated $\mathbb A$-graph  that represents $[\mathcal P']$ such  that  a single elementary fold of type IIIA turns $\mathfrak B'$ into an (almost) orbifold covering type. Let $\mathfrak B$ be the tame  decorated $\mathbb A$-graph that is obtained from $\mathfrak B'$  by gluing a path $p$ to the base vertex $u_1'$ of $\mathfrak{B}'$  such that $\mu_{\mathcal B'}( p)= c_{i_n}^{z_n}$.  Then $\mathfrak B$  is  tame and contains $\mathfrak B'$ as a decorated sub-$\mathbb A$-graph.  Observe that the  fold that turns $\mathfrak B'$ into an  (almost) orbifold cover can also be applied to $\mathfrak B$  and so yields a decorated $\mathbb A$-graph   which is the join of an (almost) orbifold cover and a circuit, and hence is   not a (almost) orbifold cover. Therefore the vertex represented by $\mathfrak B$ is pre-bad.

\noindent{\textit{Case 3.}} $\mathcal P'$ is reducible, or fold peripheral elements, or has an obvious relation. In this case we apply the same argument as in the previous paragraphs  with $\mathcal P'$ playing the role of $\mathcal P$.     
\end{proof}

%--------------------------------------------------------------------------

\subsection{Proof of Proposition~\ref{prop:01}  }
In this subsection  we study bifurcations in  the directed graph $\Omega_{[\mathcal P]}$ motivated by the fact that Proposition~\ref{prop:01} follows easily once we prove that  one of the following happens:
\begin{enumerate}
\item[(R1)] There is a unique root in  $\Omega_{[\mathcal P]}$ represented by a folded decorated $\mathbb A$-graph. In this case, according to    \cite[Lemma~3.23]{Dut}   $\mathcal P$ is of simple type, that is, is of type \textbf{(1)}.

\item[(R2)]   All roots of $\Omega_{[\mathcal P]}$ are pre-orbifold covers and they fold onto  equivalent  orbifold covers. In this case, according to \cite[Lemma~3.37]{Dut},    $\mathcal P$ is of  orbifold covering type, that is, $\mathcal P$  is of type \textbf{(2)}.

\item[(R3)] All roots of $\Omega_{[\mathcal P]}$ are pre-almost orbifold covers that fold onto equivalent almost orbifold covers. In this case, according to  \cite[Lemma~3.37]{Dut},  $\mathcal P$ is of almost orbifold covering type, that is, $\mathcal P$ is of type \textbf{(3)}.

\item[(R4)] All roots of $\Omega_{[\mathcal P]}$ are pre-bad. In this case,   $\mathcal P$ is of type \textbf{(4)}.
\end{enumerate} 
 
This is accomplished by proving the following two lemmas.
\begin{lemma}{\label{lemma:bad}}
Assume that  $\mathfrak b \mapsto \mathfrak b'$. If $\mathfrak b$ is pre-bad, then  $\mathfrak{b}'$ is pre-bad.
\end{lemma} 

\begin{lemma}{\label{lemma:bifurcation}}
Let $\mathfrak b$, $\mathfrak b_1$ and $\mathfrak b_2$ be vertices of $\Omega_{[\mathcal P]}$. Suppose that $\mathfrak b_1\neq\mathfrak b_2$ and $\mathfrak b\mapsto\mathfrak b_i$ for $i=1, 2$.  Then one of the following holds:
\begin{enumerate}
\item[(A)] $\mathfrak b_1$ and $\mathfrak b_2$ are pre-bad vertices.

\item[(B)] $\mathfrak{b}_1$ and $\mathfrak{b}_2$ are pre-(almost) orbifold covers that project onto equivalent  (almost) orbifold covers.

\item[(C)] there is a vertex $\mathfrak b'$ such that $\mathfrak{b}_i\mapsto \mathfrak{b}'$ for $i=1, 2$. In other words, $\mathfrak b_1$ and $\mathfrak b_2$ have a common projection in $\Omega_{[\mathcal P]}$.
\end{enumerate} 
\end{lemma}

\begin{proof}[Proof of Lemma~\ref{lemma:bad} and Lemma~\ref{lemma:bifurcation}]

For both proofs we need to consider the situation that elementary folds $F_1$ and $F_2$ are applicable to decorated $\mathbb A$-graphs equivalent to a minimal (with respect to the base vertex) tame decorated  $\mathbb A$-graphs.  $\mathfrak B$.

%For both proofs we need to consider the situation that a pair of (not necessarily  elementary) folds $F_1$ and $F_2$ is applicable to a minimal (with respect to the base vertex) tame decorated  $\mathbb A$-graph  $\mathfrak B$.

Suppose  that $F_1 $ identifies the edges $f_1$ and $f_2$  and is based on   $x:=\alpha(f_1)=\alpha(f_2)\in VB$   and   $F_2 $ identifies the edges $f_3$ and $f_4$ and is based on $x':=\alpha(f_3)=\alpha(f_4)\in VB$.   Let  $ \mathfrak B_1$ (resp.~$\mathfrak B_2$)  denote the  decorated $\mathbb A$-graph that is obtained from $\mathfrak B$  by  $F_1$ (resp.~$F_2$).

Lemma~\ref{lemma:foldsameedges} implies that any elementary  fold that identifies  $f_1$ and $f_2$ (resp.~$f_3$ and $f_4$)  in a decorated $\mathbb A$-graph  equivalent to $\mathfrak B$  yields a decorated $\mathbb A$-graph equivalent to $\mathfrak B_1$ (resp.~$\mathfrak B_2$).  This allows us to replace   $\mathfrak B$  with an equivalent decorated $\mathbb A$-graph, which we still denote by $\mathfrak B$ (observe that we only need auxiliary moves of type A2) in which   the labels of $f_1$ and $f_2$ are $(a, e, b_1)$ and $(a, e, b_2)$, and the labels of $f_3$ and $f_4$ are $(c, \bar e, d_1)$ and $(c, \bar e, d_2)$. After  applying an   auxiliary move of type A0 to  $\mathfrak B$  and possibly after  exchanging $f_3$ and $f_4 $ we can further assume that  one of the following holds:\begin{enumerate}
\item[(i)]  $f_3=f_1$ and $f_4=f_2$, i.e.~$F_1$ and $F_2$ identify the same pair of  edges in $\mathfrak B$.

\item[(ii)] $F_1$ and $F_2$   are elementary of type IIIA and    $f_3=f_1^{-1}$  and $f_4=f_2^{-1}$.
\begin{figure}[h!]
\begin{center}
\includegraphics[scale=1]{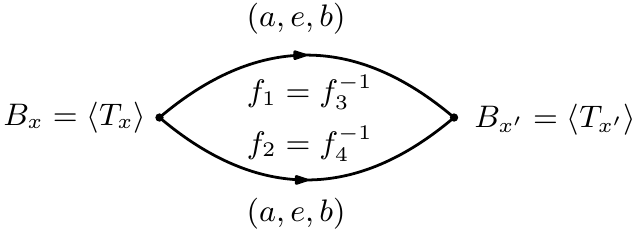}
\end{center}
\caption{$F_1$ and $F_2$ are of type IIIA.}{\label{fig:caseii}}
\end{figure}

\item[(iii)] $F_1$ and $F_2$   are  of type IA and  $\omega(f_3)=\omega(f_1)$ and $\omega(f_4)=\omega(f_2)$.
\begin{figure}[h!]
\begin{center}
\includegraphics[scale=1]{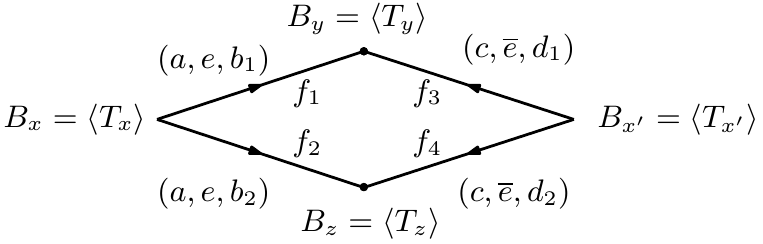}
\end{center}
\caption{$F_1$ and $F_2$ are of type IA.}{\label{fig:caseiii}}
\end{figure}

\item[(iv)]  $F_1$ and $F_2$ are elementary and commute, i.e.~if $i\neq j$ then $F_i$ does not change the topology of the subgraph that carries $F_j$.
\end{enumerate}

We first show that   $\mathfrak B_1$ is equivalent to $\mathfrak B_2$ if (i) or  (ii) occurs. In fact,  if (i) occurs  then the equivalence follows from  Lemma~\ref{lemma:foldsameedges}. If (ii) occurs, then $\mathfrak B_1$ and $\mathfrak B_2$ differ  by a trivial element at the vertex tuples   at $x$ and   $x'$, see Fig.~\ref{fig:caseii1}.  This  follows as $F_1$ adds the trivial element to the vertex tuple $T_{x'}$ and $F_2$ adds the trivial element to the vertex tuple $T_x$.  Thus $\mathfrak B_1$ and $\mathfrak B_2$ are equivalent.
\begin{figure}[h!]
\begin{center}
\includegraphics[scale=1]{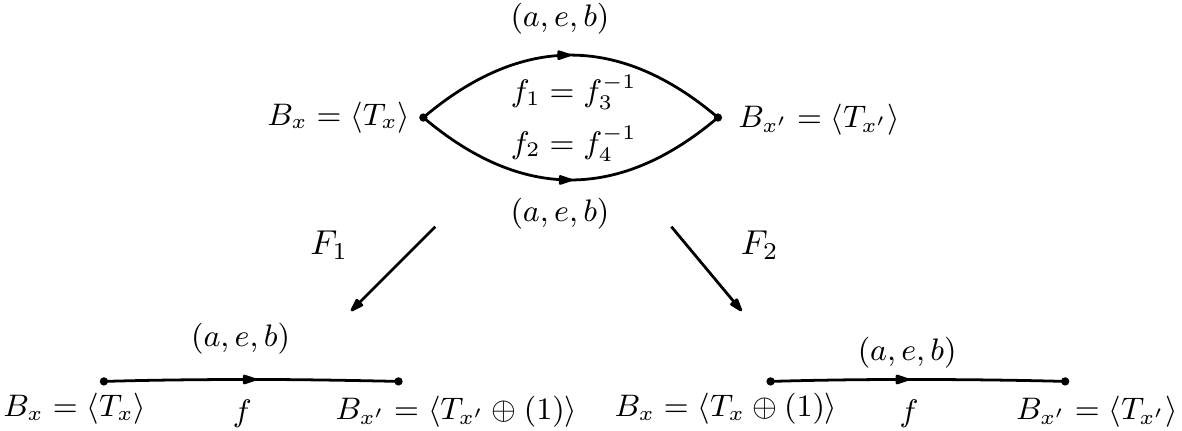}
\end{center}
\caption{$F_1$ and $F_2$ are of type IIIA.}{\label{fig:caseii1}}
\end{figure}

\smallskip

To prove  Lemma~\ref{lemma:bad} we need to consider the situation in which $\mathfrak B_1$ is tame  and $\mathfrak B_2$ is bad (or vice versa) and show that $\mathfrak B_1$ folds onto a bad decorated $\mathbb A$-graph. To prove Lemma~\ref{lemma:bifurcation}  we need to consider the situation in which $\mathfrak B_1$ and $\mathfrak B_2$ are tame but not equivalent and we need to show that one of the following holds:
\begin{enumerate}
\item  $\mathfrak B_1$ and $\mathfrak B_2$ fold onto bad decorated $\mathbb A$-graphs (thus (A) occurs).
\item  $\mathfrak B_1$ and $\mathfrak B_2$ fold onto equivalent (almost) orbifold covers (thus (B) occurs).
\item $\mathfrak B_1$ and $\mathfrak B_2$ fold onto equivalent tame decorated $\mathbb A$-graphs (thus (C) occurs).
\end{enumerate} 
We can therefore assume that  $\mathfrak B_1$  is not equivalent to $\mathfrak B_2$. We observe that configuration   (iii)  or configuration  (iv)  necessarily occurs.

 Let  $F_1'$  denote the fold that  identifies the edges $f_3$ and $f_4$ in $\mathfrak B_1$ and   $F_2'$ denote  the fold that identifies    the edges $f_1$  and $f_2$ in $\mathfrak B_2'$.  Observe that $F_1'$ and $F_2'$ are of type IIIA  if configuration  (iii) holds,    and $F_i$ and $F_i'$ ($i=1,2$)   are of same type if configuration  (iv) holds.  
$$
\begin{tikzcd}
 &  \mathfrak B  \arrow[dr,,   "F_2"]  \arrow[dl,  "F_1"'] &  \\
 \mathfrak B_1   \arrow[dr,  "F_1'"']   &  &  \mathfrak B_2 \arrow[dl,    "F_2'"]\\
   &   \mathfrak B_1' \approx  \mathfrak B_2' &   
\end{tikzcd}
$$
An argument entirely analogous to    \cite[Lemma 8]{W2} shows that the decorated $\mathbb A$-graph $\mathfrak B_1' $   that is  obtained from $\mathfrak B_1$  by  $F_1'$ is equivalent to the    decorate $\mathbb A$-graph $\mathfrak B_2'$  that is obtained from $\mathfrak B_2$  by   $F_2'$. Thus in both configurations   $\mathfrak B_1'$ is equivalent to $\mathfrak B_2'$. 

\smallskip

We consider the case in which $\mathfrak B_1 $ is tame and $\mathfrak B_2$  is bad (and therefore not tame)  and the case  in which  both $\mathfrak B_1$ and $\mathfrak B_2$  are tame.  The first case provides a proof for Lemma~\ref{lemma:bad} and the second case provides a proof for Lemma~\ref{lemma:bifurcation}.

\smallskip

\noindent{\em Case 1.} Assume that $\mathfrak B_1$ is tame and $\mathfrak B_2$  is bad.  We will show that $ \mathfrak B_2'$ is bad which implies that   $ \mathfrak B_1'$  is also bad. There are three subcases.

\noindent{\emph{Case 1(a):}}  $\mathfrak B_2$  folds squares.  By definition,   there is  $f\in EB_2$ such that $f$  is crossed twice by the peripheral paths in $\mathfrak B_2$.  Thus  $F_2'(f)\in EB_2'$ is crossed twice by the  peripheral paths in  $\mathfrak B_2'$. Therefore $\mathfrak B_2'$ (and also $\mathfrak B_1'$) folds squares.  This shows that $ \mathfrak B_1$ folds onto a bad  decorated $\mathbb A$-graph and is therefore pre-bad.

\noindent{\emph{Case 1(b):}}  $\mathfrak B_2$ does not fold squares and is collapsible. As $\mathfrak B_2$ is   not tame there is   $w\in VB_2$ such that  $T_w^2$ is not minimal.   Let  $v:=F_2'(w)\in VB_2'$.  Thus   $T_w^2$ is a sub-tuple of  $T_v^{2'}$ and so $T_v^{2'}$ is not minimal. If $\mathfrak B_2'$  is collapsible, then it is bad (and so is $\mathfrak B_1'$).  If  $\mathfrak B_2'$  is not collapsible,  then it follows from Lemma~3.20 of \cite{Dut} that $F_2'$ is of type IIIA. Thus $F_2'$ adds an element to some vertex tupleof $\mathfrak B_2$. Thus  two vertices have non empty tuple in $\mathfrak B_2'$ and one of them is not minimal   some vertex in $\mathfrak B_2'$ has tuple of size greater than $1$. In both cases   $\mathfrak B_2'$ cannot be an (almost) orbifold cover.  Therefore  $\mathfrak B_2'$ (and therefore also $\mathfrak B_1'$) is bad.
 
\noindent{\em Case 1(c):}  $ \mathfrak B_2$ does not fold squares and is not collapsible.  Lemma~3.20 of \cite{Dut}    implies  that  $F_2$ is of type IIIA, i.e.~$\omega(f_3)=\omega(f_4)$. Thus the tuple associated to   $y_2:=F_2(\omega(f_3))=F_2(\omega(f_4))\in VB_2$ is non-empty.    Since the property of being non-collapsible is preserved under folds we see that   $\mathfrak B_2'$ is not collapsible. Lemma~\ref{lemma:foldalmost} implies that $\mathfrak B_2'$ cannot be an (almost) orbifold cover because $\mathfrak B_2$ is bad. Therefore $\mathfrak B_2'$ (and therefore also  $\mathfrak B_1'$) is bad. This completes the proof of case (1).   

\smallskip

\noindent{\em Case 2.}  Assume now that $\mathfrak B_1$ and $\mathfrak B_2$ are tame. Denote by $\mathfrak b_i$ the  vertex of $\Omega_{[\mathcal P]}$ represented by $ core(\mathfrak B_i, u_i)$. If $\mathfrak B_i'$ is bad, then clearly $\mathfrak b_1$ and $\mathfrak b_2$  are pre-bad. If $\mathfrak B_i'$ is an (almost) orbifold cover, then $\mathfrak b_i$ are  pre-(almost) orbifold covers. Finally,  if $ \mathfrak B_i'$ is  tame, then $\mathfrak b_1$ and $\mathfrak b_2$ have a common projection, namely    $\mathfrak b':=[core(\mathfrak B_1', u_1')]=[core(\mathfrak B_2', u_2')].$ 
\end{proof}
%--------------------------------------------------------------------------- 

\begin{proof}[{Proof of Proposition~\ref{prop:01}}]
We %observed before that it suffices to 
first show that  $\Omega_{[\mathcal P]}$ has only roots of one type meaning that exactly one of cases (R1)-(R4) spelled out in the beginning of this subsection occurs. The proof is by contradiction. Thus we assume that $\Omega_{[\mathcal P]}$ contains distinct types. As shown in $\cite{W2}$ the set 
$$\Omega_{[\mathcal P]}^*:=\{ \mathfrak b \in \Omega_{[\mathcal P]} \ | \ \mathfrak b \text{ projects onto   roots  of distinct types}\}$$ 
is  non-empty. Choose $\mathfrak b \in \Omega_{[\mathcal P]}^*$ such that $h(\mathfrak b)$ is minimal. Let $\mathfrak b_1'$ and $\mathfrak b_2'$  be roots of distinct type that $\mathfrak b$ projects onto. Choose $\mathfrak b_1$ and $\mathfrak b_2$ such that $\mathfrak b\mapsto \mathfrak b_i$ and $\mathfrak b_i$ projects onto $\mathfrak b_i'$ for $i=1, 2$.  The minimality of $h(\mathfrak b)$ implies that there is no common projection of $\mathfrak b_1$ and $\mathfrak b_2$. As $\mathfrak b_1'$ and $\mathfrak b_2'$ are of distinct types,  $\mathfrak b_1$ and $\mathfrak b_2$ are not pre-(almost) orbifold covers  whose representatives  fold onto equivalent  (almost) orbifold covers.  Lemma~\ref{lemma:bifurcation} implies that $\mathfrak b_1$ and $\mathfrak b_2$ are bad vertices. Lemma~\ref{lemma:bad}  implies that  $\mathfrak b_1'$ and $\mathfrak b_2'$ are bad vertices  which means that they are the same type,  a contradiction.

%Therefore all roots in $\Omega_{[\mathcal P]}$ have the same type, i.e.~exactly of of cases (R1)-(R4) holds.  
This completes the proof of the proposition as the uniqueness claim in the case of partitioned tuples of (almost) orbifold covering type follows from Lemma~\ref{lemma:bifurcation}. 
\end{proof}

%------------------------------------------------------------

\subsection{Auxiliary results regarding partitioned tuples}

In this section we prove some auxiliary results  regarding     partitioned tuples  in the fundamental group of a small orbifold.  The main ingredient  is the notion of  \emph{foldability} of an $\mathbb A$-graph. 
 
\smallskip

Let $\mathcal C$ be an $\mathbb A$-graph and let  $u$ be a vertex of $\mathcal C$.  We say that two edges $f$ and $g$ starting at $u$  are equivalent if $[f]=[g]\in EA $ and $o_f^\mathcal C=c o_g^\mathcal C\in A_{[u]}$ for some $c\in C_u$ (that is, $f$ and $g$ can be folded by a not necessarily elementary fold). The number of distinct equivalence classes is denoted by $l_\mathcal C(u)$.

\begin{remark}\label{rem:numbereqclasses}
Observe that when  $\mathcal C$ is locally surjective the number of equivalence classes is given by 
$
l_\mathcal C(u)=val(v, A)\cdot |A_v:C_u| $ for any  $u\in VC$, where $v=[u]\in VA$.
\end{remark}

\begin{definition}
The \emph{foldability of $\mathcal C$ at  $u\in VC$}  is defined as
$$fold_\mathcal C(u):= val(u, C)-l_\mathcal C(u).$$
The \emph{foldability} of $\mathcal C$ is defined as
$$fold(\mathcal C):= \sum_{u\in VC} fold_{\mathcal C}(u).$$
\end{definition}

\begin{definition}
The \emph{foldability  of a decorated $\mathbb A$-graph $\mathfrak B$}  is defined as 
$$fold(\mathfrak B):=fold(core(\mathcal B))$$
where $\mathcal B$ is the underlying $\mathbb A$-graph of $\mathfrak B$.
\end{definition}

\begin{lemma}{\label{lemma_foldedness1}}
Let $\mathcal C$  and $\mathcal C'$ be $\mathbb A$-graphs. Assume that $\mathcal C$ is locally surjective and $\mathcal C'$ is obtained from $\mathcal C$ by a single fold of type IIIA affecting the vertex $y$ that preserves the rank of the fundamental group. 

Then the fold replaces the trivial vertex group $C_y$ by a non-trivial vertex group $C_y'$ and the following hold:
\begin{enumerate}
\item[(1)] $fold(\mathcal C')<fold(\mathcal C)$ if   and only if  the underlying surface of $\mathcal O$ is a disk and $C_y'=A_{[y]}\cong \mathbb Z_2$. Hence   $\mathcal O=D^2(2, m)$ for some $m\ge 2$.

\item[(2)] $fold(\mathcal C')= fold(\mathcal C)$  if  and only if   one of the following occurs:
\begin{enumerate}
\item the underlying  surface of $\mathcal O$ is an annulus and  $C_y=A_{[y]}\cong \mathbb Z_2$. Hence  $\mathcal O=A(2)$. 

\item the underlying surface of $\mathcal O$ is a disk, $C_y'\cong \mathbb Z_2$,  and  $A_{[y]}\cong \mathbb Z_4$. Hence   $\mathcal O=D^2(4, m)$ for some $m\ge 2$.
\end{enumerate}
\end{enumerate}
\end{lemma}
\begin{proof}
Assume the fold identifies the edges $f_1$ and $f_2$ with $x:=\alpha(f_1)=\alpha(f_2)$. As the fold affects the vertex $y$ we have $\omega(f_1)=\omega(f_2)=y$. 

The first claim is trivial as the fold is assumed to be rank preserving.  Note further that  $\mathcal C'$  is also locally surjective since any vertex  of $\mathcal C'$ is the image of some vertex of $\mathcal C$.  We can therefore use remark~\ref{rem:numbereqclasses} to compute   $fold(\mathcal C)$ and  $fold(\mathcal C')$.

At   $x$ we have   $fold_{\mathcal C'}(x)=fold_{\mathcal C}(x)-1$   
since   $C_x$ does not change and two equivalent edges starting at $x$ get identified. As two edges starting at $y$ also get identified, we obtain  
\begin{eqnarray}
fold_{\mathcal C'}(y)-fold_\mathcal C(y) & = &  val(y, C')  - l_{\mathcal C'}(y)- val(y, C)+l_{\mathcal C}(y) \nonumber\\
& = &  (val(y, C)-1) - val([y], A)\cdot |A_{[y]}:C_y'| -  \nonumber\\
&  &  - (val(y, C)   -val([y], A)\cdot |A_{[y]}:C_y| ) \nonumber\\
& = &  (val(y, C)-1) - val([y], A)\cdot |A_{[y]}:C_y'| -  \nonumber\\
&  &  - (val(y, C)   -val([y], A)\cdot |A_{[y]}| ) \nonumber\\
& =  & -1+val([y], A)\cdot (|A_{[y]}|-|A_{[y]}:C_y'|). \nonumber 
\end{eqnarray}
Therefore
\begin{eqnarray}
fold(\mathcal C')-fold(\mathcal C) & = & fold_{\mathcal C'}(x)-fold_{\mathcal C}(x)+fold_{\mathcal C'}(y)-fold_{\mathcal C}(y) \nonumber \\
 & = & -1 -1 +val([y], A) \cdot (|A_{[y]}| - |A_{[y]}:C_y'|)\nonumber\\
 & = & -2 + val([y], A) \cdot (|A_{[y]}| - |A_{[y]}:C_y'|)\nonumber
\end{eqnarray}

\noindent (1) The previous calculation shows that  $fold(\mathcal C')<fold(\mathcal C)$ if and only if  
$$val([y], A) \cdot (|A_{[y]}| - |A_{[y]}:C_y'|)\le 1,$$ 
which implies that 
$$val([y], A)=1 \ \text{ and }   \ |A_{[y]}:C_y'|=|A_{[y]}|-1.$$  
Note that $$|A_{[y]}:C_y'|=|A_w|-1 \ \text{ iff } \ A_{[y]}=C_y'\cong \mathbb Z_2.$$ Moreover $val([y], A)=1$ implies that the underlying surface of $\mathcal O$  is a disk which proves the claim.

\smallskip 

\noindent (2)  The calculations from the previous paragraphs  show that    $fold(\mathcal C')= fold(\mathcal C)$  if and only if  $val([y], A)\cdot (|A_{[y]}|-|A_{[y]}:C_y'|)=2.$  i.e.~if one of the following occurs:
\begin{enumerate}
\item[(i)] $val([y], A)=2$ and $|A_{[y]}:C_y'|=|A_{[y]}|-1$.

\item[(ii)] $val([y], A)=1$ and $|A_{[y]}:C_y'|=|A_{[y]}|-2$.
\end{enumerate}
Item  (i) occurs  if and only if  $A_{[y]}=C_y\cong \mathbb Z_2$. Hence we are in situation (2.a).   Item (ii) occurs if and only if  $A_{[y]}\cong \mathbb Z_4$ and $C_y'\cong \mathbb Z_2$ since  $\mathbb Z_n$ contains a non-trivial subgroup of index $n-2$  if and only if  $n=4$. As $val([y], A)=1$ the underlying surface  of $\mathcal O$ is a disk  and so we are in situation (2.b). 
\end{proof}

\begin{lemma}{\label{lemma_foldedness2}}
Let $\mathcal C$  and $\mathcal C'$ be $\mathbb A$-graphs. Assume that $\mathcal C$ is locally surjective and that  $\mathcal C'$ is obtained from $\mathcal C$ by a single fold of type IA identifying vertices $y_1$ and $y_2$ that preserves the rank of the fundamental group. 

Then $fold(\mathcal C')\ge  fold(\mathcal C)$ and $fold(\mathcal C')= fold(\mathcal C)$ if and only if one of the following hold:
\begin{enumerate}
\item[(a)] the underlying  surface of $\mathcal O$ is an annulus and    $A_{[y_1]}=A_{[y_2]}=1$.

\item[(b)] the underlying  surface of $\mathcal O$  is a disk and  $A_{[y_1]}=A_{[y_2]}\cong \mathbb Z_2$.
\end{enumerate}
\end{lemma}
\begin{proof}
Assume  that the fold  is based on the edges  $f_1$ and $ f_2$ with $x:=\alpha(f_1)=\alpha(f_2)$. Thus  $y_1=\omega(f_1)$ and  $y_2=\omega(f_2)$. Denote the common image of $y_1$ and $ y_2$ (resp.~$f_1$ and  $f_2$) under the fold by $y$ (resp.~$f$). Thus $[y]=[y_1]=[y_2]\in VA$ and $  [f]=[f_1]=[f_2]\in EA.$ 

As in the previous lemma   we have $fold_{\mathcal C'}(x)=fold_{\mathcal C}(x)-1$.   On the other hand, as the fold is ranking preserving,   we can assume that $C_{y_1}=1$. Consequently    $C_y'=C_{y_2}$. A simple calculation then  yields 
$$ fold_{\mathcal C'}(y)-fold_{\mathcal C}(y_1)-fold_{\mathcal C}(y_2)=  -1 + val([y], A) \cdot |A_{[y]}| $$
and so $$fold(\mathcal C')-fold(\mathcal C)=-2+val([y], A)\cdot |A_{[y]}|.$$ 
Consequently,
$$fold(\mathcal C')<fold(\mathcal C) \ \text{  iff } \   val([y], A)=1   \text{ and }  |A_{[y]}|=1.$$ 
But this cannot occur  because  if $\mathbb A$ has only one edge then its vertex groups are non-trivial.

Moreover,  $fold(\mathcal C')=fold(\mathcal C)$ if and only if one of the following holds:
\begin{enumerate}
\item[(i)]  $val([y], A)=2$ and $|A_{[y]}|=1$.

\item[(ii)] $val([y], A)=1$ and $|A_{[y]}|=2$
\end{enumerate}
Therefore we are in case (a) if (i) occurs and in case (b) if (ii) occurs. 
\end{proof}

\begin{definition} We say that  a partitioned tuple $\mathcal P$ in $\pi_1^o(\mathcal O)$ is \emph{critical} if one of the following holds:
\begin{enumerate} 
\item  $\mathcal P$ is of simple type and is equivalent to $ (T', P')$ such that $T'$ contains an angle-minimal element.

\item  $\mathcal P$ is of type \textbf{(2)}, that is,  $\mathcal P$ is of orbifold covering type.

\item $\mathcal P$ is of type \textbf{(4)}, that is,   $\mathcal P$ is reducible, folds peripheral elements,  or has an obvious relation.
\end{enumerate}
Otherwise we say that $\mathcal P$ is \emph{non-critical}.
\end{definition}

\begin{lemma}{\label{lemma:anglemin}}
Let $\mathcal P$ be a partitioned tuple in $\pi_1^o(\mathcal O)$. If $\mathcal P$ is equivalent to $(T, P)$ such that $T$ contains an angle-minimal element,  then $\mathcal P$ is critical. 
\end{lemma}
\begin{proof}
It suffices to show that $\mathcal P$ cannot be of almost orbifold covering type.   It is not hard to construct a tame decorated $\mathbb A$-graph $\mathfrak B$  that represents $[ \mathcal P]$  with the property that some vertex tuple contains an angle-minimal element.   Any decorated $\mathbb A$-graph (tame or not) that is obtained from $\mathfrak B$  by finitely many folds will have the same property. This shows  that $\mathfrak B$ does not fold onto a pre-almost orbifold cover because in  pre-almost orbifold covers all vertex tuples are angle-minimal. Therefore $\mathcal P$  is not of almost orbifold covering type. 
\end{proof}

\begin{proposition}{\label{lem:critical1}}
Let  $(T,  P)$ be  a  partitioned tuple of simple type in $\pi_1^o(\mathcal O)$ and  let  $b$ be an arbitrary element of $ \pi_1^o(\mathcal O)$. If   $U:=\langle T\oplus  P\rangle\le \pi_1^o(\mathcal O)$  is of finite index, then   $(T\oplus (b),  P)$ is critical.
\end{proposition}
\begin{proof} 
Choose a  folded  and  minimal (i.e.~$\mathcal B=core(\mathcal B, u_1)$)  tame decorated $\mathbb A$-graph    $$\mathfrak B=(\mathcal B,u_1, (T_u)_{u\in VB}, (p_j)_{1\le  j\le  n})$$ that represents the equivalence class of $(T, P)$. We construct a tame decorated $\mathbb A$-graph $\mathfrak B'$  such that   (i)  $\mathfrak B'$   represents  the equivalence class of $(T\oplus(b), P)$  and (ii)  $fold(\mathfrak B')= 2$.

We first   glue a circuit $p$ to the base vertex $u_1\in VB$ such that the corresponding $\mathbb A$-path $\mu_{\mathcal B'}(p)$ represents $b\in\pi_1(\mathbb A, v_1)\cong \pi_1^o(\mathcal O)$.  Denote the resulting $\mathbb A$-graph by $\bar{\mathcal B}$. 
\begin{figure}[h!]
\begin{center}
\includegraphics[scale=1]{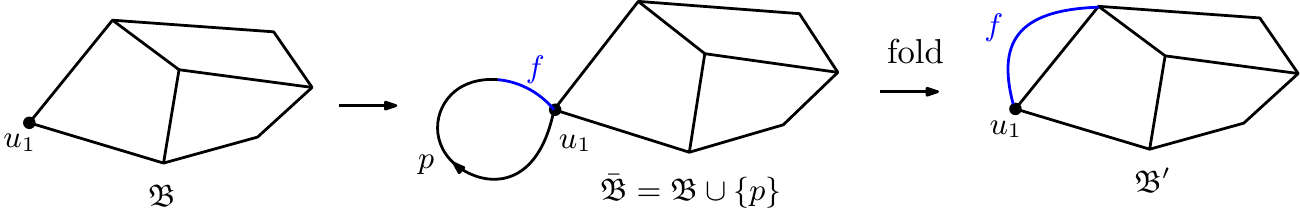}
\end{center}
\caption{Folding the loop $p$ into $\mathfrak{B}$.}{\label{fig:circuitsimple}}
\end{figure}
 Observe that $\mathcal B$ is a sub-$\mathbb A$-graph of $\bar{\mathcal B}$. Therefore we can define $\bar{\mathfrak B}:=(\bar{\mathcal B}, u_1, (T_u)_{u\in VB}, (p_j)_{1\leq  j\le  n}).$   
The hypothesis that $|G:U|<\infty$ implies that $\mathcal B$  is locally surjective. Therefore,  all but one edge of $p$, say $f$, can be folded into $\mathcal B$.  Let $\mathfrak B'$ denote the resulting decorated $\mathbb A$-graph. The construction of $\mathfrak B'$ is described in Fig.~\ref{fig:circuitsimple}.   Note that $\mathfrak B'$ is  tame since $\mathfrak B$ is not affected in this folding sequence.  As Fig.~\ref{fig:circuitsimple} suggests,   one can also obtain   $\mathfrak B'$  from ${\mathfrak B}$ by gluing a  non-loop edge  to $\mathcal B$. 
Since $VB'=VB$ and $B_u= B_u'$ for all $u\in VB'$ it follows that the $\mathbb A$-graph  underlying $\mathfrak B'$ is locally surjective. Moreover, the above description of $\mathfrak B'$  implies that  $fold(\mathfrak B')=2$.

Any fold that  preserves tameness necessarily   preserves  the rank of the fundamental group. Therefore, if $(T\oplus (b), P)$ is of almost orbifold covering type (resp.~of simple type) then $\mathfrak B'$ folds onto a pre-almost orbifold cover  (resp.~onto a folded decorated $\mathbb A$-graph) $\mathfrak B''$. In both cases $fold(\mathfrak B'')<2$  and therefore  some fold decreases foldability. Lemmas~\ref{lemma_foldedness1} and~\ref{lemma_foldedness2}  imply that there is a tame decorated $\mathbb A$-graph  $\mathfrak B'''$ that represents  $( T\oplus (b), P)$  such that some vertex tuple contains an angle-minimal element. Therefore $(T\oplus (b), P)$ is equivalent to $(T', P')$ such that $T'$ contains an angle-minimal element.  Lemma~\ref{lemma:anglemin} implies that $(T\oplus (b), P)$ is critical. 
\end{proof}

\begin{proposition}\label{lemma:IIAbad}
Let $\mathcal P=(T,  P)$  be a partitioned tuple of simple type and let $c\in \pi_1^o(\mathcal O)$ correspond to a boundary component of $\mathcal O$, that is, $c$ is a maximal peripheral element of $\pi_1^o(\mathcal O)$.
Assume that  $\langle c \rangle  \cap \langle T\oplus  P\rangle =\langle c ^z\rangle$ for some $z\neq 0$ and  let $w\in\mathbb Z\setminus\{0\}$.  Then  $(T, (c^w)\oplus   P )$ is critical. 
\end{proposition}
\begin{proof}
Put $U:=\langle T\oplus P\rangle$ and let $\eta:\mathcal O'\to \mathcal O$ be the covering corresponding to $U\leq \pi_1^o(\mathcal O)$. We may  assume that there is a one-to-one correspondence between the elements in $(c^z)\oplus P$  and the compact boundary components of  $\mathcal O'$. In fact, if there is $\gamma\in P$ such that $\gamma$ and $c^z$ correspond to the same boundary component of $\mathcal O'$ then $(T, (c^w)\oplus P)$  folds peripheral elements; hence it is critical. If  there is a boundary component of $\mathcal O'$ that does not correspond to any element in  $(c^z)\oplus P$ then  $(T , P)$ is equivalent to  $(T'\oplus (c^z), P')$ and so $(T , (c^w)\oplus P)$ is equivalent to $(T'\oplus (c^z), (c^w)\oplus P')$ which  is either reducible or has an obvious relation; again  it is  critical.  Therefore $U$ is of finite index in $\pi_1^o(\mathcal O)$. Now an argument similar to the argument given in the previous lemma shows that $(T, (c^w)\oplus P)$ is critical.  
\end{proof}

\begin{lemma}{\label{lemma:foldalmostorbifol}}
Let  $\mathfrak B=(\mathcal B, u_1, (T_u)_{u\in VB}, (p_j)_{1\le j\leq n})$  be an almost orbifold cover with exceptional vertex $y\in VB$. Then    $fold(\mathfrak B)\ge 2$  unless the following hold: 
\begin{enumerate}
\item[(i)] the underlying surface of $\mathcal O$ is a disk  and $A_{[y]} \cong \mathbb Z_{2l+1}$ for some $l\ge 1$.

\item[(ii)] $val(y, B)=2$. In particular, $T_y =(s_{[y]}^2)$ and hence  $B_y=A_{[y]}$.   
\end{enumerate} 
\end{lemma}
\begin{proof} 
By definition,  $\mathfrak B$ is folded  at all vertices except  at $y$. Moreover, $B_y=\langle s_{[y]}^d\rangle$ where
$$
val([y], A)\cdot |A_{[y]}|> val([y],A) \cdot d=length \ \Gamma_{\mathfrak B}( y) > val([y], A) \cdot |A_{[y]}:B_y|.$$ 
The previous inequalities imply that the generator   $s_{[y]}^d$ of $B_y$ is not angle-minimal. Thus  there is $r\geq 2$ such that $d= r\cdot |A_{[y]}:B_y|$.  Therefore
\begin{eqnarray}
fold(\mathfrak B)& = & fold_{\mathcal B}(y)\nonumber\\
 & = & val(y, B)-l_{\mathcal B}(y) \nonumber \\
                     & = &  length \ \Gamma_{\mathfrak B}( y) - val([y], A) \cdot |A_{[y]}:B_y|\nonumber\\
                     & =& r\cdot  val([y] , A)\cdot  |A_{[y]}:B_y|-val([y], A) \cdot |A_{[y]}:B_y|\nonumber \\
                     & = & (r-1)\cdot  val_A([y]) \cdot |A_{[y]}:B_y | \nonumber 
\end{eqnarray}
Therefore $fold(\mathfrak B)\le 1$ if  and only if   $val([y], A)=1$, $r=2$ and  $|A_{[y]}:B_y|=1$.   The equality  $|A_{[y]}:B_y|=1$ implies that $d=r$ and the  equality  $val([y], A)=1$ implies that  underlying surface of $\mathcal O$  is a disk  since $val([y], A)$ coincides with the number of boundary components of $\mathcal O$. Moreover  
$$d-1 =  val([y], A)  d-val([y], A)  |A_{[y]}:B_y|=   length \ \Gamma_{\mathfrak B}( y) - val([y], A)   |A_{[y]}:B_y| \le 1 $$                    
implies that $d\le 2$ and therefore $d=2$. As $\langle s_{[y]}^2\rangle =B_y=A_{[y]}$ it follow that $A_{[y]}\cong \mathbb Z_{2l+1}$. 
\end{proof}

\begin{proposition}{\label{prop:00}}
Let $(T, P)$ be  a partitioned tuple of (almost) orbifold covering type  in  $\pi_1^o(\mathcal O)$. Then the partitioned tuples $(T\oplus (b), P) $ and  $(T, (\gamma)\oplus  P)$  
are critical  for any $b\in  \pi_1^o(\mathcal O)$ and for any labeled peripheral element $\gamma\in  \pi_1^o(\mathcal O)$.
\end{proposition}
\begin{proof}
Assume first that $(T, P)$ is of orbifold covering type.  We will show that $(T\oplus (b), P)$  is critical.  The argument to show that $(T, (\gamma)\oplus P)$ is critical is   analogous. 
 
Let $\eta:\mathcal O'\rightarrow \mathcal O$ be the covering corresponding to $U:=\langle  T \oplus  P\rangle\leq \pi_1^o(\mathcal O)$. Choose a boundary component $C\subseteq \partial \mathcal O$ and let $c\in \pi_1^o(\mathcal O)$ correspond to $C$.  We  may arrange $(T, P)$  so that   $P=P_0\oplus (\gamma_1, \ldots, \gamma_s)$  
where $\gamma_1, \ldots,   \gamma_s \in \pi_1^o(\mathcal O')$ correspond to the boundary components  of $\mathcal O'$ that cover $C$.  By definition,  there is  are  $a_i\in \pi_1^o(\mathcal O)$ and $z_i\in \mathbb Z\setminus\{0\}$  such that  $\gamma_i= a_i c^{z_i} a_i^{-1}$ for all $1\le i\le  s$. After conjugating $(T, P)$ by $a_1^{-1}$ we may assume that   $\gamma_1=c^{z_1}$.

Since maximal peripheral subgroups   of $U$  have the form $u\langle \gamma\rangle u^{-1}$   with $u\in  U$ and $\gamma\in P$, there is $u\in U$ and $j\in \{1, \ldots , s\}$ such that
$u \langle \gamma_j\rangle u^{-1}= b \langle c\rangle  b^{-1}\cap U$. Thus $b=u a_j c^z$ for some $u\in U$ and some $z\in \mathbb Z$. If $j\ge 2$ then  
\begin{eqnarray}
(T \oplus (b), P_0\oplus (\gamma_1, \ldots, \gamma_s))& \sim &   (T\oplus (a_jc^z),  P_0\oplus (\gamma_1, \ldots , \gamma_j,  \ldots, \gamma_s)) \nonumber \\
                   & \sim &    (T\oplus (a_jc^z), P_0\oplus (a_jc^z \gamma_1 c^{-z}a_j^{-1}, \ldots, \gamma_j,  \ldots, \gamma_s))\nonumber \\
                   & = &  (T\oplus (a_j c^z) , P_0\oplus (a_j c^{z_1}a_j^{-1}, \ldots a_j c^{z_j} a_j^{-1},  \ldots, \gamma_s)). \nonumber 
\end{eqnarray}
which shows that $\mathcal P$ folds peripheral elements. If $j=1$ then 
\begin{eqnarray}
(T \oplus (b), P_0\oplus (\gamma_1, \ldots, \gamma_s))& =  &   (T\oplus (u c^z), P_0\oplus (c^{z_1}, \gamma_2 \ldots ,   \ldots, \gamma_s)) \nonumber \\
& \sim & (T\oplus (c^z), P_0\oplus (c^{z_1}  , \gamma_2, \ldots, \gamma_s))\nonumber \\
& = & (T\oplus (c^{w}) , P_0\oplus ( c^{z_1} , \gamma_2,  \ldots, \gamma_s)). \nonumber 
\end{eqnarray}
where $z=qz_1+w$ with  $0\le w< z_1$. Therefore $\mathcal P$ is reducible (if $w=0$) or has an obvious relation (if $w\neq 0)$. 

Assume now that $(T, P)$ is of almost orbifold covering type. Let  $(\eta:\mathcal O'\to \mathcal O, [\mathcal P'])$ be a marking of almost orbifold covering type that represents the equivalence class of $(T, P)$. By definition, there is a cone point $x$ of $\mathcal O$ such that $$\eta|_{\mathcal O'\setminus \eta^{-1}(x)}:\mathcal O'\setminus\{\eta^{-1}(x)\}\to \mathcal O\setminus\{x\}$$ is an orbifold cover of finite degree. Therefore $U:=\langle T\oplus P\rangle$ is of finite index in $\pi_1^o(\mathcal O)$. Let $\lambda:\bar{\mathcal O}\rightarrow \mathcal O$ be the covering corresponding to  $U:=\langle T \oplus P\rangle\le \pi_1^o(\mathcal O)$ and let $\delta_1, \ldots , \delta_m\in  U$  correspond to the boundary components of $\bar{\mathcal O}$. Since $U$ is of finite index it follows that any peripheral element of $U$ is (in $U$) conjugate into $\langle\delta_i\rangle$ for some $1\le i\le m$. 
 From the fact that $\eta|_{\mathcal O'\setminus\{\eta^{-1}(x)\}}$ is of finite degree we conclude that for any  $i\in \{1, \ldots , m\}$ there is (not unique) $\gamma\in P$ such that   $\gamma$ is conjugate (in $U$) into $\langle\delta_i\rangle$.  The argument is now similar to the argument given in the previous case.
\end{proof}

\begin{proposition}{\label{prop:02}}
Let $(T, P)$ be an arbitrary partitioned tuple in $G:=\pi_1^o(\mathcal O)$ and let  $\gamma$ a labeled peripheral element of $G$. Then one of the following holds:
\begin{enumerate}
\item[(1)] $(T\oplus (\gamma), P)$ is critical.

\item[(2)] $(T\oplus (\gamma), P)$ and  $(T, (\gamma)\oplus  P)$ are non-critical and of simple type.
\end{enumerate}  
\end{proposition}
\begin{proof}
Put $\mathcal P:=(T\oplus (\gamma), P)$ and $\mathcal P':=(T, (\gamma)\oplus  P)$. Note first that the definition of equivalence implies that  if $\mathcal P'$   is equivalent to $(T',(\gamma')\oplus  P')$ then $\mathcal P$ is equivalent to $(T'\oplus (\gamma'), P')$. 

We need to consider  all possibilities for $\mathcal P'$ spelled out in Proposition~\ref{prop:01} 
 
\smallskip

\noindent \emph{(1)  $\mathcal P'$ is of simple type.}  By definition,  $\mathcal P'$ is equivalent to $(T' , (\gamma')\oplus (\gamma_1', \ldots , \gamma_n'))$  such that  
$U=\langle T'\rangle \ast   \langle \gamma'\rangle \ast \langle \gamma_1'\rangle \ast \ldots \ast \langle \gamma_n'\rangle$ and $rank(U)=size(T)+size(P)+1$
where $U:=\langle T\oplus(\gamma)\oplus P\rangle$   and $\gamma'$ correspond to $\gamma$.  We observed in the  previous paragraph   that   $\mathcal P$ is equivalent to $(T'\oplus (\gamma'), (\gamma_1', \ldots , \gamma_n'))$. Therefore $\mathcal P'$ of simple type implies that $\mathcal P$ is also of simple type.

Let $\eta:\mathcal O'\to \mathcal O$ be the cover corresponding to $U\leq G$.  Observe that $\mathcal P'$ is non-critical iff $T\oplus (\gamma)\oplus P$ is a rigid generating tuple of $\pi_1^o(core(\mathcal O'))$ which is equivalent to $\mathcal P$ be non-critical.   Therefore $\mathcal P$ and $\mathcal P'$ are either non-critical of simple type (which puts us exclusively in case (2)) or both are critical (which puts us exclusively into case (1)).

\smallskip

\noindent\emph{(2) $\mathcal P'$ is of  orbifold covering type}. Let $\eta:\mathcal O'\to \mathcal O$ be the orbifold covering corresponding to $U:=\langle T\oplus (\gamma)\oplus P\rangle \le \pi_1^o(\mathcal O)$.  If $\mathcal O'$ is a surface, then $\mathcal P$ is clearly reducible since $\gamma$ is a consequence of the remaining generators.   If $\mathcal O'$ has at least one cone point then  we the standard presentation of  $U=\pi_1^o(\mathcal O')$ reveals that (up to equivalence) the elements in $P$ determine free factors of $U$. Since $ T \oplus (\gamma)\oplus  P $ is a minimal non-rigid generating tuple of $U$ we conclude that $\mathcal P$  is equivalent to $(T', P')$ with $T'$ containing an angle-minimal element. Therefore $\mathcal P$  is critical and of simple type.

\smallskip
 
\noindent  \emph{(3) $\mathcal P'$ is of almost orbifold covering type}. Let  $\mathfrak B'=(\mathcal B', u_1', (T_{u'}')_{u'\in VB'}, (p_j')_{1\le j\le n+1})$ 
be a minimal (i.e.~$\mathcal B'= core(\mathcal B') $) almost orbifold cover that corresponds to $[\mathcal P']$ and let $y \in VB'$ be the exceptional vertex of $\mathfrak B'$. We may assume that the peripheral path  $p_{n+1}'$ corresponds to~$\gamma$. Observe that  $\mathcal P=(T\oplus (\gamma), P)$ is   represented by the tame and minimal decorated $\mathbb A$-graph 
 $\mathfrak B:=(\mathcal B, u_1 , (T_u)_{u\in VB}, (p_j)_{1\leq j\le n})$ where $(\mathcal B, u_1)=(\mathcal B', u_1')$, $T_u=T_u'$ for all $u\in VB=VB'$, and    $p_j=p_j'$ for $1\le j\le 
n$, that is, we obtain $\mathfrak B$ by   simply removing the peripheral path $p_{n+1}' $   from the defining data of $\mathfrak B'$. 

In what follows  we can assume that  all folds preserve the  rank of the fundamental groups  since otherwise $\mathcal{P}$ is either of type \textbf{(4)} or of surface covering type,  and therefore critical. There are two cases to consider depending on the foldability of $\mathfrak B$.

\noindent\emph{Case 1.}  $ fold(\mathfrak B ) =fold(\mathfrak B') \ge  2$. Since the foldability of  pre-(almost) orbifold covers is equal to one   and the foldability of   folded decorated $\mathbb A$-graphs is  equal to zero,  we conclude that,  if $\mathfrak B$ folds onto  such a decorated $\mathbb A$-graph, there must be some fold that decreases foldability. According to Lemma~\ref{lemma_foldedness1} and Lemma~\ref{lemma_foldedness2}, an element of order two must be added to some vertex tuple, which implies that $\mathcal P$ is equivalent to $(\bar{T}, \bar{P})$ with  $\bar{T}$ containing an element of order two and therefore  angle-minimal. This shows that  $\mathcal P$  is  critical.

\noindent\emph{Case 2.}  $ fold(\mathfrak B ) =fold(\mathfrak B')  =1$.   Lemma~\ref{lemma:foldalmostorbifol} applied to the almost orbifold cover  $\mathfrak B'$   implies that the following hold:
\begin{enumerate}
\item[(i)] the underlying surface of $\mathcal O$ is a disk and $A_{[y]}=\langle s_{[y ]} \ | \ s_{[y ]}^{2l+1}\rangle$.

\item[(ii)] $val(y, B)=val(y , B' )=2$ and $T_y=T_{y }' =(s_{[y ]}^2)$.
\end{enumerate}
Observe that (i) and  (ii) imply that  $B_y= B_y'= A_{[y]}$.  Put  $Star(y , B ):=\{f_1, f_2\}$. As $fold(\mathfrak B)=fold(\mathfrak B')= 1$  we conclude that  the only fold  that  can be applied to  $\mathfrak B$  is the fold, which we denote by $F$,  that identifies $f_1$ and $f_2$. 
Let $\mathfrak B''$ be the decorated $\mathbb A$-graph  that is obtained from $\mathfrak B$ by $F$.

If the foldability   increases then the argument given in case 1 shows that  $\mathcal P$ is critical.

If  $fold(\mathfrak B'')<fold(\mathfrak B)=1$ then  $fold(\mathfrak B'')=0$ and so  $\mathfrak B''$ is folded.  in particular $\mathfrak B''$ cannot be an almost orbifold cover.  Lemma~\ref{lemma_foldedness2} tells us that folds of type IA do not decrease foldability. Thus $F$    it of type IIIA; hence $z:=\omega(f_1)=\omega(f_2)\in VB$.   Lemma~\ref{lemma_foldedness1}(1) implies that the following hold:
\begin{enumerate}

\item[(i')] $A_{[z]}\cong \mathbb Z_2$; hence $\mathcal O=D^2(2, 2l+1)$.

\item[(ii')] $F$ adds $s_{[z]}^{\pm1}$ to $B_z=B_z'=1$; hence $B_z''=A_{[z]}$.
\end{enumerate}
From this we conclude  that $\mathcal P$ if $\mathcal P$ is of simple type then it is critical since $\mathfrak B''$  contains a  tuple with an angle-minimal element. 

Finally we  deal with the case that $F$ preserves foldability.  Lemma~\ref{lemma_foldedness2}  combined with the fact that   $\mathcal O=D^2(m_1, m_2)$ with $m_2=2l+1$,   implies that   $F$ is of type IIIA, that is,  $z:=\omega(f_1)=\omega(f_2)\in VB=VB'$.  Lemma~\ref{lemma_foldedness1} implies that  $A_{[z]}\cong \mathbb Z_4 $ and $F$ adds $s_{[z]}^2$ to  $B_z''$ which is angle-minimal.   From this we conclude,    that $\mathcal P$ cannot be of almost orbifold cover and if $\mathcal P$ is of simple type then it is critical.   Both claims follow  from the fact that  any decorated $\mathbb A$-graph that is obtained from $\mathfrak B''$ by a fold will contain a  vertex tuple with an angle-minimal element.

\smallskip

\noindent\emph{(4) $\mathcal P'$ is reducible.} Then $\mathcal P$ is reducible and therefore critical.

\medskip

\noindent\emph{(5) $\mathcal P'$  has an obvious relation}. Thus $\mathcal P'=(T, (\gamma)\oplus P)$ is equivalent to a partitioned tuple  $(T'', (\gamma'')\oplus P'' )$ such that 
$$z'':=|o''\langle c_{i''}\rangle (o'')^{-1}: U''\cap o''\langle c_{i''}\rangle (o'')^{-1}|< | o''\langle c_{i''}\rangle (o'')^{-1}:  \langle \gamma''\rangle|$$
where $U''= \langle T''\oplus P''\rangle$ and $(o'', i'')$ is the label of $\gamma''$. Let $\gamma'\in (\gamma'')\oplus P''$  correspond to $\gamma$, that is, $\gamma$ is carried by the elementary equivalences onto  $\gamma'$

If   $\gamma'\in P''$  then the previous inequality shows that $\mathcal P$ has an obvious relation. Thus assume $\gamma''=\gamma'$, that is, $\gamma''$ corresponds to $\gamma$. We can also assume that  $\mathcal P'':=(T'', P'')$  is of simple type since all other possibilities for $\mathcal P''$  imply that $\mathcal P$ is critical. We can  further assume that the elements of $P''$  correspond to all but one boundary component of the orbifold   corresponding to $U''$ as otherwise we can assume that $T''$  contains    $o''c_{i''}^{z''}(o'')^{-1}$     and so  $\mathcal P$ is  reducible.  These assumptions on $\mathcal P''$  imply that $ |\pi_1^o(\mathcal O):U''|<\infty.$   It now follows from Lemma~\ref{lem:critical1} that $(T\oplus (\gamma), P)$   is critical.

\medskip

\noindent\emph{(6) $\mathcal P'$ folds peripheral elements}.  By definition $\mathcal P'$ is equivalent to $$(T'', (\gamma_1'', \gamma_2'')\oplus P'')$$  such that the labels of $\gamma_{1}''$ and $\gamma_2''$ satisfy  $i:=i_{\gamma_1''}=i_{\gamma_2''}  $ and $o_{\gamma_1''}=o_{\gamma_2''}c_{i}^z$ for some $z\in \mathbb Z$. If none of $\gamma_1'', \gamma_2''$ correspond to $\gamma$,  then $P''=(\gamma'')\oplus P_0''$ with $\gamma''$ corresponding to $\gamma$. Thus $\mathcal P$ is equivalent to 
$$(T''\oplus (\gamma''), (\gamma_1'', \gamma_2'')\oplus P_0'')$$
 which shows that $\mathcal P$ folds peripheral elements. Thus assume that $\gamma_1''$ corresponds to $\gamma$. Thus  $\mathcal P$   is equivalent to 
$$( T''\oplus (\gamma_1''), (\gamma_2'')\oplus P'')$$
which in turn is equivalent to
$$(T''\oplus (\gamma''),  (\gamma_2'')\oplus P'')$$
with $\gamma''\in \langle \gamma_1'', \gamma_2''\rangle$ such that  $\gamma''=1$ or 
$$|o_{\gamma_1''} \langle c_{i_{\gamma_1''}} \rangle o_{\gamma_1''}^{-1}: \langle \gamma'' \rangle |< |o_{\gamma_1''} \langle c_{i_{\gamma_1''}}\rangle o_{\gamma_1''}^{-1}: \langle \gamma_1'' \rangle |.$$
Therefore $\mathcal P$ is either reducible or has an obvious relation. 
\end{proof}

\begin{corollary}{\label{lemma:sumtuples}}
Let $(T, P)$  and $(T_0, P_0)$  be a partitioned tuples in $\pi_1^o(\mathcal O)$.  If $(T, P)$ is critical  or of (almost) orbifold covering type,  then  $(T\oplus T_0, P\oplus P_0)$ is critical.
\end{corollary}
\begin{proof}
If $(T, P)$ is of (almost) orbifold covering type then the result follows from Proposition~\ref{prop:00}.   

If $(T, P)$ is of type \textbf{(4)} then the result is easily verified. 

Finally if $(T, P)$ is of simple type then $(T, P)$ is equivalent to $(\bar{T}, \bar{P})$  such that $\bar{T}$ contains an angle-minimal element.  The result now follows as  $(T\oplus T_0, P\oplus P_0)$ is equivalent to $(\bar{T}\oplus T_0 , \bar{P}\oplus P_0)$.
\end{proof}

%------------------------------------------------------------------

\section{The global picture}{\label{section_global_picture}}

In this section we will recall some definitions contained in Section~4 of \cite{Dut}. We  nevertheless assume that the reader is  familiar with the language developed in \cite{Dut}.

Let $\mathcal O$ be a sufficiently large orbifold. Then there  is  a non-empty (not unique) collection  $\gamma_1,\ldots, \gamma_t$  of  pairwise disjoint simple closed curves on $\mathcal O$ such that the closure of  each component of  $\mathcal O- \gamma_1\cup \ldots \cup \gamma_t$  is a small orbifold. Let $\mathbb A':=\mathbb A'(\mathcal O, \gamma_1,  \ldots,  \gamma_t)$  be the corresponding graph of groups  and let  $v_0\in VA'$. The orbifold version of the Seifert-Van Kampen theorem implies that  $\pi_1(\mathbb A',v_0)\cong \pi_1^o(\mathcal O)$.    For technical reasons we replace $\mathbb A'$ by  its barrycentric subdivision $\mathbb A$. Thus for each edge $e$ of $\mathbb A'$ there is a vertex $v_e$ with infinite cyclic group corresponding to $\gamma_i$ for some $i$. The vertices of $\mathbb A$ arising from edges of $\mathbb A'$ are called \emph{peripheral} and the remaining vertices will be called \emph{non-peripheral}.

We use the  same terminology for  $\mathbb A$-graphs, i.e.~we say that a vertex $u$ of an $\mathbb A$-graph  $\mathcal B$ is \emph{peripheral} (resp.~\emph{non-peripheral}) if  the corresponding vertex  $[u]\in VA$ is  peripheral (resp. non-peripheral).  The set of peripheral (resp.~non-peripheral) vertices   is denoted by $V_pB$ (resp.~$V_{np}B$). 

\smallskip

Let $\mathcal B$ be an $\mathbb A$-graph  and  $f\in EB$   labeled $(a,e,b)$ such that $x:=\alpha(f)$ is non-peripheral and  $B_f\neq 1$.  Then any generator of  the cyclic group $\alpha_f(B_f)=a\alpha_e(B_f)a^{-1}\leq B_x$ 
is called a  \emph{peripheral element associated to $f$}. Any peripheral  element associated to $f$ will be  denoted by $\gamma_f$.
\emph{Throughout we will always assume that any such  $\gamma_f$  is labeled  $(a,e)\in A_{[x]}\times Star([x], A)$.} For more details see   \cite[Subsection~4.1]{Dut}.

%-----------------------------------------------------------------------------

\subsection{Marked $\mathbb A$-graphs}
In this section we define marked $\mathbb A$-graphs.  Before we need to recall the notion of $\mathbb A$-graph of orbifold type defined in Subsection~4.2 of \cite{Dut}.

\begin{definition}{\label{def:orbtype}} We say that a finite and  folded $\mathbb A$-graph $\mathcal C$ is of \emph{orbifold type} if:
\begin{enumerate}
\item  $\mathcal C$ is non-empty (possibly consisting of a single   vertex).

\item  $|A_{[x]}:C_x|<\infty$ for all $x\in VC\cup EC$.

\item $\mathcal C$ is locally surjective at non-peripheral vertices.

\item  there is at least one vertex at which  $\mathcal{C}$ is not locally surjective.
\end{enumerate}
An $\mathbb A$-graph of orbifold type consisting of a single vertex is called \emph{degenerate}. Otherwise we say that  it is \emph{non-degenerate}.
\end{definition}

\begin{definition}
Let $\mathcal C$ be an $\mathbb A$-graph of orbifold   type.   A \emph{boundary vertex  of $\mathcal C$}  is a vertex at which  $\mathcal{C}$ is not locally surjective.
\end{definition}

To any $\mathbb A$-graph of orbifold type $\mathcal C$ there corresponds a unique  compact orbifold $\mathcal O_\mathcal C$  with non-empty boundary which is  defined as the union of the finite covers of   the vertex orbifolds $\mathcal O_{[u]}$ corresponding to the subgroups $C_u\leq A_{[u]}=\pi_1^o(\mathcal O_{[u]})$ as $u$ ranges over the vertex set of $\mathcal C$.   The fundamental group of the graph of groups associated to $\mathcal C$ is canonically isomorphic (via Seifert-van-Kampen Theorem) to the fundamental group of $\mathcal O_\mathcal C$.   We  say that  $\mathcal O_\mathcal C$  is \emph{the orbifold associated to  $\mathcal C$}. 

Observe that $\mathcal O_\mathcal C$ is  a simple closed curve  iff $\mathcal C$ is degenerate. If $\mathcal C$ is non-degenerate, then there is a one-to-one correspondence between the boundary vertices of $\mathcal C$ and the boundary components of  $\mathcal O_\mathcal C$.

\begin{definition}
A \emph{marked $\mathbb A$-graph  of orbifold type}  is a triple $(\mathcal C, u_{\mathcal C}, T_{\mathcal C})$   where $\mathcal C$ is an $\mathbb A$-graph of orbifold type, $u_{\mathcal C}$ is a vertex of $\mathcal C$ and $T_{\mathcal C}$ is  a rigid generating tuple of $\pi_1(\mathbb C, u_{\mathcal C})$ meaning that   the corresponding generating tuple of $\pi_1^o(\mathcal O_\mathcal C)\cong \pi_1(\mathbb C, u_{\mathcal C})$  is rigid. 
\end{definition}

Let  $\mathcal B$ be  an $\mathbb A$-graph and let $\mathscr C$ be  a collection of marked $\mathbb A$-graphs  of orbifold type such that the following hold:
\begin{enumerate}
\item[\textbf{(A.1)}] for each  $(\mathcal C, u_{\mathcal C}, T_{\mathcal C})\in \mathscr C$, the $\mathbb A$-graph  $\mathcal C$  is a sub-$\mathbb A$-graph of $\mathcal B$.

\item[\textbf{(A.2)}]  the members of $\mathscr C$ are pairwise disjoint. 
\end{enumerate}
The set of \emph{vertices of orbifold type   with respect to $\mathscr C$} and  the set of \emph{edges of orbifold type    with respect to $\mathscr C$} are defined as
$$V_{orb}^{\mathscr{C}}B:=\bigcup_{(\mathcal C, u_{\mathcal C},  T_{\mathcal C})\in \mathscr{C}} VC  \ \ \text{ and } \  \  E_{orb}^{\mathscr C}B:=\bigcup_{(\mathcal C, u_{\mathcal C},  T_{\mathcal C})\in \mathscr{C}} EC.$$
The set of \emph{exceptional vertices  with respect to $\mathscr C$}, and the set of \emph{exceptional edges with respect to $\mathscr C$}  are  defined   as 
$$V_{exc}^{\mathscr C}B =V_{np}B\setminus  V_{orb}^{\mathscr C} B\  \  \text{ and } \  \  E_{exc}^\mathscr CB:= EB \setminus E_{orb}^{\mathscr C}B.$$

\begin{definition}{\label{def:marked}} Let $\mathcal B$ be an $\mathbb A$-graph.  A \emph{marking} of $\mathcal B$  consists of the following:
\begin{enumerate}
\item[(I)] a finite collection $\mathscr C$   of   marked $\mathbb A$-graphs  of orbifold type  satisfying conditions  (A.1) and (A.2).

\item[(II)]  a partitioned tuple $\mathcal P_u=(T_u, P_u)$ in  $A_{[u]}=\pi_1^o(\mathcal O_{[u]})$ for each $u\in V B$  that is exceptional with respect to $\mathscr C$, that is,~for each $u\in V_{exc}^{\mathscr C}B$.   
\end{enumerate}
such that the following hold:
\begin{enumerate}
\item  for any  $u\in V_pB$ with    $B_u\neq 1$  there is  $(\mathcal C, u_{\mathcal C}, T_{\mathcal C})\in \mathscr C$ such that $u\in VC$.

\item $P_u=(\gamma_{f_1}, \ldots, \gamma_{f_n})$ where $\{f_1, \ldots, f_n\}=Star(u, B)\cap E(\mathcal B)$ where $  E(\mathcal B)$ denotes the set $\{f\in EB  \ |  \ B_f\neq 1\}.$   

\item $T_u\oplus P_u$ generates $B_u\leq A_{[u]}$.

\item for any $f\in E_{exc}^{\mathscr C}B\cap E(\mathcal B)$  with  $\alpha(f)\in V_pB$  the following hold:\begin{enumerate}

\item  there is $(\mathcal C, u_{\mathcal C}, T_{\mathcal C})\in \mathscr C$ such that  $\alpha(f)$ is a boundary vertex of $\mathcal C$.

\item  $\omega(f)$ is an  exceptional vertex  with respect to $\mathscr C$, i.e.~$\omega(f)\in V_{exc}^{\mathscr C}B$.

\item  $B_f = \alpha_{[f]}^{-1}(o_f^{-1}B_{\alpha(f)}o_f)= \alpha_{[f]}^{-1}(B_{\alpha(f)}).$
\end{enumerate}
\item for any  $f\neq g\in E(\mathcal B)$ such that   $\alpha(f)=\alpha(g)\in V_pB$  it holds $[f]\neq [g]$. 
\end{enumerate}
\end{definition}

\begin{definition}{\label{def:marked1}}
A \emph{marked $\mathbb A$-graph} is a tuple  
 $\mathbf B=(\mathcal B,   \mathscr C,(\mathcal P_u)_{u\in V_{exc}^{\mathscr C}B})$ 
where:
\begin{enumerate}
\item $\mathcal B$ is $\pi_1$-surjective, i.e.~the homomorphism  $\phi_{\mathcal B}:\pi_1(\mathbb B, u_0)\rightarrow \pi_1(\mathbb A, v_0)$ 
associated to $\mathcal B$ is surjecvitve for some (and therefore any) $u_0\in VB$ such that $[u_0]=v_0$.

\item $(\mathscr C,(\mathcal P_u)_{u\in V_{exc}^{\mathscr C}B})$  is a marking of $\mathcal B$. 
\end{enumerate} 
The $\mathbb A$-graph  $\mathcal B$ (resp.~the underlying  graph $B$ of $\mathcal B$) is called the underlying $\mathbb A$-graph (resp.~underlying graph) of $\mathbf B$. 
\end{definition}

\begin{definition}{\label{def:almostorb}}
A marked $\mathbb A$-graph $\mathbf B=(\mathcal B,  \mathscr C,(\mathcal P_u)_{u\in V_{exc}^{\mathscr C}B})$ 
is a \emph{special almost  orbifold covering with a good marking} if the following hold:
\begin{enumerate}
\item  $V_{exc}^\mathscr CB  =\{u\}$ and $\mathcal P_u$ is of almost orbifold covering type.

\item  $E_{exc}^\mathscr C B =Star(u, B)^{\pm 1}$   and $B_{f}\neq 1$ for all  $f\in  Star(u, B)$.

\item  For each   $(\mathcal C, u_{\mathcal C}, T_{\mathcal C})\in \mathscr C$   with $\mathcal C$ is non-degenerate and  each  boundary vertex of $u'$  of $\mathcal C$   there is a  unique $f\in Star(u, B)$ such that $\omega(f)=u'$.

\item For each   $(\mathcal C, u_{\mathcal C}, T_{\mathcal C})\in \mathscr C$     such that $\mathcal C$  is degenerate  there are  unique $f\neq f'\in Star(u, B)$  such that   $\omega(f)=\omega(f')=u'$.
\end{enumerate}  
\end{definition}

\begin{definition} We say that a  marked $\mathbb A$-graph $\mathbf B=(\mathcal B,  \mathscr C,(T_u)_{u\in V_{exc}^{\mathscr C}B})$  is \emph{tame} if $\mathcal P_u$ is non-critical and of simple  type for any $ u\in V_{exc}^\mathscr C B$.
\end{definition}

\begin{remark} 
Observe that a special almost orbifold covering with a good marking  is not tame since it has an exceptional vertex   such that the corresponding partitioned tuple is not of simple type. 
\end{remark}

%---------------------------------------------------------------------------

\noindent{\textbf{The tuple associated to a marked $\mathbb A$-graph}.} In  \cite[Section 4.4]{Dut}  we explained  how  a marked $\mathbb A$-graph yields a Nielsen equivalence class of generating tuples of   $\pi_1^o(\mathcal O)\cong \pi_1(\mathbb A, v_0)$. We recall the construction.  Assume that 
 $\mathbf B=(\mathcal B,  \mathscr C,(\mathcal P_u)_{u\in V_{exc}^{\mathscr C}B})$ 
is a (not necessarily tame) marked $\mathbb A$-graph. Let $Y\subseteq B$ be a maximal $\mathscr C$-subtree (that is,  $Y$ is a maximal subtree of $B$  such that $Y\cap C$ is a maximal subtree of $C$ for each   $(\mathcal C, u_{\mathcal C}, T_{\mathcal C})\in  \mathscr C$). Let $E$ be a subset of $EB$ such that $E$ contains either $e$ or $e^{-1}$ for every pair $\{e, e^{-1}\}\subseteq E_{exc}^\mathscr C B\setminus EY$.  Let $u_0\in VB$ with $[u_0]=v_0\in VA$ (the base vertex of $\mathbb A$).  Finally,  for any $u\in VB$ let 
$$p_u:=1, e_{u,1}, 1, \ldots, 1, e_{u, d_u}, 1$$
where $e_{u,1}, \ldots,e_{u, d_u}$ is the unique reduced path in $Y$ from  $u_0$ to $u$.  
We define $T_{Y, u_0}^\mathbf B$ as the tuple consisting of the following elements:
\begin{enumerate}
\item  for every $e\in E $, the element $g_e:=[p_{\alpha(e)}\cdot 1, e, 1\cdot  p_{\omega(e)}^{-1}]$. 

\item for each   $u\in V_{exc}^{\mathscr C}B$ and each $b$ in  $T_u$, the element $[p_{u}\cdot b\cdot p_{u}^{-1} ]$.

\item  for each  $(\mathcal C, u_{\mathcal C}, T_{\mathcal C})\in \mathscr C$ and   $b$ in  $T_\mathcal C$,   the element $[p_{u_{\mathcal{C}}}\cdot b \cdot  p_{u_{\mathcal C}}^{-1}]$. 
\end{enumerate}

\begin{remark}
Note that for each   $u\in V_{exc}^\mathscr C B$ there is a subtuple $T_{Y, u_0}^u$  of $T_{Y, u_0}^{\mathbf B}$ that corresponds to   $T_u$ such that  $\langle T_{Y, u_0}^u\rangle \le \pi_1(\mathbb B, u_0)$  is isomorphic to  $\langle T_{u}\rangle \leq B_u$. Similarly, for each $(\mathcal C, u_{\mathcal C}, T_{\mathcal C})  \in \mathscr C$ there is a subtuple $T_{Y, u_0}^\mathcal C$ of $T_{Y, u_0}^{ \mathbf B }$ that corresponds to $T_\mathcal C$ such that  $\langle T_{Y, u_0}^{\mathcal C}\rangle \le \pi_1(\mathbb B, u_0)$  is isomorphic to $\pi_1(\mathbb C, u_{\mathcal C})$.
\end{remark}

It is a consequence of  \cite[Proposition 2.4]{KMW} that $T_{Y, u_0}^\mathbf B$ is a generating tuple of  $\pi_1(\mathbb B, u_0)$. It further follows from  \cite[Lemma~4.16]{Dut} that  for a fixed base vertex $u_0$ the Nielsen equivalence class determined by $T_{Y, u_0}^\mathbf B$ does not depend on the choice of the maximal $\mathscr C$-subtree $Y$. When we replace $u_0$ by $u_0'\in VB$ with $[u_0']=v_0$  we obtain a tuple that is  conjugate to $\phi_{\mathcal B}(T_{Y, u_0}^{\mathbf B})$. Since    conjugate generating tuples are Nielsen equivalent and since     
$$\pi_1(\mathbb A, v_0)=\phi_{\mathcal B}(\pi_1(\mathbb B, u_0))= \phi_{\mathcal B}(\langle  T_{Y, u_0}^{\mathbf B} \rangle )= \langle\phi_{\mathcal B}(T_{Y, u_0}^{\mathbf B})\rangle$$ we conclude that the Nielsen equivalence class of generating tuples of $\pi_1(\mathbb A, v_0)$  determined by $\phi_{\mathcal B}(T_{Y, u_0}^\mathbf B)$ does not depend on $Y$ and on $u_0$. Thus the following definition makes sense.

\begin{definition}
Let $\mathbf B$ be a marked $\mathbb A$-graph. We define $[T_\mathbf B]$ as the Nielsen equivalence class  determined by  $\phi_\mathcal B(T_{Y, u_0}^\mathbf B)$   where $Y$ is an arbitrary maximal $\mathscr C$-subtree of $B$ and $u_0$ is an arbitrary vertex of type $v_0$.
\end{definition}
 
%----------------------------------------------------------------------------

\noindent\textbf{Equivalence of marked $\mathbb A$-graphs.} In this section we define equivalence of marked $\mathbb A$-graphs.  Let  $\mathbf B=(\mathcal B,   \mathscr C,(\mathcal P_u)_{u\in V_{exc}^{\mathscr C}B})$  be a (not necessarily tame) marked $\mathbb A$-graph. An elementary move on $\mathbf B$ (see \cite[Section~4.5]{Dut})  consists of one of the following modifications: 
\begin{enumerate}
\item[(1)] A \emph{Nielsen move}, i.e. 
\begin{enumerate}
\item[(i)] for some   $(\mathcal C, u_\mathcal C, T_\mathcal C)\in\mathscr C$  replace  $T_{\mathcal C}$   by a  Nielsen equivalent tuple $T_{\mathcal C}'$. 

\item[(ii)] for some  vertex $u\in V_{exc}^{\mathscr C}B$  replace $T_u$ in  $\mathcal P_u=(T_u, P_u)$ by a Nielsen equivalent tuple $T_u'$.
\end{enumerate}

\item[(2)] A \emph{peripheral move of type}  (i), i.e.~for some $u\in V_{exc}^{\mathscr C}B$ replace   $P_u=(\gamma_1,\ldots, \gamma_n)$  in $\mathcal P_u=(T_u, P_u)$ by      
 $P_u' =(\gamma_{\sigma(1)}^{\varepsilon_1}, \ldots, \gamma_{\sigma(n)}^{\varepsilon_n})$  where  $\sigma \in S_n$ and $\varepsilon_1, \ldots, \varepsilon_n\in \{\pm 1\}$.

\item[(3)] A tame auxiliary move of type A2.

\item[(4)]  A \emph{peripheral move of type} (ii):~for some $u\in V_{exc}^{\mathscr C}B$ replace  $T_u=(g_1,\ldots, g_m)$  
in $\mathcal P_u=(T_u, P_u)$  by 
 $T_u'=(g_{1}, \ldots, g_{i-1} , \gamma g_{i}\gamma', g_{i+1},  \ldots, g_{m})$ 
where  $1\le i\le m$ and $\gamma, \gamma'\in P_{u}^{\pm 1}$.

\item[(5)] An auxiliary move of type A0 or of type A1.
\end{enumerate}
\begin{definition}
We say that two marked $\mathbb A$-graphs  $\mathbf B$ and $\mathbf B'$ are \emph{equivalent}, and write $\mathbf B\approx \mathbf B'$, if $\mathbf B'$ is obtained from $\mathbf B$ by a finite sequence of elementary moves. The equivalence class of  $\mathbf B$ is denoted by $\mathbf b$.
\end{definition}

It follows from \cite[Lemma~4.18]{Dut} that  equivalent marked $\mathbb A$-graphs   define Nielsen equivalent generating tuples of $\pi_1(\mathbb A, v_0)$. Thus it makes sense to  define
 $[T_\mathbf b]:=[T_\mathbf B]$  where $\mathbf B$  is an arbitrary  representative  of $\mathbf b$.

%--------------------------------------------------------------------------

\subsection{Tame elementary folds}{\label{sec:descriptionfolds}}
In this section we discuss tame elementary folds that can be applied to a \emph{tame  marked $\mathbb A$-graph} $\mathbf B=(\mathcal B,  \mathscr C,(\mathcal P_u)_{u\in V_{exc}^{\mathscr C}B})$. Such a fold yields an $\mathbb A$-graph $\mathcal B'$ (as any fold does) that does not always inherit a marking from $\mathbf B$. If it does, then the obtained marked $\mathbb A$-graph  $\mathbf B'=(\mathcal B',   \mathscr C',(\mathcal P_u')_{u\in {V}_{exc}^{\mathscr C'}B'})$ 
represents the same Nielsen class as $\mathbf B$, i.e.~$[T_{\mathbf B'}]=[T_{\mathbf B}]$,  and has lower complexity. We do not allow all folds that are applicable to the underlying $\mathbb A$-graph~$\mathcal B$. 
 
We first discuss folds of type IIA. They can be applied to an edge $f$ if it satisfies conditions (II.1)-(II.3) introduced below. These  conditions imply in particular that we do not apply such a fold to some edge if it  can be folded by a fold of type IA or IIIA onto an edge with non-trivial edge group. The following lemma clarifies when such an edge $f$ exists.

\begin{lemma}\label{lemmatypesIIA}
Let 
 $\mathbf B=(\mathcal B,   \mathscr C,(\mathcal P_u)_{u\in V_{exc}^{\mathscr C}B})$ 
be a tame marked $\mathbb A$-graph.   Let $f$ be an edge of $\mathcal B$   labeled $(a, e, b)$ and  initial vertex $x:=\alpha(f)$. Suppose that the  following hold:
\begin{enumerate}
\item[(II.1)] $B_f=1$.

\item[(II.2)] $B_x\cap a\alpha_e(A_e)a^{-1}\leq A_{[x]}$  is non-trivial. 

\item[(II.3)] $f$ cannot be folded with any edge in  $Star(x, B)\cap E(\mathcal B)$.
\end{enumerate}
Then one of the following holds:
\begin{enumerate}
\item[(El.0)] $x$ is a boundary vertex of $\mathcal C$ for some $(\mathcal C, u_\mathcal C, T_\mathcal C) \in\mathscr C$. 

\item[(El.1)] $x$ is exceptional  (and hence $\mathcal P_x$ is non-critical of simple type) and  $\mathcal P_x$ is equivalent to 
 $(T' \oplus (\gamma'), (\gamma_1', \ldots, \gamma_n'))$  such that $$B_x=\langle T'\rangle \ast  \langle \gamma' \rangle\ast \langle\gamma_1'\rangle \ast \ldots \ast \langle \gamma_n'\rangle$$ 
where  $\gamma'$ is conjugate to a generator of  $B_x \cap a\alpha_e(A_e) a^{-1}$.  In particular, 
 $$(T, (\gamma')\oplus (\gamma_1', \ldots, \gamma_n'))$$  
is non-critical and  of simple type. 

\item[(El.2)] $x$ is exceptional, 
 $B_x=\langle T_x\oplus P_x\rangle\le A_{[x]}$  
is of finite index, and the  elements in $P_x$ correspond to all but one boundary component of the orbifold corresponding to $B_x$. 
\end{enumerate}
\end{lemma}
\begin{proof}
Assume first that $x$ is of orbifold type (which is the case  if $x$ is peripheral since (II.2) implies  that $B_x\neq 1$). Thus there is $(\mathcal C , u_{\mathcal C}, T_{\mathcal C})\in \mathscr C$ such that $x\in VC$. By definition  $\mathcal C$ is locally surjective at all interior vertices. It therefore  follows from (II.3) that $x$ is a boundary vertex of $\mathcal C$, and so (El.0) holds.

\smallskip 
 
Assume now that $x$ is exceptional. The tameness of $\mathbf B$ implies that   $\mathcal P_x$ is non-critical and of  simple type. Let $\eta_x:\mathcal O_{B_x}\rightarrow \mathcal O_{[x]}$ be the orbifold covering  corresponding to  $B_x\le A_{[x]}=\pi_1^o(\mathcal O_{[x]})$ and let   $\gamma_f'\in A_{[x]}$ such that  $\langle \gamma_f'\rangle=B_x\cap a\alpha_e(A_e)a^{-1}\leq A_{[x]}.$   Assume that  $\{f_1, \ldots, f_n\}=Star(x, B)\cap E(\mathcal B)$. Thus  $P_x=(\gamma_{f_1},\ldots ,\gamma_{f_n})$.  

If there is $1\le i\le n$ such that $\gamma_f'$ and $\gamma_{f_i}$ correspond  to  the same  boundary component of $\mathcal O_{B_x}$,  then $f$ can be folded with $f_i$ and so we obtain a contradiction to (II.3). Thus we may assume that this is not the case. If $\mathcal O_{B_x}$ has at least $k+2$ (not necessarily compact) boundary components then 
$$B_x=B_0\ast   \langle \gamma'\rangle\ast \langle \gamma_1'\rangle \ast \ldots \ast \langle \gamma_k'\rangle $$
where $\gamma_i'$ (resp.~$\gamma'$) correspond to the same boundary component of $\mathcal O_{B_x}$ as $\gamma_{f_i}$ (resp.~$\gamma_f'$). Now an  argument similar as the one given in the proof of Lemma~\ref{lemma:equivalencesimpletype} shows that (El.1) holds. If $\mathcal O_{B_x}$ has $k+1$ boundary components  then  clearly (El.2) holds.
\end{proof}

\begin{remark}{\label{remark:noncritical}} Observe that when (El.2) holds, then Lemma~\ref{allprimitive} together with the fact that $\mathcal P_x$ is non-critical of simple type implies that  $T_x\oplus P_x$ is a rigid generating tuple of $B_x\cong \pi_1^o(\mathcal O_{B_x})$.
\end{remark}
 
\begin{remark}\label{remark_preprocessing} When  (El.1) holds, then possibly  after applying auxiliary moves to $\mathbf B$ that only affect the vertex $x$ and the edges staring at $x$,  we  may assume that $$\mathcal P_x=(T' \oplus ( \gamma'), (\gamma_1', \ldots, \gamma_n')).$$ This follows from~\cite[Lemma~4.17]{Dut}.
\end{remark}

The previous remark motivates the following definition.
\begin{definition}\label{Def_normalized}
Let $\mathbf B$, $f$ and $x$ be as in Lemma~\ref{lemmatypesIIA}. 
We say that $f$ is \emph{normalized} if (El.0) or (El.2) holds or if (El.1) holds and
$$\mathcal P_x=(T_x'\oplus(\gamma_f'), (\gamma_1', \ldots, \gamma_n')) \  \ \ \text{ and } \  \ \ B_x=\langle T_x'\rangle \ast \langle \gamma_f'\rangle \ast \langle \gamma_1'\rangle \ast \ldots \ast \langle\gamma_n'\rangle$$ where $\gamma'_f$ is a generator of $B_x \cap a\alpha_e(A_e) a^{-1}$.
\end{definition} 

\smallskip

\noindent{\bf\em{Tame elementary folds of type IIA.}} Let $f$ be a \emph{normalized} edge of $\mathbf B$ labeled $(a, e, b)$ with $x:=\alpha(f)$ and $y:=\omega(f)$ and let $\gamma'_f$ be a generator of  $B_x \cap a\alpha_e(A_e) a^{-1}\le A_{[x]}.$ Let $\mathcal B'$ be the $\mathbb A$-graph that is obtained from $\mathcal B$ in the following way:
\begin{enumerate}
\item Replace $B_f=1$ by  $B_f'\le A_{e}$  such that $\langle \gamma_f'\rangle=a\alpha_e(B_f')a^{-1}.$ 

\item Replace  $B_y$ by  $B_y'=\langle B_y, \gamma'\rangle\le A_{[y]}$ where  $\langle \gamma'\rangle =b^{ -1}\omega_e(B_f')b.$ 
\end{enumerate} 
\begin{figure}[h!]
\begin{center}
\includegraphics[scale=1]{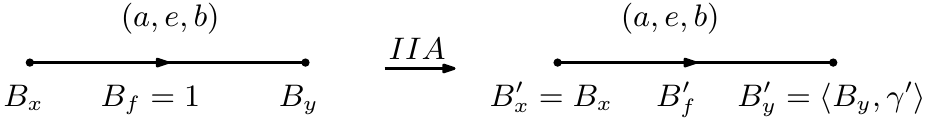}
\caption{A  tame elementary fold of type IIA based on $f\in EB$.}{\label{fig:tamefoldIIA}}
\end{center}
\end{figure}
We say that $\mathcal B'$ is obtained from $\mathcal B$ by a \emph{tame elementary fold of type IIA based on $f$}. Observe that $\mathcal B'$ contains $\mathcal C$ for each  $(\mathcal C, u_{\mathcal C}, T_{\mathcal C})\in \mathscr C$. 

\smallskip

Next we discuss when $\mathcal B'$  inherits a marking from  $\mathbf B$ yielding a marked $\mathbb A$-graph $\mathbf B'$. We also define the notion of \emph{good} and \emph{bad} tame elementary folds. Roughly speaking bad tame elementary folds will be those that reveal reducibility of the tuple associated to $\mathbf B$.

We discuss the each of the  cases (El.0), (El.1) and (El.2) separately.

\smallskip

\noindent {\emph{Fold IIA(El.0)}:} Suppose that (El.0) holds, i.e.~there is $(\mathcal C, u_\mathcal C, T_\mathcal C) \in\mathscr C$ such that  $x$ is a boundary vertex of $\mathcal C$. The fact that $f$ is normalized (more precisely, condition (II.3) of normalized edges) and Definition~\ref{def:marked}  imply that 
 $B_g=1$ for all    $g\in Star(x, B)\setminus Star(x, C).$   
We distinguish two cases depending on the type of $y=\omega(f)$.
\begin{enumerate}
\item If $y$ is of orbifold type,   then   we say that the tame elementary fold  based on $f$ is \emph{bad} and we define no marking of $\mathcal B'$. 

\item If  $y$ is exceptional,  then the tameness of $\mathbf B$ implies that $\mathcal P_y=(T_y, P_y)$ is non-critical and   of simple type.   We define 
 $\mathbf B'=(\mathcal B', \mathscr C, (\mathcal P_u')_{u\in V_{exc}^{\mathscr C}B'}) $
where $\mathcal P_u'=\mathcal P_u$ if $u\neq y$ and 
 $\mathcal P_y'=(T_y, (\gamma')\oplus P_y).$   
To decide if the  fold is good or bad we  consider all  possibilities for $\mathcal P_y'$.
\begin{enumerate}
\item If  $\mathcal P_y'$  is non-critical and of simple type, then $\mathbf B'$ is  tame. In this case we say that the fold  is \emph{good}.  

\item If $\mathcal P_y'$ is of almost orbifold covering type then let $\mathbf B''$ be the marked $\mathbb A$-graph that is obtained from $\mathbf B'$ by successively removing valence one vertices with trivial group and their adjacent edges.   One of the following occurs:
\begin{enumerate}
\item  $\mathbf B''$ is  a special almost orbifold covering with a good marking. Then  we say that the  fold is \emph{good}.

\item $\mathbf B'' $ is not a special almost orbifold covering with a good marking. Then we  say that the fold is \emph{bad}.
\end{enumerate}
\item If $\mathcal P_y'$  is critical   then we say that the fold is \emph{bad}.
\end{enumerate}
\end{enumerate}

\smallskip 

\noindent {\emph{Fold IIA(El.1):}} Suppose that  (El.1) holds, i.e.~$\mathcal P_x= (T_x'\oplus (\gamma_f'),  (\gamma_1', \ldots, \gamma_n'))$ and  $$B_x=\langle T_x'\rangle  \ast \langle \gamma_f'\rangle \ast \langle \gamma_1'\rangle \ast \ldots\ast \langle \gamma_n'\rangle$$
where $\gamma_f'$ denotes a generator of  $B_x\cap a\alpha_e(A_e)a^{-1}= a\alpha_e(B_f')a^{-1} \leq B_x$. We consider two cases depending on $B_y\leq A_{[y]}$. 
\begin{enumerate}
\item If $B_y\neq 1$, then we say that the  fold is \emph{bad} and we define no marking of $\mathcal B'$.
 
\item  If  $B_y=1$, then we define a marking of $\mathcal B'$ as follows. First define 
$$\mathscr C':=\mathscr C\cup \{(\mathcal C', u_{\mathcal C'}, T_{\mathcal C'})\}$$ 
where  $\mathcal C'\subseteq \mathcal B'$ is the degenerate  $\mathbb A$-graph of  orbifold type consisting of the single vertex $y$ with rigid generating tuple $T_{\mathcal C'} = (\gamma')$. Note that   
 $V_{exc}^{\mathscr C'}B'=  V_{exc}^\mathscr C B.$  
We then define
$$\mathbf B'=(\mathcal B',   \mathscr C', (\mathcal P_u')_{u\in V_{exc}^{\mathscr C'}B})$$
where $\mathcal P_u'=\mathcal P_u$ if $u\neq x$ and $\mathcal P_x'=(T_x', (\gamma_f') \oplus  P_x')$,  see Figure~\ref{fig:foldIIA2b}.
\begin{figure}[h!]
\begin{center}
\includegraphics[scale=1]{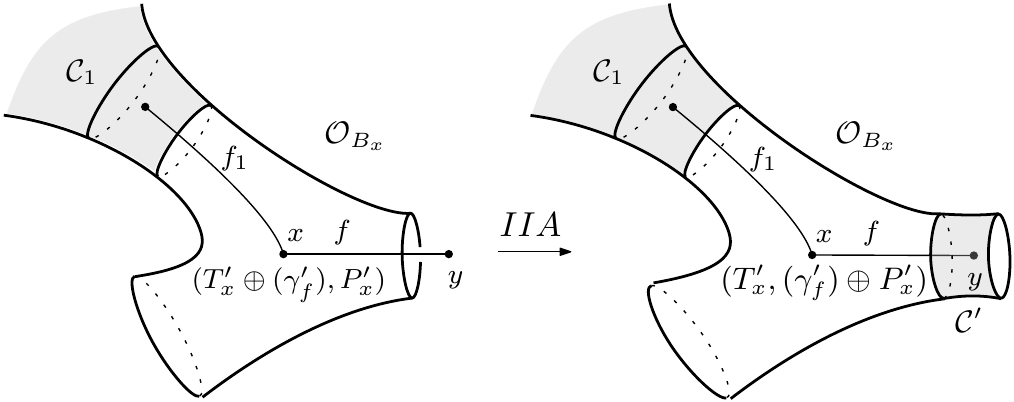}
\caption{The marking of $\mathcal B'$ when (El.1) holds.}{\label{fig:foldIIA2b}}
\end{center}
\end{figure} 
As observed in Remark~\ref{remark:noncritical},  the partitioned tuple $\mathcal P_x'$ is non-critical and of simple type. In this case we say that the fold is \emph{good}.  
\end{enumerate} 

\smallskip 

\noindent{\emph{Fold IIA(El.2):}} Suppose that (El.2) holds, i.e.~$B_x=\langle T_x\oplus P_x\rangle\le A_{[x]}=\pi_1^o(\mathcal O_{[x]})$  is of finite index and the elements in $P_x$ correspond to all but one boundary component of the orbifold $\mathcal O_{B_x}$ corresponding to $B_x$. The  geometry of this case is depicted in Figure~\ref{fig:foldIIA2b}. We consider two cases depending on $B_y$. \begin{enumerate}
\item If $B_y\neq 1$, then we say that the tame elementary  fold  based on the edge $f$ is \emph{bad} and we define no marking of $\mathcal B'$.
 
\item If $B_y=1$,  then we define a marking on $\mathcal B'$ in the following way.  We  first define the collection $\mathscr C'$ satisfying (A.1) and (A.2). Assume that 
 $Star(x, B)\cap E(\mathcal B)= \{f_1, \ldots, f_r\}.$  
For each $1\le  j\le  r$ let $(\mathcal C_j, u_{\mathcal C_j}, T_{\mathcal C_j})\in \mathscr C$   such that  $\omega(f_j)$ is a boundary vertex of  $\mathcal C_j$.  Let $C$  be the  sub-graph of $B=B'$ that has vertex set
$$ VC=  VC_1\cup \ldots \cup VC_r \cup \{x, y\}$$   
and edge set
$$EC= EC_{1}\cup \ldots \cup EC_r \cup \{f_1, \ldots, f_r, f\}^{\pm1}.$$
Denote by $\mathcal C$ (resp.~$\mathcal C'$) the sub-$\mathbb A$-graph of $\mathcal B$ (resp.~of $\mathcal B'$) carried by $C$. It follows from  condition (II.3) that $\mathcal C'$ is an $\mathbb A$-graph of orbifold type contained in  $\mathcal B'$.

Next we  provide a rigid generating tuple of $\pi_1(\mathbb C', x)$. The fold along $f$ carries  $\mathcal C$ into $\mathcal C'$. As $B_y=1$ epimorphism $\nu: \pi_1(\mathbb C, x)\rightarrow \pi_1(\mathbb C', x)$ induced by the fold is an isomorphism. Chose a maximal subtree  $Y_C\subseteq C$ and for each $1\le i\le r$  let 
$$p_i:=1, e_{i,1}, 1, \ldots, 1 , e_{i, l_i}, 1$$ 
where $e_{i, 1}, \ldots, e_{i, l_i}$ is the unique reduced path in $Y_C$ from $x$ to $u_{\mathcal C_i}$. Put  
$$ T_x':=\nu(T_x) \ \ \text{ and  } \ \ T_i':=\nu(p_i T_{\mathcal C_i} p_i^{-1})$$
 for $1\le i\le  r$.  Since $T_x\oplus P_x$ is a rigid generating tuple of $B_x$ it follows that 
$$T_{\mathcal C'}:=T_{x}'\oplus T_{\mathcal C_1}'\oplus \ldots \oplus T_{\mathcal C_r}'$$
 is a rigid generating tuple of $\pi_1(\mathbb C', x)$.   Thus $(\mathcal C', u_{\mathcal C'},  T_{\mathcal C'})$ is a marked $\mathbb A$-graph of orbifold type  such that  $\mathcal C'\subseteq \mathcal B'$, and hence  $$\mathscr C':=(\mathscr C \setminus \{(\mathcal C_1, u_{\mathcal C_1}, T_{\mathcal C_1}), \ldots, (\mathcal C_r, u_{\mathcal C_r}, T_{\mathcal C_r})\}) \cup \{(\mathcal C', x, T_{\mathcal C'})\}$$ is a collection of marked $\mathbb A$-graphs of orbifold type  satisfying condition  (A.1).    Condition (A.2)  follows as  no member of $\mathscr C$ contains $x$ or  $y$. Observe also  that the $V_{exc}^{\mathscr C'}B=V_{exc}^{\mathscr C}B\setminus\{x\}.$ 
\begin{figure}[h!]
\begin{center}
\includegraphics[scale=1]{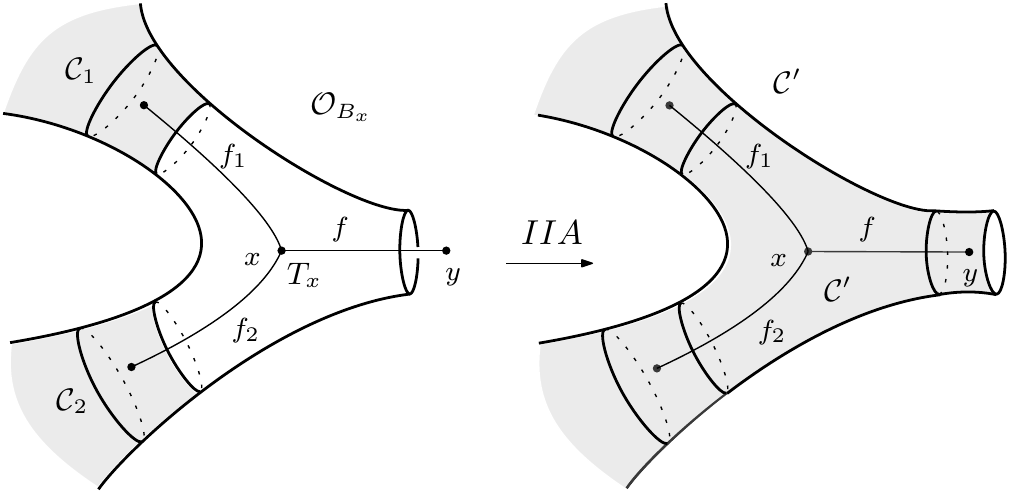}
\caption{The marking of $\mathcal{B}'$ when  (El.2) holds.}{\label{fig:foldIIA2b}}
\end{center}
\end{figure}
We then  define
$$\mathbf B'=(\mathcal B', u_0, \mathscr C', (\mathcal P_u')_{u\in V_{exc}^{\mathscr C'}B})$$ 
where  $\mathcal P_u'=\mathcal P_u$ for all $u\neq x$.  
 In this case we say that the fold is \emph{good}. Note that $x$ turn into a vertex of orbifold type in $\mathbf B'$. 
\end{enumerate}

\medskip

We now discuss tame elementary folds of type IA and IIIA. Those are elementary folds of type IA or IIIA such that the edge group of at least one of the edges involved in the fold is trivial. These types of fold are less subtle than the case of folds of type IIA as no preprocessing/normalization is required,  see Remark~\ref{remark_preprocessing} and Defintion~\ref{Def_normalized}.  Let 
$$\mathbf B=(\mathcal B,   \mathscr C, (\mathcal P_u)_{u\in V_{exc}^{\mathscr C}B})$$ 
be a \emph{tame  marked $\mathbb A$-graph} and suppose that $\mathcal B'$ is obtained from $\mathcal B$ by a tame elementary fold of type IA or IIIA. We proceed as before by considering the various cases by defining a marking of $\mathcal B'$ if appropriate and calling the folds \emph{good} or \emph{bad}. 
 
\smallskip

\noindent{\bf \em Tame elementary folds of type IA}. In this case  the fold identifies a pair of distinct  edges $f_1$ and $ f_2$ in $B$ with same label $(a, e, b)$, same initial vertex  $x:=\alpha(f_1)=\alpha(f_2)$, and distinct terminal vertices   $y_1:=\omega(f_1)$ and $y_2:=\omega(f_2)$. 

The common image of $f_1$ and $f_2$ (resp.~$y_1$ and $y_2$)  in $B'$ under the fold will be denoted by $f$ (resp.~by $y$), see Figure~\ref{fig:tamefoldIA}. As the fold is tame we may assume that $B_{f_2}$ is trivial. Thus 
 $B_f'=\langle B_{f_1}, B_{f_2}\rangle =B_{f_1}$  and $ B_{y }'=\langle B_{y_1}, B_{y_2}\rangle.$  
\begin{figure}[h!]
\begin{center}
\includegraphics[scale=1]{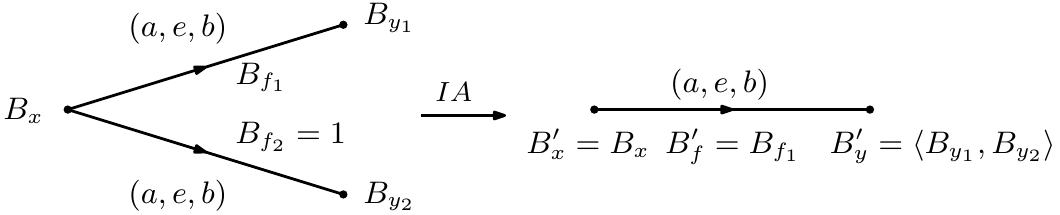}
\caption{A tame elementary fold of type IA}{\label{fig:tamefoldIA}}
\end{center}
\end{figure}
We say that $\mathcal B'$ is obtained from $\mathcal B$ by a \emph{tame elementary fold of type IA}. We distinguish the cases that $x$ is non-peripheral and that $x$ is peripheral.

\smallskip

\noindent{\emph{Fold IA(1)}}. Suppose that $x$ is non-peripheral. There are two subcases.
\begin{enumerate}
\item If $B_{y_1}=1$ or $B_{y_2}=1$, then the image of the marking of  $\mathcal B$  under the fold  is a marking for $\mathcal B'$ defining a tame marked $\mathbb A$-graph $\mathbf B'$. In this case we say that the tame elementary fold is \emph{good}. 

\item If $B_{y_1}\neq 1$ and $B_{y_2}\neq 1$  then we say that the tame elementary  fold is \emph{bad} and we define no marking of $\mathcal B'$. 
\end{enumerate}

\smallskip

\noindent{\emph{Fold IA(2)}}. Suppose that $x$ is  peripheral (and hence,  $y_1$ and $y_2$ are non-peripheral). We distinguish the cases that  at least one of the  vertices $y_1$ and $y_2$  is of orbifold type and the case that none of them is of orbifold type.

\begin{enumerate}
\item Let $i\neq j\in \{1,2\}$ and suppose that $y_i$ is of orbifold type. 

\begin{enumerate}
\item  If $B_{y_j}$ is trivial,  then  the image of the marking $(\mathscr C, (\mathcal P_u)_{u\in V_{exc}^\mathscr CB})$ of $\mathcal B$ under the fold  is a marking of $\mathcal B'$ defining a tame marked $\mathbb A$-graph $\mathbf B'$. In this case we say that  the  tame elementary fold is \emph{good}.

\item If the vertex group $B_{y_j}$ is non-trivial,  then we say that the  tame elementary fold is \emph{bad} and we define no marking of $\mathcal B'$.
\end{enumerate}

\item Suppose that $y_1$ and $y_2$ are not of orbifold type.   Thus $y_1$ and $y_2$ are exceptional, and so   $\mathcal P_{y_1} =(T_{y_1}, P_{y_1})$  and  $ \mathcal P_{y_2} =(T_{y_2}, P_{y_2})$ are non-critical and  of simple type. Observe that $\mathcal B'$ contains $\mathcal C$ for each  $(\mathcal C, u_{\mathcal C}, T_{\mathcal C})\in \mathscr C$.  We define
$$\mathbf B'=(\mathcal B',   \mathscr C , (\mathcal P_u')_{u\in V_{exc}^{\mathscr C}B'}).$$
Since $y_1$ and $y_2$ are exceptional  it follows that  
$$V_{exc}^{\mathscr C}B'=(V_{exc}^{\mathscr C}B-\{y_1, y_2\})\cup \{y\}.$$ 
The  partitioned tuples  are given by $\mathcal P_{w}'=\mathcal P_w$ if $w\neq y$  and 
 $\mathcal P_{y}'=(T_{y_1}\oplus T_{y_2}, P_{y_1}\oplus P_{y_2}).$   
 Again we consider all possible cases for $\mathcal P_y'$.  
\begin{enumerate}
\item If   $\mathcal P_y'$ is non-critical of simple type then $\mathbf B'$ is tame and we call the  fold \emph{good}.  
  
\item If $\mathcal P_{y}'$ is of almost orbifold covering type, then  $\mathbf B'$ is not tame.  

Let $\mathbf B''$ be the marked $\mathbb A$-graph that is obtained from $\mathbf B'$ by successively removing valence one  vertices  with trivial group and their adjacent edges.    Then   one of the following occurs:  
\begin{enumerate}
\item $\mathbf B''$ is  a special almost orbifold covering with a good marking. Then we say that the tame elementary  fold is \emph{good}.

\item The previous case  does not occur. Then we say that  fold is \emph{bad}. 
\end{enumerate}
 
\item If $\mathcal P_{y}'$ is not of the two types above, then again we say that the tame elementary fold is \emph{bad}.
\end{enumerate}
\end{enumerate}

\smallskip

\noindent{\bf \em Tame elementary  folds of type IIIA}.  In this case  two edges $f_1$ and $f_2$  labeled $(a, e, b_1)$ and $(a, e, b_2)$ with   $x:=\alpha(f_1)=\alpha(f_2)$ and  $y:=\omega(f_1)=\omega(f_2)$ are identified into a single edge $f$ labeled $(a, e, b_1)$, see Figure~\ref{fig:tamefoldIIIA}.
\begin{figure}[h!]
\begin{center}
\includegraphics[scale=1]{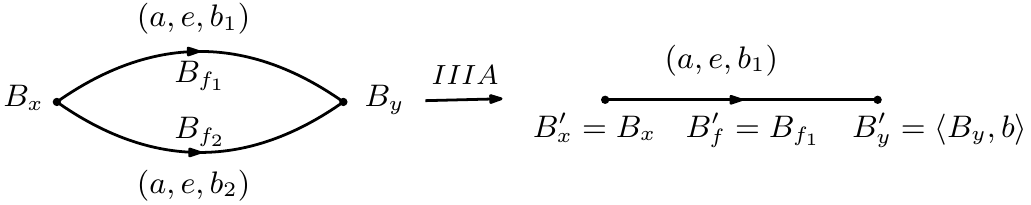}
\caption{A tame elementary fold  of type IIIA}{\label{fig:tamefoldIIIA}}
\end{center}
\end{figure}

To simplify notation we identify the vertex set of the graph $B'$ (underlying $\mathcal B'$) with the vertex set of the graph $B$ (underlying $\mathcal B$).  As the fold is tame we may assume that $B_{f_2}$ is trivial. Thus  $B_f'=\langle B_{f_1}, B_{f_2}\rangle =B_{f_1} $ and $B_y'=\langle B_y, b\rangle$  where  $b:=b_1^{-1}b_2 \in A_{[y]}$. We say that $\mathcal B'$ is obtained from $\mathcal B$ by a \emph{tame elementary fold of type IIIA}. As in the previous case  we distinguish the cases that $x$ is non-peripheral and that $x$ is peripheral. We distinguish two cases depending on the type of $x$.

\smallskip
 
\noindent{\emph{Fold IIIA(1):}}  Suppose that $x$ is non-peripheral (hence,  $y$ is peripheral).
\begin{enumerate}
\item If $B_y\neq 1$,  then we say that the fold is \emph{bad}  and we define no marking of $\mathcal B'$. 

\item  If  $B_y=1$  and $b_1=b_2$,   then we say that the  fold is \emph{bad}   and we define no marking of $\mathcal B'$.

\item  If  $B_y\neq 1$ and $b_1\neq b_2$,  then we call the fold \emph{good} and define a marking for       $\mathcal B'$  as follows.  Observe that $\mathcal B'$ contains $\mathcal C$ for each  $(\mathcal C, u_{\mathcal C}, T_{\mathcal C})\in \mathscr C$. We define   
$$\mathscr C'=\mathscr C \cup \{(\mathcal C'', u_{\mathcal C''}, T_{\mathcal C''})\}$$  
where $(\mathcal C'', u_{\mathcal C''}, T_{\mathcal C''})$ is the degenerate marked $\mathbb A$-graph of orbifold type consisting  of the peripheral vertex $y$ and rigid generating tuple  $T_{\mathcal C''}=(b)$.  Thus  $V_{exc}^{\mathscr C'}B=V_{exc}^{\mathscr C}B.$ We then  define 
$$\mathbf B'=(\mathcal B',   \mathscr C', (\mathcal P_u')_{u\in V_{exc}^{\mathscr C'}B'})$$
where   $\mathcal P_w'=\mathcal P_w$ for all $w\in V_{exc}^{\mathscr C'}B'$. 
\end{enumerate}

\smallskip

\noindent {\emph{Fold IIIA(2)}}. Suppose that $x$ is   peripheral (hence $y$ is non-peripheral). We need to distinguish the cases that $y$ is of orbifold type and the case that $y$ is  exceptional. 
\begin{enumerate}
\item If $y$ is of orbifold type,   then   the fold is \emph{bad} and we define no marking of $\mathcal B'$.

\item If $y$ is  exceptional  then $\mathcal P_y=(T_y, P_y)$ is  non-critical  of simple type.   Observe that $\mathcal B'$ contains $\mathcal C$ for each  $(\mathcal C, u_{\mathcal C}, T_{\mathcal C})\in \mathscr C$.  Thus
 $V_{exc}^{\mathscr C }B'=V_{exc}^\mathscr CB.$ 
We define 
$$\mathbf B'=(\mathcal B',  \mathscr C , (\mathcal P_u')_{u\in V_{exc}^{\mathscr C }B'})$$ 
 where $\mathcal P_u'=\mathcal P_u$ if $u\neq y$ and 
 $$\mathcal P_{y}'=(T_y\oplus (b), P_y).$$
To decide  if the fold is good or bad  we consider all possibilities for the partitioned tuple $\mathcal P_y'$ as in  FoldIA(2)(2).
\end{enumerate}

\begin{lemma} 
If a marked $\mathbb A$-graph $\mathbf B'$  is obtained from a tame marked $\mathbb A$-graph $\mathbf B$ by a tame elementary fold, then  $[T_{\mathbf B'}]=[T_\mathbf B].$ 
In particular, $[T_\mathbf b]=[T_{\mathbf b'}]$. 
\end{lemma}

\begin{proof}
 We will give a detailed  proof  for the case of folds of type IA. The remaining cases are slightly  simpler  and follows from inspecting the various cases. We follow the notation  from the previous paragraphs.
 
As the fold yields a marked $\mathbb A$-graph  at least one of the vertices $y_1$ and $y_2$, say $y_2$, is not of orbifold type. This implies that   the intersection of $\mathcal C$ (with $(\mathcal C, u_{\mathcal C}, T_{\mathcal C})$  in $ \mathscr C$) with $f_1\cup f_2$ is contained in the segment $f_1$. We distinguish two cases.  

\smallskip

\noindent (1) $\mathcal C\cap (f_1\cup f_2)$ is connected for all $(\mathcal C, u_{\mathcal C}, T_{\mathcal C})\in \mathscr C$ (it might be a single vertex or the entire edge $f_1$). In this case  we can choose  a maximal  $\mathscr C$-subtree $Y$  that contains $f_1\cup  f_2$. The image $Y'$ of $Y$ under the fold is clearly a maximal subtree of $B'$ that contains  $f$.  The fold  maps $C$  isomorphically onto $C'$  for each  $(\mathcal C, u_{\mathcal C}, T_{\mathcal C})$ in $\mathscr C$. Thus   $Y'\cap C'$   is  a maximal subtree of $C'$  for each  $(\mathcal C', u_{\mathcal C'}, T_{\mathcal C'})\in \mathscr C'$, that is, $Y'$ is a maximal $\mathscr C'$-subtree. To complete the argument observe that 
 $\phi_\mathcal B(T_{Y, u_0}^{\mathbf B})=\phi_{\mathcal B'}(T_{Y', u_0'}^{\mathbf B'}).$

\noindent (2) There is $(\mathcal C, u_{\mathcal C}, T_{\mathcal C})\in \mathscr C$ such that  $\mathcal C\cap (f_1\cup f_2)=\{x,y_1\}$. In this case 
 $\bar{C} \cap (f_1\cup f_2)=\emptyset$  for all $(\bar{\mathcal C}, u_{\bar{\mathcal C}}, T_{\bar{\mathcal C}})\in \mathscr C\setminus\{(\mathcal C, u_{\mathcal C}, T_{\mathcal C})\}$. 
\begin{figure}[h!]
\begin{center}
\includegraphics[scale=1]{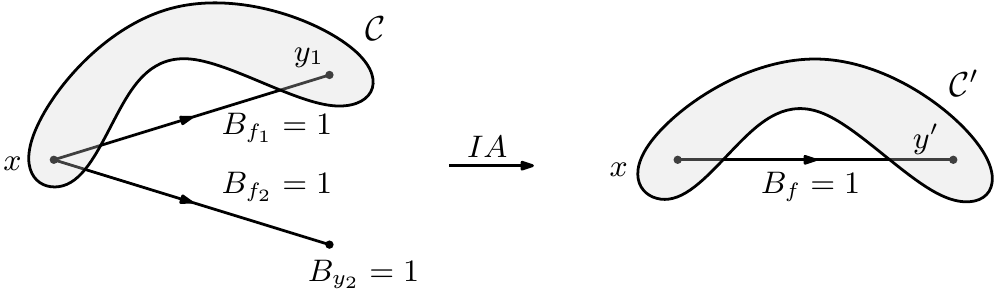}
\caption{The intersection $C\cap (f_1\cup f_2)$ is not a tree.}{\label{fig:tamefoldIAlemma}}
\end{center}
\end{figure} 
Moreover, the assumption that the fold yields a  marked $\mathbb A$-graph implies that $B_{y_2}=1$. We can therefore  choose a  maximal $\mathscr C$-subtree $Y$  that contains  $f_2$. The image of $Y\setminus\{f_1,f_2\}^{\pm1}$ under the fold is a maximal $\mathscr C'$-subtree $Y'$ of $B'$.  Observe that  $f_1\in E_{exc}^{\mathscr C}B\setminus EY$  and $  f\in E_{exc}^{\mathscr C'}B'\setminus EY'$. It is not hard to see that 
$$T_{Y, u_0}^{\mathbf B}=T\oplus (g_Y) \  \ \text{ and  } \  \ T_{Y, u_0'}^{\mathbf B'}=T'\oplus (g_{Y'}')$$ 
where  $g_Y\in \pi_1(\mathbb B, u_0)$ corresponds to $f_1$ and   $g_{Y'}'\in \pi_1(\mathbb B', u_0')$ corresponds to $f$ such that $$h:=\phi_\mathcal B(g_Y)=\phi_{\mathcal B'}(g_{Y'}').$$  
Observe now  that  each element of $\phi_{\mathcal B'}(T')$ is obtained from an element of $\phi_\mathcal B(T)$ by doing nothing,  or by left or right multiplication with $h$ or by conjugation with $h^{\pm1}$.  Therefore  $\phi_\mathcal B(T_{Y, u_0}^{\mathbf B})$ is Nielsen equivalent to $\phi_{\mathcal B'}(T_{Y', u_0'}^{\mathbf B'})$. This completes the proof in the case of a tame elementary fold of type IA.  
\end{proof}

\noindent\textbf{The core of a marked $\mathbb A$-graph}. In this section we explain how the core of the underlying $\mathbb A$-graph of  a marked $\mathbb A$-graph  inherits a  marking. Let $\mathbf B=(\mathcal B,\mathscr C,(\mathcal P_u)_{u\in V_{exc}^{\mathscr C}B})$ 
be a (not necessarily tame) marked $\mathbb A$-graph. 

If $\mathcal B=core(\mathcal B)$  then there is nothing to do. Otherwise  there is $f\in EB$ such that  $val(\omega(f), B)=1 $ and $ B_{\omega(f)}=t_f^{-1}\omega_{[f]}(B_f)t_f$.  
There are two cases depending on the vertex group at $\omega(f)$:

\begin{enumerate}
\item   $B_{\omega(f)}=1$. Let $\mathcal B'$ be the $\mathbb A$-graph  that is obtained from $\mathcal B$ by removing the edge pair $\{f, f^{-1}\}$ and the vertex $\omega(f)$.  Note that $\mathcal C$ is contained in $\mathcal B'$ for all $(\mathcal C, u_{\mathcal C}, T_{\mathcal C})$ in $\mathscr C$. Moreover,  $V_{exc}^{\mathscr C}B'=V_{exc}^{\mathscr C}B\setminus \{\omega(f)\}.$  
 
We  define 
 $\mathbf B'=(\mathcal B', \mathscr C,(\mathcal P_u')_{u\in V_{exc}^{\mathscr C}B'}) $
where $\mathcal P_u'=\mathcal P_u$ for all $u\in V_{exc}^{\mathscr C}B'$.

\item   $B_{\omega(f)}\neq 1$. There are two subcases depending on the tameness of $\mathbf B$.  
\begin{enumerate}
\item $\mathbf B$ is tame. In this case there is a unique tame marked $\mathbb A$-graph $\mathbf B'$  such that $\mathbf B$ is obtained from $\mathbf B'$ defined as in case (1) by a tame elementary fold of type IIA based on $f$.

\item $\mathbf B$ is not tame.  It is not hard to see that in this case $\alpha(f)$ is exceptional. As $B_f\neq 1$  the partitioned tuple at $\alpha(f)$ has the form  $$\mathcal P_{\alpha(f)}=(T_{\alpha(f)}, (\gamma_f)\oplus P_{\alpha(f)})$$ where $\gamma_f\in B_{\alpha(f)}$ is the peripheral element  associated to $f$. 
On the other hand   $\omega(f)$ is peripheral and of orbifold type as $B_{\omega(f)}\neq 1$.  

Let $\mathcal B'$ be the $\mathbb A$-graph that is obtained from $\mathcal B$ by replacing   $B_f$ and  $B_{\omega(f)}$  by $B_f'=1$ and $B_{\omega(f)}=1$ respectively.   
Note that $\mathcal C$ is contained in  $\mathcal B'$  for all $(\mathcal C, u_{\mathcal C}, T_{\mathcal C})$ in $\mathscr C$. 

We  define  $\mathbf B'=(\mathcal B',\mathscr C, (\mathcal P_{u}')_{u\in V_{exc}^{\mathscr C}B'})$ 
where $\mathcal P_u'=\mathcal P_u$ if $u\neq \alpha(f)$ and 
 $$\mathcal P_{\alpha(f)}'=(T_{\alpha(f)}\oplus (\gamma_f), P_{\alpha(f)}).$$   
\end{enumerate}
\end{enumerate}
In both cases the new marked $\mathbb A$-graph $\mathbf B'$ such that (i) $core(\mathcal B)\subseteq \mathcal B'$ and (ii) $|EB'|<|EB|$ or $|EB'|=|EB|$ but  $|E(\mathcal B')|<|E(\mathcal B)|$.  
Now we repeat the argument with $\mathbf B'$ playing the role of $\mathbf B$. After finitely many steps we obtain a  marking for   $core(\mathcal B)$. We denote this marked $\mathbb A$-graph by $core(\mathbf B)$. \emph{We will always assume that $core(\mathcal B)$  is equipped with this marking also called the induced marking.}

\begin{lemma}{\label{lemma:equivcore}}
Let $\mathbf B$  and $\mathbf B'$ be   tame marked $\mathbb A$-graphs. If  $\mathbf B$ is equivalent to $\mathbf B'$, then $core(\mathbf B)$ is equivalent to $core(\mathbf B')$.  
\end{lemma} 
\begin{proof}
Note that in going from $\mathbf B$ to  $core(\mathbf B)$  we do not need to apply elementary moves. This implies that we can carry $core(\mathbf B)$ onto $core(\mathbf B')$ by the same elementary moves that carry $\mathbf B$ onto $core(\mathbf B')$.   
\end{proof}

%---------------------------------------------------------

\subsection{The graph associated to a generating tuple of $\pi_1^o(\mathcal O)$}
We now proceed as in  the local case,  that is, to any  generating tuple  of $\pi_1^o(\mathcal O)$ we associate a  directed graph. We then show that in the case of a non-standard irreducible generating tuple  this graph has  a unique root which corresponds to the special almost orbifold covering with a good marking that represents the Nielsen class of the generating tuple.  

\smallskip

Let $T=(g_1, \ldots, g_m)$ be a generating tuple of $\pi_1^o(\mathcal O)=\pi_1(\mathbb A, v_0)$.  The directed graph $\Omega_{[T]}$ associated to the Nielsen class of $T$ is defined as follows. 
\begin{enumerate}
\item[•] The vertex set  $V\Omega_{[T]}$  of $\Omega_{[T]}$ is the set of all equivalence classes of minimal marked $\mathbb A$-graphs $\mathbf b$   such that  $[T_\mathbf b]=[T]$ and  some (and therefore any) representative  of $\mathbf b$   is either tame  or is a special almost orbifold cover  with a good marking.  

\item[•] There is a directed edge $\mathbf b_1\mapsto \mathbf b_2$  from  $\mathbf b_1$ to   $\mathbf b_2$ if there are representatives $ \mathbf{B}_1$ and  $\mathbf{B}_2$ of $\mathbf b_1$  and $\mathbf b_2$    such that $\mathbf{B}_2$ is the core  (with the induced marking) of the \emph{marked $\mathbb A$-graph} that is obtained from $\mathbf B_1$ by a \emph{good tame  elementary fold}.
\end{enumerate}
We define a height function  $h:V \Omega_{[T]}\to \mathbb N_0\times \mathbb N_0$ by  $h(\mathbf b)= (|EB|, |EB|-|E(\mathcal B)|)$   
where $\mathcal B$ and $B$ are the underlying $\mathbb A$-graph and underlying graph respectively of an arbitrary representative of $\mathbf b$.  Throughout this section we assume that  $\mathbb N_0\times \mathbb N_0$ is endowed with the lexicographic order.  Recall that  $E(\mathcal B)=\{f\in EB \ | \ B_f\neq 1\}.$  
Observe that $h$ is well defined as  equivalent tame marked $\mathbb A$-graphs have the same underlying graph and all edges have isomorphic edge groups.

\smallskip

Let $\mathbf b$ and $\mathbf b'$ be vertices of $\Omega_{[T]}$.  We say that $\mathbf b$ \emph{projects} onto  $\mathbf b'$, and write $\mathbf b\rightsquigarrow \mathbf b'$,  if there is a directed path in $\Omega_{[T]}$ from $\mathbf{b}$ to $\mathbf b'$.  It is clear that   $h(\mathbf b')<h(\mathbf b)$ if $\mathbf b\rightsquigarrow \mathbf b'$.

\smallskip

We say that $\mathbf b\in V\Omega_{[T]}$ is a \emph{root} if there is no $\mathbf{b}'\in \Omega_{[T]}$  such that $\mathbf b\mapsto \mathbf b'$. The following lemma characterizes the roots of $\Omega_{[T]}$. 
\begin{lemma}{\label{lemma:roots}}
A vertex $\mathbf{b}$ of $\Omega_{[T]}$ is a root if and only if one of the following occurs:
\begin{enumerate}
\item[(1)] some (and therefore any) representative of $\mathbf b$ is  folded. 

\item[(2)] some (and therefore any) representative of $\mathbf b$ is a  special almost orbifold cover with a good marking. 

\item[(3)] if a tame elementary fold is applicable to a representative  of $\mathbf b$, then the fold is bad.  
\end{enumerate}
\end{lemma}
\begin{proof}
The lemma  follows immediately from the definition of the edge set of $\Omega_{[T]}$  and  the fact that no tame elementary fold is applicable to a special almost orbifold cover with a good marking.  
\end{proof}

\begin{remark}
According to \cite[Lemma~4.20]{Dut}  a generating tuple of $\pi_1^o(\mathcal O)\cong \pi_1(\mathbb A, v_0)$ that is represented   by a folded marked $\mathbb A$-graph  is either reducible or standard. This follows from the fact that a folded $\pi_1$-surjective $\mathbb A$-graph recovers the splitting of $\pi_1^o(\mathcal O)$.  We conclude that $T$ is standard  or reducible whenever $\Omega_{[T]} $ has a root represented by a folded marked $\mathbb A$-graph. 
\end{remark}

\begin{definition}
We  say that  $\mathbf b\in V\Omega_{[T]}$ is \emph{pre-bad} if a bad tame elementary fold is applicable to some representative of $\mathbf b$. 
\end{definition}

\begin{remark}
Note that roots of type (3)  are pre-bad vertices.  However,  pre-bad vertices are not necessarily roots  as a good tame elementary fold can still be applicable to some representative of the vertex.
\end{remark}

\begin{lemma}{\label{lemma:corerevealsbad}}
Let $\mathbf B$ be a tame marked $\mathbb A$-graph. Then $\mathbf B$ admits a bad tame elementary fold if, and only if, $core(\mathbf B)$ admits a bad tame elementary fold.
\end{lemma}
\begin{proof}
Assume that $core(\mathbf B)$ is obtained from $\mathbf B$ by removing a single edge $f$. We can further assume that the bad tame elementary  fold  $F$ identifies $f$ with some edge  $g$  in $Star(\alpha(f), B)$ as all other folds can also be applied to $core(\mathbf B)$.

We give a complete argument for the case $\omega(f)$ is peripheral (the case $\omega(f)$ non-peripheral  can be proven similarly).  In this case the badness of $F$   means that $B_{\omega(f)}\neq 1$ (which implies that $B_f\neq 1$ since $f$ is not in the core of $\mathcal B$) and  $B_{\omega(g)}\neq 1$.  The tameness of $F$ implies that $B_g=1$. The fact that $F$ is elementary means that $f$ and $g$ have the same label. Thus  
 $$B_{\alpha(g)} \cap o_g\alpha_{[f]}(A_{[g]})o_g^{-1} = B_{\alpha(f)} \cap o_f\alpha_{[f]}(A_{[f]})o_f^{-1}\neq 1,$$
and so a fold of type IIA based on $g$ can be applied to $core(\mathcal B)$.  This fold is tame since  $\alpha(f)$ and $f$ lie in some sub-$\mathbb A$-graph of orbifold type  or  $\alpha(f)$ is exceptional and $\mathcal P_{\alpha(f)}$ is of simple type (and so the edges in $Star(\alpha(f), B)\cap E(\mathcal B)$ correspond to distinct boundary components of the orbifold corresponding to $B_{\alpha(f)}$).   This fold is bad since   $B_{\omega(g)}\neq 1$.
\end{proof}

\begin{lemma}\label{lemma:corealsot}
Let $\mathbf B$ be a marked $\mathbb A$-graph. Then $core(\mathbf B)$ is a special almost orbifold covering with a good marking iff  $B_f=1$ for all  $f\notin core(\mathcal B)$.
\end{lemma}
\begin{proof}
Assume that $core(\mathbf B)$ is obtained from $\mathbf B$ by removing a single edge $f$ with non-trivial group.  We will show that $core(\mathbf B)$ cannot be an almost orbifold cover with a good marking.

Assume fist that  $\omega(f)$ is non-peripheral. Then $\alpha(f)$ is peripheral and  of orbifold type since $B_{\alpha(f)}\neq 1$.    Condition (5) of Definition~\ref{def:marked}   implies that $B_h=1$ for all $h\in Star(\alpha(f), B)\setminus\{f\}$ with $[h]=[f]$.  Therefore  $\alpha(f)$ is a vertex of orbifold type  in $core(\mathbf B)$  for  which $Star(\alpha(f), B_{core})\cap E(core(\mathcal B))$ contains at most one  edge.  However, it follows from the definition   that in  a special almost orbifold cover with a good marking  all peripheral vertices have  exactly two incident edges with non-trivial group.  

Assume now that $\omega(f)$ is peripheral. Then $\alpha(f)$ is exceptional. In going from $\mathbf B$ to $core(\mathbf B)$ we turn $\alpha(f)$ into a exceptional vertex  of simple type if $\alpha(f)$ is of orbifold type  or we replace $\mathcal P_{\alpha(f)}=(T_{\alpha(f)}, (\gamma_f)\oplus P_{\alpha(f)})$ by $(T_{\alpha(f)}\oplus (\gamma_f), P_{\alpha(f)})$. In the second case, Proposition~\ref{prop:02} implies that  $(T_{\alpha(f)}\oplus (\gamma_f), P_{\alpha(f)})$ is not of almost orbifold covering type. This shows that in both cases $core(\mathbf B)$ cannot be an almost orbifold cover with a good marking since it violates condition (1) of Definition~\ref{def:almostorb}.  
\end{proof}

\begin{lemma}
The graph $\Omega_{[T]}$ is connected.
\end{lemma}
\begin{proof}
Let $\mathbf b_1=[\mathbf B_1]$ and $\mathbf b_2=[\mathbf B_2]$ be distinct vertices of $\Omega_{[T]}$. We need to show that $\mathbf{b}_1 $ and $\mathbf{b}_2$ lie in the same component of  $\Omega_{[T]}$.

In \cite[Lemma~3.24]{Dut} it is  observed  that $\mathbf B_i$ can be unfolded until we obtain a  marked $\mathbb A$-graphs $\mathbf B_i'$ with trivial edge groups. Note that $\mathbf B_i'$ is  minimal since otherwise $\mathbf B_i$ would not be minimal. Let $\mathbf b_i'$ be the vertex of $\Omega_{[T]}$ represented by  $\mathbf B_i'$. Since $\mathbf B_i$ is obtained from $\mathbf B_i'$ by finitely many good tame elementary folds of type IIA it follows that $\mathbf b_i' \rightsquigarrow\mathbf b_i$. In particular  $\mathbf b_i' $ and $\mathbf b_i$ lie in the same component of $\Omega_{[T]}$. It therefore suffices to show that $\mathbf b_1'$ and $\mathbf b_2'$ lie in the same component.

Observe now that $\mathbf B_i'$ is a marked $\mathbb A$-graph in the sense of \cite{W2}. Thus the argument used in \cite[Lemma 7]{W2} applies in our case  to show that the vertices $\mathbf b_1'$ and $\mathbf b_2'$ lie in the same component of $\Omega_{[T]}$. Therefore $\mathbf b_1 $ and $\mathbf b_2$ lie in the same component of  $\Omega_{[T]}$.
\end{proof}

The main steps in the proofs of our main results are the following two Propositions whose proofs we postpone to the next section.

\begin{proposition}{\label{lemma:fork2}}
Let $\mathbf b$ and $\mathbf b'$ be vertices of $\Omega_{[T]}$ such that $\mathbf b\mapsto \mathbf b'$. If $\mathbf b$ is pre-bad bad, then $\mathbf b'$ is pre-bad.
\end{proposition}

\begin{proposition}{\label{lemma:fork1}}
Let  $\mathbf b$,   $\mathbf b_1$ and   $\mathbf b_2$ be distinct vertices of $\Omega_{[T]}$. Assume that  $\mathbf b\mapsto \mathbf b_1$ and $\mathbf b\mapsto\mathbf b_2$. Then one of the following holds:
\begin{enumerate}
\item $\mathbf b_1$ and $\mathbf b_2$ are pre-bad vertices.

\item $\mathbf b_1\mapsto \mathbf b_2$ or $\mathbf b_2\mapsto \mathbf b_1$.

\item There is a vertex $\mathbf b' \in \Omega_T$ such that $\mathbf b_i\mapsto \mathbf b'$ for $i=1,2$.
\end{enumerate}
\end{proposition}
Assuming Proposition~\ref{lemma:fork2} and Proposition~\ref{lemma:fork1} we can finish the argument for our main results.
 
\begin{proposition}\label{propuniqueroot}
If $\Omega_{[T]}$ has a root  represented by a special almost orbifold covering with a good marking, then this root is unique.
\end{proposition}
\begin{proof}
Let $\mathbf b_1$ be a root represented by a special almost orbifold covering with a good marking. Suppose that $\Omega_{[T]}$ has  a root  distinct from $\mathbf b_1$. We claim that there is a root $\mathbf b_2\neq \mathbf b_1$ such that for some vertex  $\mathbf b$ we have $\mathbf b \rightsquigarrow \mathbf b_1$ and $\mathbf b\rightsquigarrow \mathbf b_2$. 

Choose a root $ \mathbf b_2$ and a path $p$ connecting $\mathbf b_1$ and $\mathbf b_2$ such that the number of local minima with respect to the height function is minimal among all pairs $(\mathbf b_2, p)$.  It clearly suffices to show that $p$ has no local minima besides $\mathbf{b}_1$ and $\mathbf{b}_2$, as $\mathbf{b}$ can then be chosen to be the vertex of $p$ at which this local maximum is attained. If $p$ has another local minimum   at a vertex $\mathbf b^*$,   then $\mathbf b^*$  projects onto either $\mathbf b_1$ or onto some root distinct from $\mathbf b_1$. In both situations we find a pair as above with fewer local minima which is a contradiction.

Choose $(\mathbf b_2, p)$ as above such that the single maximum of $p$ occurs at a minimal height among all such pairs.  Choose $\mathbf b_1'$ and $\mathbf b_2'$ such that $\mathbf b \mapsto \mathbf b_i'$ and that   $\mathbf b_i' \rightsquigarrow  \mathbf b_i$. The minimality assumption implies that there is no vertex $\mathbf b''$ such that $\mathbf b_1'\rightsquigarrow \mathbf b''$ and  $\mathbf b_2'\rightsquigarrow \mathbf b''$.   Proposition~\ref{lemma:fork1}  implies that $\mathbf b_1'$ and $\mathbf b_2'$ are pre-bad. Proposition~\ref{lemma:fork2}  implies that $\mathbf b_1$ is pre-bad, a contradiction since  vertices represented by special  almost orbifold covers with good markings are not pre-bad.  the last claim follows from the fact that special  almost orbifold covers with good markings do not admit tame elementary folds.  
\end{proof}

\begin{proof}[Proof of Theorem~\ref{thm01}.] Let $T$ be a non-standard irreducible generating tuple of $\pi_1^o(\mathcal O)\cong \pi_1(\mathbb A, v_0)$. The main Theorem of  \cite{Dut} says that $[T]$ can be represented by  a special marking  $(\eta:\mathcal O'\rightarrow \mathcal O, [T])$. Let $\mathbf b$ be the corresponding vertex of $\Omega_{[T]}$. It follows from the previous proposition that $\mathbf b$ is the unique root of $\Omega_{[T]}$. Therefore  the special  marking $(\eta:\mathcal O'\rightarrow \mathcal O, [T'])$  representing the Nielsen class of  $T $ is unique.
\end{proof}

\begin{proof}[Proof of Theorem~\ref{thm02}.] Suppose to the contrary that $T:=\eta_{\ast}(T')$ is reducible.  Thus  $\Omega_{[T]}$ contains a pre-bad vertex and therefore, according to Proposition~\ref{lemma:fork1}, $\Omega_{[T]}$  contains  a pre-bad root. This contradicts the uniqueness established in Proposition~\ref{propuniqueroot}.
\end{proof}

%---------------------------------------------------------

\subsection{Edges of $\Omega_{[T]}$}

In this section we study edges in  $\Omega_{[T]}$. The subtle part being that vertices are defined as equivalence classes of marked $\mathbb A$-graphs while edges a defined using representatives.

Throughout this section we assume the following:  $\mathbf B_1=(\mathcal B_1, \mathscr C_1, (\mathcal P_{u,1})_{u\in V_{exc}^{\mathscr C_1}B_1}) $ and  $\mathbf B_2=(\mathcal B_2, \mathscr C_2, (\mathcal P_{u,2})_{u\in V_{exc}^{\mathscr C_2}B_2})$ are   minimal tame marked $\mathbb A$-graphs  which are equivalent  and  tame elementary folds $F_1$ an $F_2$ are applicable to  $\mathcal B_1$ and $\mathcal B_2$ respectively. We denote the resulting $\mathbb A$-graph by $\mathcal B_i'$. 

\begin{convention}\label{conventiontwofolds}
To have a unified notation  we will stick to the following notation: 
\begin{enumerate}
\item[•] The underlying graph of $\mathcal B_1$ (and therefore also of $\mathcal B_2$ since $\mathbf B_1$ and $\mathbf B_2$ are equivalent) will be denoted by $B$ and the underlying graph of $\mathcal B_i'$  will be denoted by $B_i'$.

\item[•] If $x\in VB\cup EB$, then the  group of $x$ in $\mathcal B_i$ will be denoted by $B_{x, i}$.  

\item[•]  If $F_i$ induces a marking on $\mathcal B_i'$,  then the resulting marked $\mathbb A$-graph will be denoted by
 $$\mathbf B_i'=(\mathcal B_i',  \mathscr C_i', (\mathcal P_{u, i}')_{u\in V_{exc}^{\mathscr C_i'}B_i'}) $$
where $\mathcal P_{u, i}'=(T_{u,i}', P_{u,i}')$ for all $u\in V_{exc}^{\mathscr C_i'}B_i'$.

\item[•] If $F_1$ (resp.~$F_2$) is of type IA/IIIA, then it identifies the edges $f_1$ and $f_2$ (resp.~$g_1$ and $g_2$) with  $x:=\alpha(f_1)=\alpha(f_2)\in VB$   (resp.~$w:=\alpha(g_2)=\alpha(g_2)\in VB$). 
 
\item[•] $F_1$ (resp.~$F_2$) is of type IIA, then it is based on $f\in EB$ (resp.~$g\in EB$) with $x:=\alpha(f)\in VB$ and  $y:=\omega(f)\in VB$ (resp.~$w:=\alpha(g)\in VB$ and $z:=\omega(g)\in VB$).   
\end{enumerate} 
\end{convention}

\begin{remark}{\label{remark:fatsequivalent}}
The following remarks, which  follow immediately from the definition of  elementary moves, will be frequently used without further comment: 

\smallskip

\noindent\textbf{(i)}  $y\in VB$ is peripheral (resp.~non-peripheral)  in $\mathbf B_1$ iff $y$ is peripheral (resp.~non-peripheral) in $\mathbf B_2$.  Thus there is no ambiguity  in saying that $y\in VB$ is a  peripheral (resp.~non-peripheral) vertex. 

\smallskip

\noindent\textbf{(ii)} For any peripheral vertex $y\in VB$ (resp.~for any edge $f\in EB$) we have $B_{y, 1}=B_{y, 2}$ (resp.~$B_{f, 1}=B_{f, 2}$). It therefore makes sense to write $B_y$ (resp. $B_f$) instead of $B_{y, 1}$ and $B_{y, 2}$ (resp. $B_{f, 1}$ and $B_{f, 2}$).  In particular,   $E(\mathcal B_1)=E(\mathcal B_2) \subseteq EB.$ 
 
\smallskip 
 
\noindent\textbf{(iii)}  $y\in VB$  is exceptional in $\mathbf B_1$ (resp.~of orbifold type) iff  $y$ is exceptional (resp.~orbifold type) in $\mathbf B_2$. Thus 
 $V_{orb}^{\mathscr C_1}B=V_{orb}^{\mathscr C_2}B$  and $V_{exc}^{\mathscr C_1}B=V_{exc}^{\mathscr C_2}B.$ 
We can  therefore simply say that $y$ is an exceptional vertex (resp.~is a vertex of orbifold type). 

\smallskip

\noindent\textbf{(iv)} If $y\in VB$ is exceptional, then   there is $g\in A_{[y]}$ and a partitioned tuple $(T, P)$ in $A_{[y]}=\pi_1^o(\mathcal O_{[y]})$   equivalent to $\mathcal P_{y, 1}$ such that  $\mathcal P_{y, 2}=g(T, P)g^{-1}.$  This shows that the orbifold covering corresponding  to  $\langle T_{y,1}\oplus P_{y,1}\rangle\le A_{[y]}$   coincides with the orbifold covering corresponding to $\langle T_{y,2}\oplus P_{y,2}\rangle\le A_{[y]} $. It therefore makes sense to denote it simply by $\eta_y:\mathcal O_y'\rightarrow \mathcal O_{[y]}$.   

\smallskip

\noindent\textbf{(v)} Observe that there are exactly  two types of bad tame elementary folds. The first type consists of those folds that do not induce a marking on the resulting $\mathbb A$-graph, which occurs exactly when a vertex of orbifold type is affected. The second type of bad folds consists of those folds that do induce a marking on the resulting $\mathbb A$-graph but the core  fails to be tame or an almost orbifold cover.
\end{remark}

Applying tame elementary folds of the same type to the same edge(s) of equivalent tame marked $\mathbb A$-graphs does not necessarily yield equivalent marked $\mathbb A$-graphs as in the local picture. The following lemma clarifies when this is the case. 
\begin{lemma}\label{lemma_edges1} Assume that  the following hold:
\begin{enumerate}
\item[(a)] $F_i$ induces a marking on $\mathcal B_i'$ for $i=1,2$.

\item[(b)] $F_1$ and $F_2$  are of the same type and affect the same edge(s).
\end{enumerate}

Then the marked $\mathbb A$-graphs  $\mathbf B_1'$ and $\mathbf B_2'$ are equivalent unless $F_1$ and $F_2$ are of type IA or of type IIIA and the following holds:
\begin{enumerate} 
\item[(i)] $B_{f_1}=B_{f_2}=1$.

\item[(ii)] $o_{f_1}^{\mathcal B_2}= g\alpha_1 o_{f_1}^{\mathcal B_1}  \alpha_e(c)$ and $o_{f_2}^{\mathcal B_2}= g\alpha_2  o_{f_2}^{\mathcal B_1}  \alpha_e(d)$ with $g\in A_{[x]}$,  $\alpha_1\neq \alpha_2\in B_{x,1}$ and $c\neq d\in A_e$,  where $e:=[f_1]=[f_2]\in EA$. 
\item[(iii)]
$t_{f_1}^{\mathcal B_2}=\omega_e(c)t_{f_1}^{\mathcal B_1}\beta_1h_1^{-1} $ and $ t_{f_2}^{\mathcal B_2}=\omega_e(d)t_{f_2}^{\mathcal B_1}\beta_2h_2^{-1}$  
 with  $\beta_i\in B_{\omega(f_i),1}$ and  $h_1, h_2\in A_{[\omega(f_1)]}=A_{[\omega(f_2)]}$ such that $h_1=h_2$ if $\omega(f_1)=\omega(f_2)$.   
\end{enumerate}
In particular,   
 $B_{x,1}\cap o_{f_1}^{\mathcal B_1}   \alpha_e( A_e ) (o_{f_1}^{\mathcal B_1})^{-1}\neq 1$  and  $B_{x,2}\cap o_{f_2}^{\mathcal B_2}   \alpha_e( A_e ) (o_{f_2}^{\mathcal B_2})^{-1}\neq 1.$ 
\end{lemma}
\begin{remark}
As $F_i$  is elementary it holds $o_{f_1}^{\mathcal B_i}=o_{f_2}^{\mathcal B_i}\in A_{[x]}$. Moreover, $t_{f_1}^{\mathcal B_i}=t_{f_2}^{\mathcal B_i}$ if $\omega(f_1)\neq \omega(f_2)$, that if $F_1$ and $F_2$ are of type IA.
\end{remark}}

\begin{proof}
The proof broadly follows the proof of \cite[Lemma 6]{W2}. The main difference is that peripheral moves in general do not commute with tame  auxiliary moves of type A2 that are based on the same vertex.

\smallskip

\noindent{\emph{Folds of type IIA.}} We first deal with the case in which both folds $F_1$ and $F_2$ are of type IIA. Assume that the label of $f=g$ in $\mathcal B_i$ is $(a_i, e, b_i)$.

Assume first that $x$ is exceptional (and hence $y$ is peripheral). In this case  $B_y=1$ as $F_1$ and $F_2$ induce markings on the resulting $\mathbb A$-graphs.  It is not hard to see that $F_1$ commutes with all elementary moves that affect neither   $x$ nor the initial elements of edges staring at $x$. $F_1$ also commutes with all auxiliary moves of type A0 that are based on  $x$ and with all auxiliary moves of type A1 that are based on edges  starting at $x$. Thus we can assume that $\mathbf B_2$ is obtained from $\mathbf B_1$ by tame auxiliary moves of type A2  based on edges  starting at $x$  and Nielsen/peripheral moves based on $x$. Therefore $\mathbf B_1$ and $\mathbf B_2$ differ only at $x$ and $B_{x, 1}=B_{x,2}$.  It follows from \cite[Lemma~4.17(i)]{Dut} that $\mathcal P_{x, 1}$ and $\mathcal P_{x, 2}$ are equivalent. Since $F_1$ and $F_2$ are tame  and  elementary it follows that  one of the following occurs:
\begin{enumerate}
\item[(El.1)]   $\mathcal P_{x,i}=(T_{x, i}\oplus (\gamma_{f, i}'), P_{x, i})$ such that  $B_{x,i}$ splits as
$$B_{x,i}=\langle T_{x,i}\rangle \ast \langle \gamma_{f,i}'\rangle\ast \langle\gamma_{f_1,i}\rangle\ast \ldots\ast \langle \gamma_{f_n, i}\rangle\le A_{[x]}$$
where $\gamma_{f,i}'$ is a generator of  
 $a_i\alpha_e(A_e)a_i^{-1}\cap B_{x, i}$  
and $P_{x,i}=(\gamma_{f_1,i}, \ldots, \gamma_{f_n, i})$. The definition of folds says that   $x$ stays exceptional in $\mathbf B_i'$ and  
 $$\mathcal P_{x, i}'=(T_{x, i}, (\gamma_{f, i}')\oplus   P_{x, i}).$$  
 Proposition~\ref{prop:02}  implies that $\mathcal P_{x,i}'$  is non-critical and  of simple type.    It follows from  Lemmas~\ref{lemma:equivalencesimpletype} (as $P_{x,1}$ and $\mathcal P_{x,2}$ are equivalent) that $\mathcal P_{x, 1}'$ and $\mathcal P_{x, 2}'$ are equivalent.  Lemma~4.17(ii) of \cite{Dut} implies that  $\mathbf B_1'$ and  $\mathbf B_2'$ are equivalent.

\item[(El.2)] $B_{x,i}=\langle T_{x,i}\oplus P_{x,i}\rangle\le A_{[x]}$ is of finite index and  the  elements in $P_{x,i}$ correspond to all but one boundary component of $\mathcal O_{x}'$. In this case $\mathbf B_1'$ and $\mathbf B_2'$ are equivalent as  the new  marked $\mathbb A$-graph of orbifold type that emerges differ  by auxiliary moves and by replacing its generating tuple by a Nielsen equivalent tuple. 
\end{enumerate}

We now deal with the case that $x$ is peripheral. As $F_1$ and $F_2$  yield marked $\mathbb A$-graphs  we conclude that $y$ is exceptional in $\mathbf B_1$ (and therefore also in $\mathbf B_2$). In this case all elementary moves commute with the fold. This is easily verified unless $\mathbf B_2$ is obtained from  $\mathbf B_1$ by a tame auxiliary move of type A2 based on $f^{-1}$. Thus assume that  $\mathbf B_2$ is obtained from $\mathbf B_1$ by a tame auxiliary move of type A2 based on $f^{-1}$ that replaces its label $(a_1, e, b_1)$ by  $(a_2,e, b_2)=(a_1, e, b_1\beta)$  for some $\beta\in T_{y,1}^{\pm 1}\oplus P_{y,1}^{\pm 1}$. Thus 
$$\mathcal P_{y,1}'=(T_{y,1}, (\gamma')\oplus P_{y,1}) \  \  \text{ and }  \ \ \mathcal P_{y,2}'=(T_{y,2}, P_{y,2})=(T_{y,1}, (\beta^{-1}\gamma'\beta)\oplus P_{y,1})$$ 
where $\gamma'$ is the peripheral element added by $F_1$. The A2 move based on $f^{-1}$    is  tame when applied to $\mathbf B_1'$ since $\beta$ lies in $ T_{y, 1}^{\pm 1}\oplus (\gamma')\oplus P_{y,1}.$  This move  replaces  $\mathcal P_{y,1}'$ by $\mathcal P_{y,2}'$. Thus $\mathbf B_1'$ and $\mathbf B_2'$ are equivalent.

\smallskip

\noindent\emph{Folds of type IA/IIIA.} Assume now that  $F_1$ and $F_2$ are of type IA or of type IIIA.  We will give the argument  in the case that $F_1$ and $F_2$ are  of type IIIA  and leave the similar case of folds of type IA to the reader.

We argue in the case that $x$ is peripheral, the case of non-peripheral $x$ is simpler as both $B_{f_1}$ and $B_{f_2}$ are trivial by the fact that $F_1$ and $F_2$ yield marked $\mathbb A$-graphs.  Observe that $y:=\omega(f_1)=\omega(f_2)$ is exceptional as    $F_1$ and $F_2$ yield marked  $\mathbb A$-graphs. Hence,  both 
  $\mathcal P_{y,1}=(T_{y, 1}, P_{y,1}) $  and $  \mathcal P_{y, 2}=(T_{y, 2}, P_{y,2})$  are non-critical and  of simple type. Assume that the label of  $f_1$ (resp.~$f_2$)   in $\mathcal B_1$ is   $(a, e, b_1)$ (resp.~$(a, e, b_2)$). We can also assume that  $B_{f_2}=1$ as $F_1$ is tame.

All elementary moves that  do not affect  $f_1\cup f_2$ clearly  commute with  $F_1$ and $F_2$.  Thus we can restrict our attention to those moves that  do affect $f_1\cup f_2$.  

Auxiliary moves of type A0 that are based on $x$ or on  $y$ as well as  peripheral and  Nielsen moves that are based on $y$ also commute with $F_1$ and $F_2$. If $\mathbf B_2$ is obtained from $\mathbf B_1$  by tame auxiliary moves of type A2  based on  $f_1^{-1}$ and $f_2^{-1}$    then $$\mathcal P_{y,1}'=(T_{y,1}\oplus (b_1^{-1}b_2), P_{y,1}) \ \text{ and  } \ \mathcal P_{y,2}'=(T_{y,2}, P_{y,2})=(T_{y,2}\oplus (\beta_1^{-1}b_1^{-1}b_2\beta_2), P_{y,2}'),$$    
where $\beta_1, \beta_2\in  T_{y,1}^{\pm1}\oplus   P_{y,1}^{\pm 1}$ and  where $P_{y,2}'=P_{y,1}$ if $B_{f_1,1}=1$ and $P_{y,2}'$ is obtained from $P_{y,1}$ by conjugating the peripheral element associated to $f_1^{-1}$  by $\beta_1^{-1}$ if $B_{f_1, 1}\neq 1$. In both cases one easily checks that these partitioned tuples are  equivalent  which implies that $\mathbf B_1$ and $\mathbf B_2$ are equivalent.

Therefore we can assume that $\mathbf B_2$ is obtained from $\mathbf B_1$ by auxiliary moves of type A2 and auxiliary moves of type A1  based on  $f_1$ and $f_2$ as described in Fig.~\ref{fig:equivfoldIIIA}, where $\alpha_1, \alpha_2\in B_x$ and $c, d \in A_e$.  
As fold $F_2$ is elementary we have $\alpha_1 a \alpha_e(c^{-1})=\alpha_2 a \alpha_e(d^{-1}).$  Therefore  $\alpha_2^{-1}\alpha_1= a\alpha_e(cd^{-1})a^{-1}, $  and hence   $a\alpha_e(cd^{-1})a^{-1}\in B_x$.
\begin{figure}[h!]
\begin{center}
\includegraphics[scale=1]{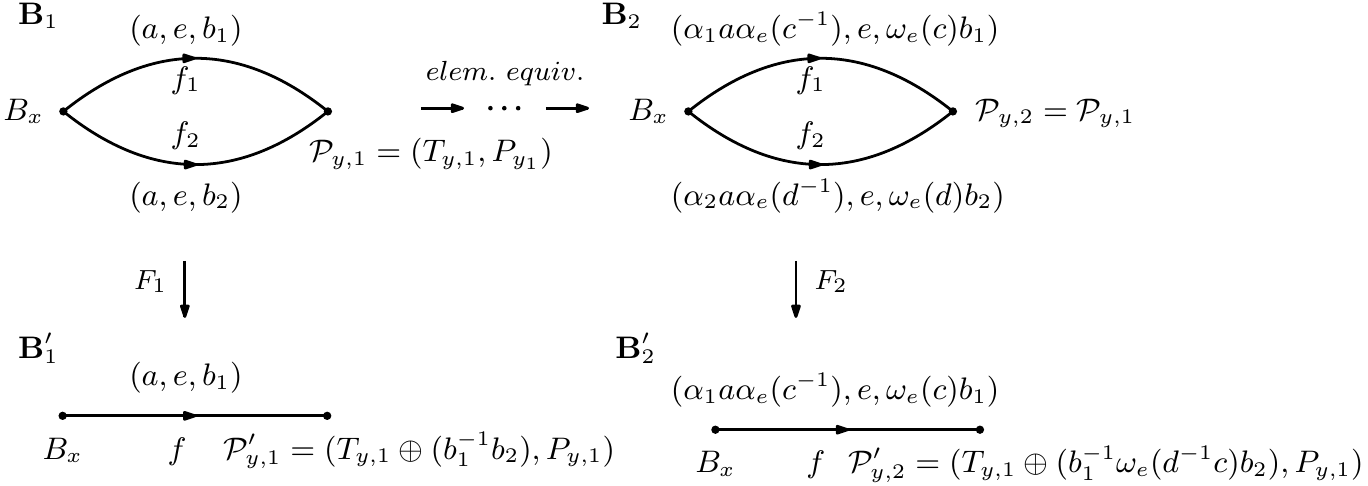}
\caption{Elementary moves  carrying $\mathbf B_1 $ onto $\mathbf B_2$.}{\label{fig:equivfoldIIIA}}
\end{center}
\end{figure}

We will show that $\mathbf{B}_1$ is equivalent to $\mathbf{B}_2$  if either    $c=d$ or $B_{f_1}\neq 1$. In fact,  if $c=d$,  then the assertion is trivial  as in this case the fact that $F_2$ is elementary  implies that   $\alpha_1=\alpha_2$, and hence   the A1 and the A2 moves commute with the fold $F_1$.

If $B_{f_1}\neq 1$ then  condition (4)  of marked $\mathbb A$-graphs  guarantees that  $B_{f_1}=\alpha_e^{-1}(a^{-1}B_xa).$ 
Thus,  $cd^{-1}$ lies in $B_{f_1}$  and therefore  $b_1^{-1}\omega_e(cd^{-1})b_1\in B_y.$ 
 Thus     $$\mathcal P_{y,1}=(T_{y, 1}\oplus (b_1^{-1}b_2), P_{y,1}) \ \ \text{  and   } \ \  \mathcal P_{y,2}=(T_{y,1} \oplus (b_1^{-1}\omega_e(d^{-1}c)b_2), P_{y,1})$$ are equivalent. Therefore  $\mathbf B_2$ is obtained from $\mathbf B_1$  by  first  applying auxiliary moves of type A1 based  $f_1$ and $f_2$ with  $c, d\in A_e$, then applying tame auxiliary moves of type A2 which are also based on  $f_1 $ and $f_2$ with  $\alpha_1,\alpha_2\in B_x$ respectively, and finally applying  finitely many Nielsen and peripheral moves  based on $y$ that replace the element $b_1^{-1}b_2$  by   $$b_1^{-1}\omega_e(d^{-1}c)b_2= b_1^{-1}\omega_e(d^{-1}c)b_1\cdot b_1^{-1}b_2.$$    
\end{proof}

\begin{corollary}{\label{corollary_edges2}} Assume that $F_1$ is of  type IA or IIIA and induces a marking on $\mathcal B_1'$.   If  $B_{f_1}=1$, then there exists a tame marked $\mathbb A$-graph $\mathbf B$  such that the following hold:
\begin{enumerate}
\item $\mathbf B $ is equivalent to $\mathbf B_2$ and is obtained from $\mathbf B_2$ in the following way:
\begin{enumerate}
\item By a (tame) auxiliary move of type A2 affecting the initial element of $f_1$,
 \item followed by an auxiliary move of type A1 applied to $f_1$ 
\item and in the case that $F_1$ is  of type IA, followed by an auxiliary move of type A0 applied to either $\omega(f_1)$ or $\omega(f_2)$. 
\end{enumerate}
\item A tame elementary fold  based on $f_1$ and $f_2$ can be applied to $\mathbf B$ such that the resulting marked $\mathbb A$-graph $\mathbf B'$ is equivalent to $\mathbf B_1'$.
\end{enumerate}
$$\begin{tikzcd}
\mathbf B_1   \arrow[d, "F_1"']   \arrow{r}[description]{\cdots}{\substack{elem. \\ moves }} &  \mathbf B_2 \arrow{r}[description]{\cdots}{\substack{elem. \\ moves }}  & \mathbf B \arrow{dl}{\substack{tame \  elem. \\ fold}} \\
 \mathbf B_1'  &     \arrow{l}[description]{\cdots}{\substack{elem. \\ moves }}  \mathbf B'      &   
\end{tikzcd}
$$
\end{corollary}
\begin{proof}
The proof follows immediately from the proof of the previous lemma. 
\end{proof}

\begin{lemma}\label{lemma:ruledoutcases}
Suppose that at least one of the folds  $F_1$ and  $F_2$ is good. Then the following configurations cannot  occur: 
\begin{enumerate}
\item[(1)]   $F_1$ and $F_2$ are of type IA/IIIA,   $f_1=g_1^{\varepsilon}$ and $f_2\neq g_2^{\varepsilon}$ with $B_{f_2}\neq 1\neq B_{g_2}$ for some $\varepsilon\in \{\pm 1 \}$ (after exchanging the edges $f_1$ and $f_2$ if necessary).

\item[(2)] $F_1$ is of type IA/IIIA  with $B_{f_2}\neq 1$ and $F_2$ is of type IIA  with $g\in \{f_1^{\pm 1}\}$ or vice versa, i.e.~$F_2$ is of type IA/IIIA with $B_{g_2}\neq 1$  and $F_1$ is of type IIA with $f\in \{g_1^{\pm 1}\}$.

\item[(3)] $F_1$ and $F_2$ are of type IIA with $g=f^{-1}$.  
\end{enumerate}
\end{lemma}
\begin{remark}
As the underlying graph of the  graph of groups $\mathbb A$ has no loop edges, the configurations $f_1=g_1^{\varepsilon}$ and $f_2=g_2^{-\varepsilon}$  with $\varepsilon\in \{\pm 1\}$ cannot occur.
\end{remark}

\begin{proof}
\noindent{(1)}  Assume first that   $f_1=g_1$ and $f_2\neq g_2$ with $B_{f_2}\neq 1\neq B_{g_2}$. Since $f_2$ can be folded with $f_1$ and $f_1$ can be folded with $g_2$  we conclude that   $f_2$ can be folded with $g_2$. But this  contradicts  \cite[Lemma~4.12(1)]{Dut} as  $B_{f_2}\neq 1\neq B_{g_2}$  and $\mathbf B_1$ is tame.

Assume  now that   $g_1=f_1^{-1}$, $g_2\neq f_2^{-1}$  and $B_{f_2}\neq 1\neq B_{g_2}$. We will get a contradiction  by showing  that both  $F_1$ and $F_2$ are bad. After exchanging the roles of $F_1$ and $F_2$,  if necessary, we can assume that $x\in VB$ is   peripheral.   Since  $B_{f_2}\neq1\neq B_{g_2}$ and $[g_2^{-1}]=[g_1^{-1}]=[f_1]=[f_2]\in EA$  it follows from condition (5) of marked $\mathbb A$-graphs that  $\omega(g_2)\neq x=\omega(g_1)$. In particular $F_2$ is of type IA.

We  give a complete argument in the case that $F_1$ is of type IIIA, the similar and slightly easier case of a fold of type IA is left to the reader.  Denote $y:=\omega(f_1)=\omega(f_2)$. As $F_1$ is elementary, we can assume that  $f_i$ ($i=1,2$) has label $(a, e, b_i)$  in $\mathcal B_1$.  Assume that in $\mathcal B_1$ the edge  $g_i$ ($i=1, 2$) has label $(c_i, e^{-1}, d_i)$.   Observe that,  as $g_1=f_1^{-1}$,   we have $$(c_1, e^{-1}, d_1)=(b_1^{-1}, e^{-1}, a^{-1}).$$

The badness of  $F_2$ follows immediately  from the fact that $x$ and $\omega(g_2)$ are peripheral and  of orbifold type  which puts  $F_2$ in  case FoldIA(1)(1). 

It remains to show that $F_1 $ is bad.  If $y$ is of orbifold type this is immediate since $F_1$ falls into case FoldIIIA(2)(1). Thus we may assume $y$ is  exceptional.  This implies that $F_1$ induces a marking on $\mathcal B_1'$. Observe that the peripheral tuple $P_{y, 1}$   has the form 
$$(\gamma_{f_2^{-1}}, \gamma_{g_2})\oplus \bar{P}_{y, 1}$$
where $\gamma_{f_1^{-1}}$ and $\gamma_{g_2}$ are the peripheral elements associated to $f_2^{-1}$ and $g_2$ respectively. By definition,  $\gamma_{f_2^{-1}}=b_2^{-1} \omega_e(c_{f_2^{-1}})b_2 $  and $ \gamma_{g_2}=c_2 \omega_e(c_{g_2})c_2^{-1}, $   
  where $c_{f_2^{-1}}, c_{g_2}\in A_e$.  Since   $g_1=f_1^{-1}$ and $g_2$ can be folded in $\mathcal B_1$ (not necessarily  an elementary fold),   there are $\beta\in B_{y, 1}$ and  $c\in A_e$ such that   
$$c_2=\beta o_{f_1^{-1}}^{\mathcal B_1} \alpha_{e^{-1}}(c)= \beta b_1^{-1} \omega_e(c).$$
Hence $\gamma_{g_2}=\beta b_1^{-1} \omega_e(c_{g_2}) b_1\beta^{-1}$. 
By definition,   $F_1$  replaces   $\mathcal P_{y, 1}$  by  
$$\mathcal P_{y, 1}' = ( T_{y,1} \oplus (b_1^{-1}b_2) ,  (\gamma_{g_2}, \gamma_{f_2^{-1}})\oplus \bar P_{y,1}).$$
After  applying an elementary move on $\mathcal P_{y,1}'$ that conjugates $\gamma_{g_2}$ by   $b_2b_1^{-1}\beta^{ -1}$,   we obtain   the partitioned tuple
$$(T_{y,1}\oplus (b_1^{-1}b_2) , (b_1^{-1}\omega_{e}(c_{g_2}) b_1, b_1^{-1}\omega_e(c_{f_2^{-1}})b_2)\oplus \bar P_{y,1 }  ) $$
which clearly  folds peripheral elements. Therefore $\mathcal P_{y,1}'$  folds peripheral elements and hence $F_1$ is bad.

\smallskip

\noindent {(2)} After exchanging the roles of $F_1$ and $F_2$ if necessary,   we can assume that $F_1$ is of type IA/IIIA and $F_2$ is of type IIA. Note that  $g =f_1$ cannot occur because this violates the  tameness of  $F_2$ as $g$ can be folded with  $f_2$ and  by hypothesis $B_{f_2}\neq 1$. Thus assume  $g=f_1^{-1}$. We will show that both folds  $F_1$ and $F_2$ are bad which contradicts  the hypothesis that at least one of them is good. 

\noindent{\em Case 1:} $F_1$ is of type IA.  Assume that in $\mathcal B_1$  the labels of $f_1$ and $f_2$ are equal to  $(a, e, b)$  and  that  in $\mathcal B_2$ they are   $(a_1, e, b_1)$ and $(a_2, e, b_2)$ respectively.  We consider two cases according to the type of $x$.

\noindent{\emph{Subcase a:}} $x$ is non-peripheral. Thus  $y_1 $ and $y_2$ are peripheral.  Observe that  $B_{y_1}\neq1\neq B_{y_2}$,  and hence  $y_1$ and $y_2$   are of orbifold type. Thus $F_1$ is bad  as it falls into  case  FoldIA(1)(2).  

We show that $F_2$ is  bad. If $x$ is of orbifold type, the claim is trivial as  $F_2$ adds an element to a vertex of orbifold type, i.e.~$F_2$ is in case FoldIIA(El.0)(1). Thus assume that $x$ is exceptional. The peripheral tuple $P_{x, 2}$ contains the peripheral element $\gamma_{f_2}$ associated to $f_2$ which, by definition, is given by
 $\gamma_{f_2}=a_2\alpha_e(c_{f_2})a_2^{-1}$  for some $c_{f_2}\in A_e$. The fold $F_2$ replaces $\mathcal P_{x, 2}$ by
$$\mathcal P_{x, 2}'=( T_{x, 2} ,  (\gamma')\oplus P_{x, 2})=( T_{x, 2} , (\gamma_{f_2}, \gamma')\oplus  \bar P_{x, 2})$$ 
where $\gamma'=  a_1\alpha_e(c')a_1^{-1}$ with   $c'\in A_e$ such that $\langle c'\rangle=\omega_e^{-1}(B_{y_1, 2})$. 
Since $f_1$ and $f_2$ can be folded in $\mathcal B_2$ (not necessarily an elementary fold), there are $\beta \in B_{x, 2}$ and $d\in A_e$ such that $a_1=\beta a_2 \alpha_e(d)$. Thus   $\gamma'= \beta a_2 \alpha_e(c')\beta^{-1} a_2^{-1}.$  Therefore, after applying an elementary move on $\mathcal P_{x, 2}'$  that conjugates $\gamma'$ by $\beta$ we obtain
$$(T_{x, 2}, ( \underbrace{a_2\alpha_e(c_{f_2})a_2^{-1}}_{ \gamma_{f_2}} , a_2\alpha_e(c')a_2^{-1})\oplus \bar P_{x, 2}).$$
Therefore $\mathcal P_{x, 2}'$ folds peripheral elements and hence $F_2$ is bad.

\smallskip

\noindent{\em Subcase b:} $x$ is peripheral. Since   $B_{f_2}\neq 1$ it follows that   $B_{x}\neq 1$  and therefore $x$  of orbifold type.  Thus  $F_2$ is bad because it adds a peripheral element to a vertex of orbifold type,  i.e.~$F_2$ is as in case  FoldIIA(El.1)(1) or FoldIIA(El.2)(1). 

It remains to show that $F_1$ is bad.  If $y_2$ is of orbifold type,  $F_1$ is bad as we are in case FoldIA(2)(1b).  Thus assume that $y_2$ is exceptional.    The fact that  $F_2$ is tame,  implies that $y_1:=\omega(f_1)$ is  exceptional. The fact that $F_2$ can be applied to $\mathcal B_2$ clearly implies that  $$ \langle  a_e^z \rangle= \omega_e^{-1}( bB_{y_1,1}  b^{-1})$$
for some non-zero integer $z$,  where $a_e$ denotes a generator of $A_e$. By definition,  $F_1$ replaces  $\mathcal P_{y_1, 1}$ and $\mathcal P_{y_2, 1}$ by  
 $$\mathcal P_{y',1} ' =(T_{y_1} \oplus T_{y_2} , P_{y_1}\oplus P_{y_2})$$ 
where $y'\in VB_1'$ denotes the image of $y_1$ and $y_2$  under $F_1$ in $\mathcal B_1'$. The peripheral element associated to $f_2^{-1}$ (in $\mathcal B_1$)  is equal to $b^{-1}\omega_e(a_e^w)b$ for some non-zero integer $w$.  Therefore,  by Proposition~\ref{lemma:IIAbad}
,   $$\mathcal P'':=(T_{y_1}, (\gamma_{f_2^{-1}})\oplus  P_{y_1})$$ is critical.   Since  $\mathcal P''$ is a partitioned  subtuple of $\mathcal P_{y, 1}'$, the latter is critical. Therefore  $F_1$ is bad.   
 
 \smallskip
 
\noindent{\emph{Case 2.}} $F_1$ is of type IIIA.  As $F_1$ is elementary, we can assume that in $\mathcal B_1$ the edges  $f_1$ and $f_2$ have labels of type $(a, e, b_1)$ and $(a, e, b_2)$ respectively. Observe that if $y$ is peripheral,  then   $g=f_1^{-1}$ and $f_2^{-1}$ can be folded since they are of same type. i.e.  $[f_2^{-1}]=[f_2^{-1}]  \in EA.$  But this contradicts the hypothesis  that $F_2$ is tame since $B_{f_2}\neq 1$.   Thus  $y$ is exceptional and hence  $x$ is peripheral. In this case $F_2$ is bad as it  affects a vertex of orbifold type. Lemma~\ref{lemmatypesIIA}  implies that  $\mathcal P_{y, 1}$ is non-critical  of simple type as $\mathbf B_1$ and $\mathbf B_2$ are tame.  

We claim that $F_1$ is also bad. The fact that $F_2$ can be applied to $\mathcal B_2$ clearly  implies that there is a non-zero integer $z$ such that $$\langle b_1^{-1}\omega_e(a_e^z)b_1\rangle=b_1^{-1}\omega_e(A_e)b_1\cap B_{y, 1}.$$  
Put $\gamma':=b_1^{-1}\omega_e(a_e^z)b_1$. 

If $B_{y, 1}=\langle T_{y, 1}\oplus  P_{y, 1}\rangle$ has finite index in $A_{[y]}$,  then the badness of $F_1$  follows from  Lemma~\ref{lem:critical1}.  Thus assume that  $B_{y,1}$ has infinite index in $A_{[y]}$. In this case  $\mathcal P_{y, 1}$ is equivalent to 
$$\mathcal P'=(T'\oplus (\gamma'), (\gamma_{f_2^{-1}})\oplus  P') $$
such that 
$$B_{y,1}=\langle T'\rangle\ast \langle \gamma'\rangle \ast \langle \gamma_{f_2^{-1}}\rangle \ast \langle \gamma_1'\rangle \ast \ldots \ast \langle \gamma_r'\rangle
$$ 
where $P'=(\gamma_1', \ldots, \gamma_r')$.   
By definition,  $F_1$ replaces  $\mathcal P_{y, 1}$  by  $$(T_{y, 1} \oplus (b_1^{-1}b_2) , P_{y,1})$$ 
which is equivalent to $$(T'\oplus (\gamma', b_1^{-1}b_2),(\gamma_{f_2^{-1}})\oplus  P' ).$$
But $\gamma_{f_2^{-1}}=b_2^{-1}\omega_e(a_e^w)b_2$ for some non-zero integer  $w$. Thus $\mathcal P_{y, 1}'$  has an obvious relation (if $z$ is not a multiple of $w$) or $\mathcal P_{y, 1}'$  is reducible (if $z$ is a multiple of $w$). Therefore  $F_1$ is bad. 
  
\smallskip

\noindent{(3)} We can assume without loss of generality that $x$ is peripheral. Let $(a, e, b)$ be the label of $f$ in $\mathcal B_1$.  Since a fold along $f$ is possible we have    
 $B_{x}\neq 1,$  
and hence   $x$ is of orbifold type.  On the other hand, the fact that a tame fold of type IIA based on $f^{-1}$ can be applied to $\mathcal B_2$  implies that $y$  is  exceptional and that there is a non-zero integer  $z$ such that  
$$\langle b^{-1}\omega_e(a_e^z)b\rangle =B_{y, 1} \cap b^{-1}\omega_e(A_e)b.$$
The badness of $F_2$ follows easily since   $F_2$  adds an element to $B_x$.  On the other hand,     $F_1$ replaces   $\mathcal P_{y, 1}$  by 
$$\mathcal P_{y, 1}'=(T_{y, 1}, (b^{-1}\omega_e(a_e^w)b)\oplus P_{y, 1} )$$
for some $w\in \mathbb Z$ such that $\langle a_e^w\rangle=\alpha_e^{-1}(A_e)$.    Proposition~\ref{lemma:IIAbad}  implies that $\mathcal P_{y, 1}'$ is critical which puts $F_1$ in case FoldIIA(El.0)(2.c) and so $F_1$ is   bad.
\end{proof}

The following  configurations \textbf{(D1)-(D4)} play an important role in the proofs of  Proposition \ref{lemma:fork1} and Proposition~\ref{lemma:fork2} because   they cannot be ``normalized"  in the sense of   Lemma~\ref{lemma5cases} below.    
\begin{enumerate}
\item[\textbf{(D1)}] $F_1$ and $F_2$ are of type IA/IIIA and  $\{f_1,f_2\}=\{g_1,g_2\}$  or $F_1$ and $F_2$ are of type IIA and $f=g$.
 
\item[\textbf{(D2)}] $F_1$ and $F_2$ are folds of type IA  and  after  exchanging   $f_1$ and $f_2$ if necessary,  we may assume  $y_1:=\omega(f_1)=\omega(g_1)$ and $y_2:=\omega(f_2)=\omega(g_2)$. All possible configurations are described in Figure~\ref{fig:3optionscase(iii)}.
\begin{figure}[h]
\centering
\includegraphics[scale=1]{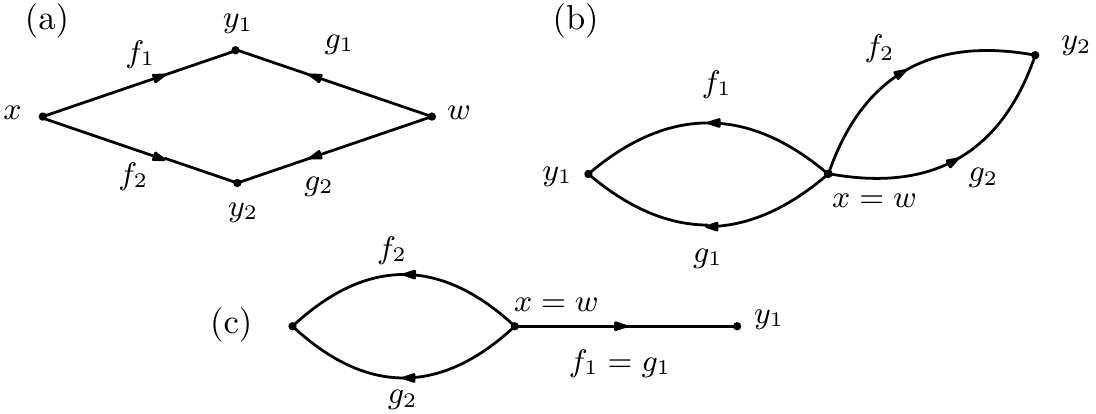}
\caption{The three configurations that can occur in case (2).}
\label{fig:3optionscase(iii)}
\end{figure} 

\item[\textbf{(D3)}]  $F_1$ and $F_2$ are folds of type IIIA such that,     after exchanging $g_1$ and $g_2$ if necessary,  we have $f_1=g_1^{-1}$ and $f_1=g_2^{-1}$.

\item[\textbf{(D4)}] $F_1$ and $F_2$ are of type IIA with $f\neq g$ such that   $x=w$ is non-peripheral,  and $F_1$ and $F_2$ satisfy the elementary condition (El.1) and  correspond to distinct boundary components of the orbifold  $\mathcal O_{x}'$  corresponding to  $B_{x,1}=\langle T_{x,1}\oplus P_{x,1} \rangle\le A_{[x]}$   
and $\mathcal P_{x, 1}$ is not equivalent to a partitioned tuple of the form
 $(\bar{T}\oplus (\bar{\gamma}_f, \bar{\gamma}_g), \bar{P})$  
such that  
$$B_{x,1}=\langle \bar{T}\rangle \ast \langle \bar{\gamma}_f\rangle\ast\langle\bar{\gamma}_g\rangle\ast \langle \bar{\gamma}_1 \rangle\ast \ldots   \ast \langle \bar{\gamma}_r \rangle$$
where   $\bar{\gamma}_f$ and $\gamma_f'$ (resp.~$\bar{\gamma}_g$ and $\gamma_g'$) corresponding to the same boundary of $\mathcal O_x'$ and $\bar{P}=(\bar{\gamma}_1 , \ldots, \bar{\gamma}_r)$, see Figure~\ref{fig:commnonel}. 
\begin{figure}[h!]
\centering
\includegraphics[scale=1]{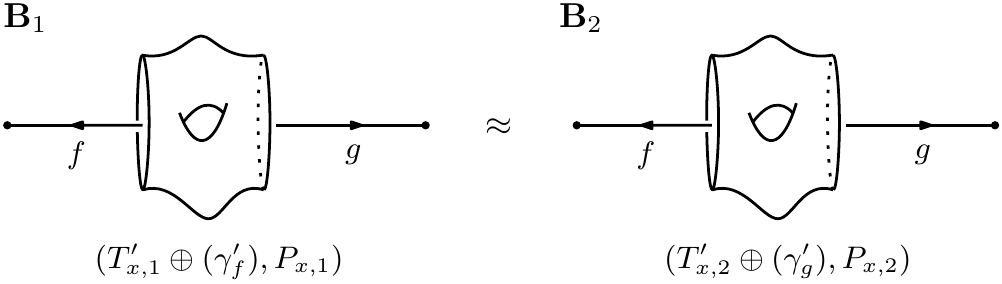}
\caption{$F_1$  and $F_2$ correspond to distinct boundary components of  $\mathcal O_x'$.}
\label{fig:commnonel}
\end{figure}
 
\end{enumerate}

In the set-up it is desirable to have $\mathbf B_1=\mathbf B_2$ so that $F_1$ and $F_2$ are tame elementary folds that are applied to the same marked $\mathbb A$-graph. The following lemma tells us  that, up to equivalence, we can almost always assume that this occurs  without substantially  changing the properties of $F_1$ and $F_2$ unless one of the configurations \textbf{(D1)}-\textbf{(D4)}  occurs.

\begin{lemma}\label{lemma5cases}  
Assume that there is $i\in \{1, 2\}$  such that $F_i$ is good.  If  none of the configurations \textbf{(D1)}-\textbf{(D4)} occurs,  then  there exists a marked $\mathbb A$-graph  $\mathbf B$  that  is  equivalent to $\mathbf B_i$ ($i=1, 2$)  and  admits tame elementary folds $\bar F_1$ and $\bar F_2$  yielding $\mathbb A$-graphs $\bar{\mathcal B}_1$ and $\bar{\mathcal B}_2$  such that the following hold:
\begin{enumerate}
\item[(a)] If $F_j$ induces a marking  on $\mathcal B_j'$, then $\bar F_j$ induces a marking on $\bar{\mathcal B}_j$ and the marked $\mathbb A$-graphs  $\bar{\mathbf B}_j$ and  $\mathbf B_j'$ are equivalent.

\item[(b)] If $F_j$  does not induce a marking on $\mathcal B_j$, then $\bar F_j$ does not induce a marking on $\bar{\mathcal B}_j$.
\end{enumerate}
In particular,  if $F_i$ is good (resp.~bad),   then $\bar F_i$ is good (resp.~bad). 
\end{lemma}
\begin{remark}{\label{remark:donotinduce}}
It follows  immediately from the description of folds that  $F_1$ (resp.~$F_2$) does not induce a marking on $\mathcal B_1'$  (resp.~$\mathcal B_2'$) if  and only if  $F_1$ (resp.~$F_2$)  changes the group of some vertex of orbifold type. Therefore, if $\mathbf B$ is equivalent to $\mathbf B_1$ (resp.~to $\mathbf B_2$) and if $\bar F_1$ (resp.~$\bar F_2$) is a tame  elementary fold that  that can be applied to $\mathbf B$   and  is based on the same vertex and affects the same edges as $F_1$ (resp. $F_2$),  then $F_1$     does not induce a marking on $\mathcal B_1'$ if and only if $\bar F_1$  does not  induce  a marking on the resulting $\mathbb A$-graph.
\end{remark} 
\begin{proof}
We consider all possible configurations  for $F_1$ and $F_2$. The main idea is to turn the  not necessarily elementary fold that identifies  $f_1$ and $f_2$ in $\mathcal B_2$ (or  that is  based  on  $f$ in  the case $F_1$ is of type IIA)  into an elementary fold  without affecting the labels of  $g_1$  and $g_2$ in $\mathcal B_2$ (or of $g$ in the case $F_2$ is of a type IIA).

\smallskip
 
\noindent{\bf Case 1:} $F_1$ and $F_2$ are of type IA/IIIA. As $F_1$ and $F_2$ are tame,  we can assume that $B_{f_1}=1=B_{g_1}$. Moreover, after  interchanging $F_1$ and $F_2$  if necessary, we may assume that  if $B_{f_2}\neq 1$  then  $B_{g_2}\neq 1$.   

Assume that  $f_1\notin \{g_1^{-1}, g_2^{-1}\}$.     Corollary~\ref{corollary_edges2} implies that there is a marked $\mathbb A$-graph $\mathbf B$  such that the following holds: 
\begin{enumerate} 
\item[(1)] $\mathbf B$ is equivalent to $\mathbf B_2$ and is obtained from $\mathbf B_2$ in the following way:
\begin{enumerate}
\item By a tame auxiliary move of type A2 applied to $f_1$, 

\item followed by an auxiliary move of type A1 based on $f_1$,  

\item and in the case $F_1$ is of type IA, followed by an auxiliary move of type A0 applied to either $\omega(f_1)$ or to $\omega(f_2)$. 
\end{enumerate}

\item[(2)]   A tame elementary fold $\bar F_1$   based on $f_1$ and $f_2$  can be applied to $\mathcal B$ yielding an $\mathbb A$-graph $\bar{\mathcal B}_1$  such that:

\begin{enumerate}
\item if $F_1$ induces a marking on $\bar{\mathcal B}_1$,  then $\bar F_1$ induces a marking on $\bar{\mathcal B}_1$  and  the resulting marked $\mathbb A$-graphs $\mathbf B_1'$ and $\bar{\mathbf B}_1$ are equivalent.

\item If $F_1$ does not induce a marking on $\mathcal B_1'$,  then $\bar F_1$ does not induce a marking on $\bar{\mathcal B}_1$. 
\end{enumerate} 
\end{enumerate}

Since configuration \textbf{(D2)} is excluded, we conclude that if $F_1$ is of type IA, then  there is $k\in \{1, 2\}$ such that $\omega(f_k)\notin\{\omega(g_1), \omega(g_2)\}$.  Thus  in going from $\mathbf B_2$ to $\mathbf B$  we can arrange the auxiliary moves (a)-(c) so that  the labels of $g_1$ and $g_2$ are not affected,  and hence  the fold $\bar F_2$ that identifies $g_1$ and $g_2$ in $\mathcal B$  is also  elementary.

By  Remark~\ref{remark:donotinduce},  $F_2$   induces a marking on $\mathcal B_2'$  if  and only if $\bar F_2$  induces a marking on $\bar{\mathcal B}_2$.  In the affirmative case, Lemma~\ref{lemma_edges1}  implies that the the marked $\mathbb A$-graph $\bar{\mathbf B}_2$, which is obtained from $\mathbf B$ by $\bar F_2$,  is equivalent to $\mathbf B_2'$.

Now assume that  $f_1\in\{g_1^{\pm 1}, g_2^{\pm 1}\}$. We claim that   $f_2\notin \{g_1^{\pm 1}, g_2^{\pm 1}\}$ and $B_{f_2}=1$. Therefore   we can apply the same argument from the previous paragraph with  $f_2$ playing the role of $f_1$.

If $f_1=g_i$ for some $i$,  then   $f_2\neq g_j$  ($i\neq j\in \{1, 2\}$) because configuration \textbf{(D1)} is excluded. If $f_1=g_i^{-1}$, then $f_2\neq g_j^{-1}$ ($i\neq j\in \{1, 2\}$)  because configuration \textbf{(D3)} is excluded.  This proves the first part of the claim.  

It remains to show that $B_{f_2}=1$ if $f_1\in \{g_1^{\pm 1}, g_2^{\pm1}\}$.  The claim is trivial if $f_1=g_2$ or  $f_1=g_2^{-1}$ because we assumed that $B_{f_2}=1$ if $B_{g_2}=1$.  Thus assume  that $f_1=g_1$. Since  $f_2$ and $g_2$  can be  folded in  $\mathcal B_2$, it follows from \cite[Lemma~4.12]{Dut} that $B_{f_2}=1$ or $B_{g_2}=1$. Therefore  $B_{f_2}=1$. Now assume that  $f_1=g_1^{-1}$. The previous lemma and the hypothesis  that at least one of the folds  $F_1$ or $F_2$ is good, implies that $B_{f_2}=1$. 

\smallskip

\noindent{\bf Case 2.}   $F_1$ is of type IA/IIIA and $F_2$ is of type IIA. Note that  $F_2$ stays tame elementary in any marked $\mathbb A$-graph equivalent to $\mathbf B_2$ unless $F_2$ satisfies the elementary condition (El.1) and   $w=\alpha(g)$ is affected  either by peripheral/Nielsen moves or  $g$ is affected by auxiliary moves of type A2.  Thus if  either $w$ is not exceptional or $w$ is exceptional but $F_2$ satistfies the elementary condition (El.2),  then the lemma holds with $\mathbf B=\mathbf B_1$ and $\bar F_1=F_1 $ and $\bar F_2$ the tame elementary fold of type IIA  based on $g$.

Thus we assume that $w$ is exceptional and  $F_2$ satisfies the elementary condition (El.1).  Hence, $\mathcal P_{w, 2}=(T_{w,2}'\oplus(\gamma_g'),  P_{w,2})$  such that 
$$B_{w, 2}=\langle T_{w,2}'\rangle\ast \langle \gamma_{g}'\rangle  \ast \langle \gamma_1\rangle\ast \ldots   \ast \langle \gamma_r \rangle $$ 
where $\gamma_g'$ is the peripheral element corresponding to $F_2$ and $P_{w,2}=(\gamma_1, \ldots, \gamma_r)$.

As  $F_1$ is tame   we can assume that $B_{f_1}=1$. If $f_1\neq g$, then  we  can apply the same argument as in the previous case as the A2 move applied to $f_1$ does not affect $w$. If $f_1 =g$ then it follows from Lemma~\ref{lemma:ruledoutcases}(2) that $B_{f_2}=1$.  Therefore,  we can apply the argument from Case 1 to $f_2$ instead of $f_1$.

\smallskip

\noindent {\bf Case 3.} $F_1$ and $F_2$ are of type IIA. If $x\neq w$,  then the result follows easily as the elementary moves necessary to make the folds elementary  affect distinct vertices. Thus assume that $x=w$.    The result is also  trivial  if $x$  is peripheral since in this case we can take $\mathbf B=\mathbf B_1$ and $\bar{F}_1=F_1$ and $\bar{F}_2$ the tame elementary fold of type IIA based on $g$. 

Assume that $x$ is non-peripheral. Lemma~\ref{lemmatypesIIA} implies that $x$ is exceptional.   As  configuration \textbf{(D1)} is excluded, $f\neq g$.  Assume that the labels of $f$ and $g$ in $\mathcal B_2$ are $(a, e, b)$  and $(c, e', d)$ respectively. The fact that $F_2$ is elementary implies that $\mathcal P_{x, 2}= (T_{x, 2}'\oplus (\gamma_g'), P_{x,2})$ such that 
$$B_{w, 2}=\langle T_{w,2}'\rangle\ast \langle \gamma_{g}'\rangle  \ast \langle \gamma_1\rangle\ast \ldots   \ast \langle \gamma_r \rangle $$ 
where $P_{x,2}=(\gamma_1, \ldots, \gamma_r)$  and  $\gamma_g'$ is the peripheral element  corresponding  to the fold $F_2$, i.e. 
 $\langle\gamma_g'\rangle = c\alpha_{e'}(A_{e'})c^{-1}\cap B_{x,2}.$   
If $F_1$ and $F_2$ correspond to the same boundary component of the orbifold $\mathcal O_x'$ corresponding to $B_{x,2}$ then $e=e'$ and $a=\beta c \alpha_{e'}(d)$ for some $\beta \in B_{x,2}$ and some $d\in A_{e'}$. In this case   $\mathbf B$ is  defined as  the marked $\mathbb A$-graph that is obtained from $\mathbf B_2$ by an auxiliary move  of type A2  based on $f$  followed by an A1 move based on $f$ that  makes the label of $f$ equal to $(c, e, b')$,  see  Figure~\ref{fig:commnonel2}. This clearly  makes  both folds  elementary.  The result now follows from Lemma~\ref{lemma_edges1}. 
\begin{figure}[h!]
\centering
\includegraphics[scale=1]{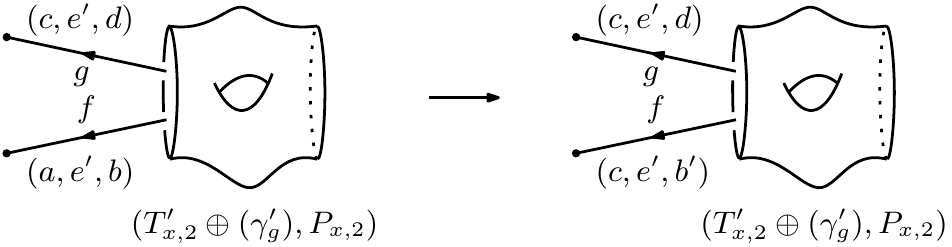}
\caption{$b'=\omega(d)b\in A_{[\omega(f)]}$. }
\label{fig:commnonel2}
\end{figure}

Assume now that $F_1$ and $F_2$  correspond to distinct boundary components of $\mathcal O_x'$. As  configuration \textbf{(D4)} is excluded, $\mathcal P_{x,2}$  is equivalent to $\mathcal P_x':=(T''\oplus (\bar \gamma_f, \bar \gamma_g), P'') $ such that 
 $$B_{x,1}=\langle T''\rangle \ast \langle \bar{\gamma}_f\rangle\ast\langle\bar{\gamma}_g\rangle\ast \langle \gamma_1''\rangle\ast \ldots   \ast \langle \gamma_r''\rangle$$
 where   $\bar{\gamma}_f$ and $\gamma_f'$ (resp. $\bar{\gamma}_g$ and $\gamma_g'$) corresponding to the same boundary of the orbifold  $\mathcal O_x'$ and $P''=(\gamma_1'', \ldots, \gamma_r'')$. Let $a_f, c_g\in B_{x,2}$ such that 
 $\langle \bar{\gamma}_f\rangle =  a_g \alpha_e(A_e)a_g^{-1}$ and $\langle\bar{\gamma}_g \rangle =c_g \alpha_{e'}(A_{e'})c_g^{-1}\cap B_{x,2}.$  
In this case $\mathbf B$ is the marked $\mathbb A$-graph  that is obtained from   $\mathbf B_2$  by elementary moves   based on $x$  that replaces $\mathcal P_{x, 2}$ by  $\mathcal P_x'$,  followed  by auxiliary moves of type A2  based on $f$ and $g$   making the labels of $f$ and $g$ equal to $(a_f, e, b')$ and $(c_g, e', d')$ respectively, see Figure~\ref{fig:commnonel3}.  In $\mathbf B$   the folds along $f$ and $g$ are elementary.  The result now follows from Lemma~\ref{lemma_edges1}.  
 \begin{figure}[h!]
\centering
\includegraphics[scale=1]{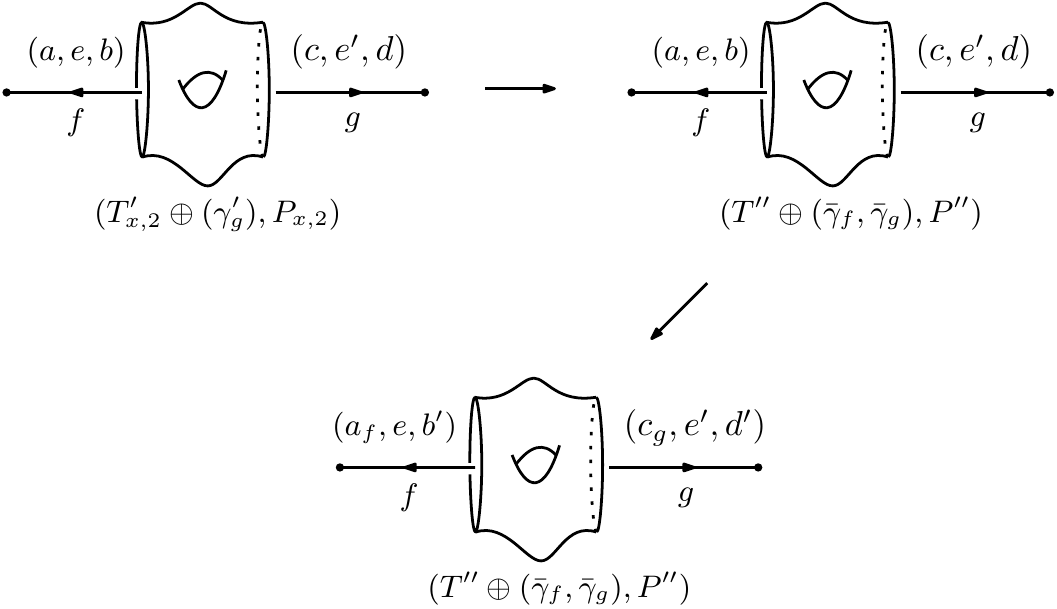}
\caption{The marked $\mathbb A$-graph $\mathbf B$. }
\label{fig:commnonel3}
\end{figure}
\end{proof}

When   $\mathbf B_1=\mathbf B_2$ (so that $F_1$ and $F_2$ are tame elementary folds that are applied to the same marked $\mathbb A$-graph) is it possible that one fold is no longer tame after the other is applied.   The following lemma tells us  when this occur.

\begin{lemma}{\label{lemma5cases1}}
Assume that $\mathbf B:=\mathbf B_1=\mathbf B_2$. Then $F_i$ stays tame after $F_j$ is applied unless:
\begin{enumerate}
\item[\textbf{(D5)}] $F_1$ and $F_2$ are of type IIA   with $f\neq g$ and $\alpha(f)=\alpha(g)$ such that  $f$ and $ g$ can be folded in $\mathcal B:=\mathcal B_1=\mathcal B_2$ by an elementary fold.
\end{enumerate} 
\end{lemma}
\begin{proof}
The claim follows by an inspection of the various cases.
\end{proof}

Observe that if (\textbf{D5}) occurs then $F_i$ is not tame after $F_j$  is applied since $f$ and $g$ can be folded  and one of them has non-trivial group.

%--------------------------------------------------------------------------

\subsection{Proof of Proposition~\ref{lemma:fork2}}
This section is dedicated to the proof of Proposition~\ref{lemma:fork2}. We assume the setup of the previous section as  established in Convention~\ref{conventiontwofolds}. Let $i\neq j\in \{1, 2\}$. According to Lemma~\ref{lemma:equivcore}  and  Lemma~\ref{lemma:corerevealsbad} it suffices to show that if $F_i$ is good and $F_j$ is bad, then there exists a tame marked $\mathbb A$-graph $\mathbf B_i''$ that is  equivalent to $\mathbf B_i$ and  admits a bad  tame elementary fold. We first consider the ``non-normalizable" configurations  \textbf{(D1)}-\textbf{(D4)} and configuration \textbf{(D5)}. We then consider generic case, that is,  the case when the conclusion of Lemmas~\ref{lemma5cases} and \ref{lemma5cases1} hold.

\smallskip

\noindent \textbf{(D1)} Possibly after exchanging the roles of $F_1$ and $F_2$ we can assume that $F_1$ is good and $F_2$ is bad. It is not hard to see that  $F_1$ induces a marking on  $\mathcal B_1'$ iff   $F_2$ induces a marking on  $\mathcal B_2'$. As one of the folds is good,  we conclude that both  $F_1$ and $F_2$ induce markings on the resulting  $\mathbb A$-graphs.   According to  Lemma~\ref{lemma_edges1}, $\mathbf B_1'$ and $\mathbf B_2'$ are equivalent   if $F_1$ and $F_2$ are of type IIA, and hence $F_1$ is good iff $F_2$ is good.   Consequently $F_1$ and $F_2$ are of type IA/IIIA as $F_2$ is bad.

\begin{claim}{\label{claim:01}}
The vertices $\omega(f_1)$ and $\omega(f_2)$ (which may coincide if $F_1$ is of type IIIA)  are exceptional and therefore non-critical and  of simple type. 
\end{claim}
\begin{proof}  
Assume that $F_1$ and $F_2$ are of type IA, that is,   $\omega(f_1)\neq \omega(f_2)$. The case where $F_1$ and $F_2$ are of type IIIA is handled similarly. 

If  $\omega(f_1)$ or $\omega(f_2)$, say $\omega(f_1)$,  is of orbifold type then the badness of $F_2$ implies that $B_{\omega(f_2), i}\neq 1$. But then $F_1$ is also bad, a contradiction.   Thus neither $\omega(f_1)$ nor $\omega(f_2)$ is of orbifold type. 
 
If  $\omega(f_1)$ and $\omega(f_2)$ are peripheral, then the previous paragraph implies that both have trivial group.  But this implies that   $F_1$ and $F_2$ are good which is a  contradiction.  

Therefore $\omega(f_1)$ and $\omega(f_2)$ are exceptional vertices.
\end{proof}

The previous claim implies that $F_1$ and $F_2$ induce markings on $\mathcal B_1'$ and $\mathcal B_2'$ respectively. Moreover,  as $F_2$ is bad and $F_1$ is good,  we conclude that  $\mathbf B_1'$ is not equivalent to $\mathbf B_2'$. As $\omega(f_i)$  is non-peripheral  it follows that   $x$ is peripheral and so $B_{x,1}=B_{x,2}$.  
Lemma~\ref{lemma_edges1} implies that the following holds: 
\begin{enumerate}

\item[(i)] $B_{f_1}=B_{f_2}=1$.

\item[(ii)] there are $g\in A_{[x]}$, $\alpha_1\neq \alpha_2\in B_x$ and $c\neq d\in A_e$ such that 
$$o_{f_1}^{\mathcal B_2}=  g\alpha_1  o_{f_1}^{\mathcal B_1}  \alpha_e(c^{-1}) \ \ \text{ and } \ \ \ o_{f_2}^{\mathcal B_2}=g \alpha_2  o_{f_2}^{\mathcal B_1}  \alpha_e(d^{-1}).$$ 
In particular, 
$$B_{x}\cap o_{f_1}^{\mathcal B_1}   \alpha_e( A_e ) (o_{f_1}^{\mathcal B_1})^{-1}\neq 1 \ \ \text{ and } \ \  B_{x}\cap o_{f_1}^{\mathcal B_2}   \alpha_e( A_e ) (o_{f_1}^{\mathcal B_2})^{-1}\neq 1.$$
Moreover, 
 $$t_{f_1}^{\mathcal B_2}=\omega_e(c)t_{f_1}^{\mathcal B_1}\beta_1h_1^{-1} \ \ \text{ and } \ \  t_{f_2}^{\mathcal B_2}=\omega_e(d)t_{f_2}^{\mathcal B_1}\beta_2h_2^{-1}$$  
with  $\beta_i\in B_{\omega(f_i),1}$ and  $h_1, h_2\in A_{[\omega(f_1)]}=A_{[\omega(f_2)]}$ such that $h_1=h_2$ if $\omega(f_1)=\omega(f_2)$.   
\end{enumerate}

Assume that $F_1$ (and therefore  $F_2$) is of type IIIA, the case of a fold of type IA is similar.  Possibly after applying an A0 move based on $y=\omega(f_1)=\omega(f_2)$  with conjugating element $h_1=h_2\in A_{[y]}$ followed by A2 moves based on $f_1$ and $f_2$ with elements $\beta_1^{-1}, \beta_2^{-1}\in B_{y,2}$  respectively    (since both moves commute with $F_1$ and $F_2$) we can  assume that $t_{f_1}^{\mathcal B_2}=\omega_e(c)t_{f_1}^{\mathcal B_1}$  and $t_{f_2}^{\mathcal B_2}=\omega_e(d)t_{f_2}^{\mathcal B_1}$. In particular,   $\mathcal P_{y,1}=(T_{y,1}, P_{y,1})$ and $\mathcal P_{y,2}=(T_{y,2}, P_{y,2})$ are equivalent.

By definition, $F_1$ (resp. $F_2$)  replaces   $ \mathcal P_{y, 1}$ (resp. $\mathcal P_{y,2}$)
by 
$$\mathcal P_{y, 1}'= (T_{y, 1}\oplus (b_1^{-1}b_2), P_{y, 1}) \ \  (\text{resp. }    \mathcal P_{y, 2}'= (T_{y,2}\oplus (b_1^{-1}\omega_e(c^{-1}d)b_2), P_{y,2})), $$
where $b_i:=t_{f_1}^{\mathcal B_1}$. The fact that  $F_1$ is good (resp.~$F_2$ is bad) combined with the fact that  $core( \mathcal B_1')$  and $core(\mathcal B_2')$  cannot be  almost  orbifold cover with good marking as $f$ lies in the core of $\mathcal B_i'$ and $B_{f,1}'=B_{f,2}' =1$, implies that  $\mathcal P_{y,1}'$ is non-critical of simple type (resp.~$\mathcal P_{y,2}'$ cannot be non-critical of simple type).  Recall that $f$ denotes  the image of $f_1$ and $f_2$ under $F_1$ and $F_2$.

Item (iii) means that  $f'\in EB'$ violates condition (F2) of folded $\mathbb A$-graph.   Let $F_1'$ be the fold of type IIA based on $f$ that replaces  $B_{f,1}'=1$ by  $$B_{f, 1}''=\langle c_{f'}\rangle$$ where $c_f\in A_e$ such that $B_{x, 1}=\langle\alpha_e(c_{f'})\rangle$. Assume first that $f$ is normalized in $\mathbf B_1'$ (and hence $F_2'$ is tame and elementary). As $x$ is peripheral this means that $f$ cannot be folded with any edge from $Star(x, B')\cap E(\mathcal B_1')$.    By definition,  $F_2'$ replaces  
$$\mathcal P_{y, 1}'=(T_{y,1} \oplus (b_1^{-1}b_2), P_{y,1})$$  
by 
$$\mathcal P_{y,1}''=(T_{y, 1} \oplus (b_1^{-1}b_2), ( b_1^{-1}\omega_e(c_{f'})b)\oplus P_{y,1})$$ 
which is clearly equivalent to 
$$(T_{y, 2} \oplus (b_1^{-1}\omega_e(c^{-1}d)b_2),  (b_1^{-1}\omega_e(c_{f'})b_1)\oplus P_{y,2})$$
because  condition (4) of marked $\mathbb A$-graphs implies that $$b_1^{-1}\omega_e(c^{-1}d) b_1\in \langle b_1^{-1}\omega_e(c_{f'})b_1\rangle.$$
The badness of $F_2'$ now follows from Corollary~\ref{lemma:sumtuples}. 

Assume now that $f$ can be folded with some edge in $Star(x, B')\cap E(\mathcal B_1')$ so that $f$ is not normalized in $\mathcal B_1'$ and hence $F_1'$ is not tame and elementary. In this case the argument used to prove item (2) of  Lemma~\ref{lemma:ruledoutcases}   shows that the fold that identifies $f$ with an edge in $Star(x, B')\cap E(\mathcal B_1')$ is bad.

\smallskip

\noindent \textbf{(D2)} We may   without  loss of generality  assume that $F_2$ is bad.  By Corollary~\ref{corollary_edges2}, there is a marked $\mathbb A$-graph  $\bar{\mathbf B}_1$    equivalent to $\mathbf B_1$ such that the following hold:
\begin{enumerate}
\item[(i)] the labels of $f_1$ and $f_2$ are  equal to $(a, e, b)$   and the labels of $g_1$ and $g_2$ are  equal to $(c, \bar e, d_1)$ and $(c, \bar e, d_2)$ respectively.

\item[(ii)] the marked $\mathbb A$-graph  $\bar{\mathbf C}_1$ that is obtained from $\bar{\mathbf B}_1$ by the fold $\bar F_1$ that identifies the edges  $f_1$ and $f_2$ is equivalent to $\mathbf B_1'$. 

\item[(iii)] if $\bar{\mathbf B}_2$ denotes the marked $\mathbb A$-graph that is obtained from $\bar{\mathbf B}_1$ by an A0 moves that makes the labels of $g_1$ and $g_2$ equal to $(c, \bar e, d)$, then  the fold  $\bar{F}_2$ that identifies $g_1$ and $g_2$  is  bad.   
\end{enumerate} 
Therefore  there is no loss  if we assume that $\mathbf B_1=\bar{\mathbf B}_1$ (resp.~$\mathbf B_2=\bar{\mathbf B}_2$) and $\bar{F}_1=F_1$ (resp.~$F_2=\bar{F}_2$).
\begin{figure}[h]
\centering
\includegraphics[scale=1]{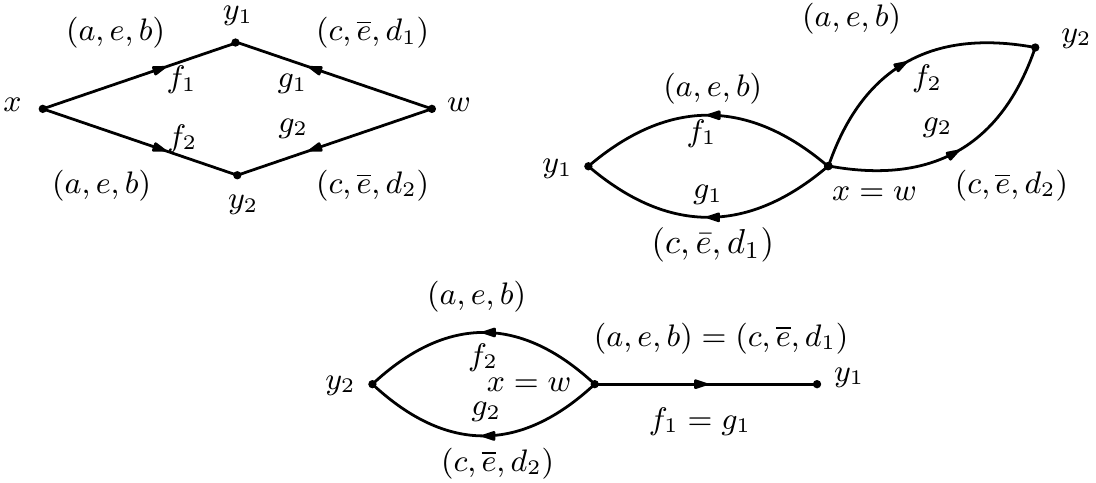}
\caption{The marked $\mathbb A$-graph $\bar{\mathbf B}_1$ in the three configurations that can occur in case (2).}
\label{fig:3optionscase(i)}
\end{figure}

\begin{claim}
The vertices  $y_1=\omega(f_1)=\omega(g_1)$ and  $y_2=\omega(f_2)=\omega(g_2)$ are exceptional.
\end{claim}
\begin{proof} The argument is entirely analogous to the argument for Claim~\ref{claim:01}.    
\end{proof}

\begin{figure}[h!]
\centering
\includegraphics[scale=1]{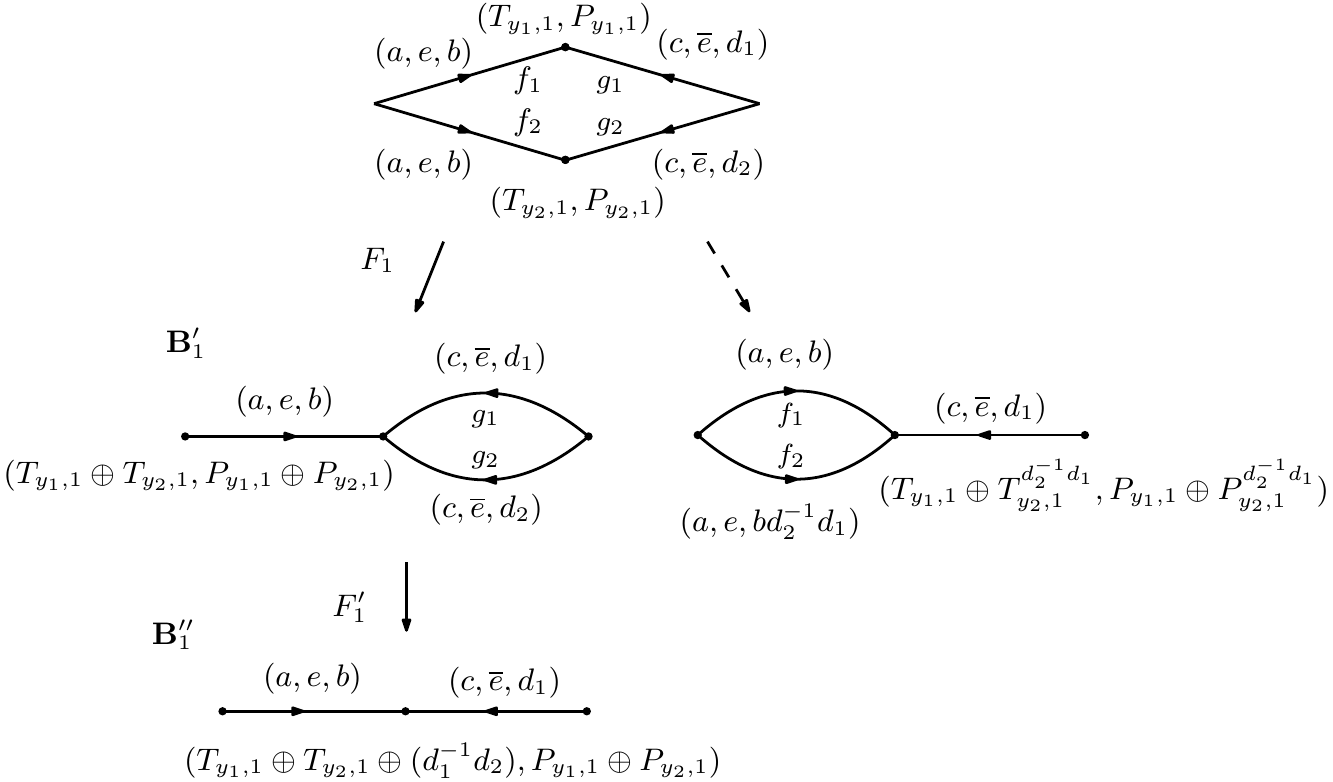}
\caption{The bad tame elementary  fold $F_1'$.}
\label{fig:commIAbad}
\end{figure}

The assumption that $\mathbf B_i=\bar{\mathbf B}_i$ clearly implies that $$\mathcal P_{y_1, 2}=\mathcal P_{y_1, 1} \ \ \text{ and }  \ \ \mathcal P_{y_2, 2}=\mathcal P_{y_2,1}^{d_2^{-1}d_1}= (T_{y,1}^{d_2^{-1}d_1} , P_{y,1}^{d_2^{-1}d_1}).$$ 
See Fig.~\ref{fig:commIAbad}.  The cores of $\mathcal B_1'$ and $\mathcal B_2'$ cannot be  almost orbifold covers   as they contain the subgraphs $f_1\cup f_2$ and  $g_1\cup g_2$  respectively. 
The  badness of $F_2$ therefore implies that 
$$\mathcal P_{y, 2}'= (T_{y_1, 1}\oplus T_{y_2, 1}^{d_2^{-1}d_1} ,   P_{y_1, 1}\oplus P_{y_1, 1}^{d_2^{-1}d_1})$$ 
is either critical or of almost orbifold covering type,   where for simplicity  $y$ denotes the image of $y_1$ and $y_2$ under both $F_1$ and  $F_2$.

Let   $F_1'$ denote  the tame elementary fold of type IIIA that identifies the edges $g_1$ and $g_2$ in ${\mathcal B}_1'$ and let ${\mathcal B}_1''$  denote the resulting  $\mathbb A$-graph. As $y$ is exceptional  it follows that $F_1'$ induces a marking on $\mathcal B_1''$ yielding a marked $\mathbb A$-graph $\mathbf B_1''$.  We will show that $F_1'$ is bad.  By definition, $F_1'$ replaces $\mathcal P_{y,1}'$ by
$$\mathcal P_{y,1}''= (T_{y_1} \oplus T_{y_2}\oplus (d_1^{-1}d_2),  P_{y_1}\oplus P_{y_2})$$
which is equivalent 
$$(T_{y_1}\oplus T_{y_2}^{d_2^{-1}d_1}\oplus (d_1^{-1}d_2),  P_{y_1}\oplus P_{y_2}^{d_2^{-1}d_1}).$$
the latter is the sum of two partitioned tuple one of which is either critical or of almost orbifold covering type. The result thus follows from  Lemma~\ref{lemma:sumtuples}.

\smallskip

\noindent \textbf{(D3)} Without loss of generality  assume that $F_1$ is good. It follows from Corollary~\ref{corollary_edges2} that $\mathbf B_1$ can be replaced by an equivalent marked $\mathbb A$-graph $\bar{\mathbf B}_1$ such that the following hold:
\begin{enumerate}
\item[(i)] the label of $f_1$ is $(a, e, b)$  and the label of  $f_2$ is $(a, e, \omega_e(c)b)$ for some non-trivial element  $c\in A_e$.

\item[(ii)]  the marked $\mathbb A$-graph that is obtained from $\bar{\mathbf B}_1$ by the fold $\bar{F}_1$ that identifies  the edges $f_1$ and $f_2$ is equivalent to $\mathbf B_1'$.    

\item[(iii)]  if $\bar{\mathbf B}_2$ denotes the marked $\mathbb A$-graph that is obtained from  $\bar{\mathbf B}_1$ by an A1 move that makes the label  of $f_2$ equal to $(a\alpha_e(c), e, b)$, then  the fold $\bar{F}_2$ that identifies $f_1^{-1}$ and $g_2^{-1}$   is bad. 
\end{enumerate}

We can therefore assume  that $\mathbf B_1=\bar{\mathbf B}_1$ (resp. $\mathbf B_2=\bar{\mathbf B}_2$)   and that $F_1=\bar{F}_1$ (resp. $F_2=\bar{F}_2$). 
Denote  $\gamma:= b^{-1}\omega_e(c)b\in A_{[w]}$ and $\gamma':=a^{-1}\alpha_e(c)a\in A_{[x]}$ where $w=\omega(g_1)=\omega(g_2)\in EB$.  There are two cases depending on the type of $x=\alpha(f_1)=\alpha(f_2)$.

\noindent{\it Case 1.} $x$ is peripheral (and hence $w$ is non-peripheral). This case is illustrated in Figure~\ref{fig:commIIIAbadperipheral}.   The assumptions that $\mathbf B_i=\bar{\mathbf B}_i$ implies that $\mathcal P_{w,1}=\mathcal P_{w,2}$. 

The badness of $F_2$ means  that $B_x\neq 1$. On the other hand, as  $F_1$ is good,    $w$ is exceptional (and hence  $\mathcal P_{w,1}$  is non-critical of simple type).  By definition,  $F_1$ replaces $\mathcal P_{w,1}$ by
 $\mathcal P_{w,1}'=(T_{w,1}\oplus (\gamma), P_{w,1}) $ 
which is non-critical.  In particular,    $B_{f_1}$ and $B_{f_2}$ must be trivial. For if some of them, say $B_{f_1}$, is non-trivial, then $\mathcal P_{w,1}'=(T_{w,1}\oplus ( \gamma) , P_{w,1})$  is reducible or has an obvious relation  as  $\gamma_{f_1^{-1}}$ and $ \gamma$ lie in $ b^{-1}\omega_e(A_e)b$.  In particular, the  image $f'$ of $f_1$ and $f_2$  under $F_1$ has trivial group in $\mathcal B_1'$.  This implies that the core of $\mathbf B_1'$ cannot be an almost orbifold cover. Thus, $\mathcal P_{w,1}'$ is non-critical of simple type and hence   
\begin{figure}[h]
\centering
\includegraphics[scale=1]{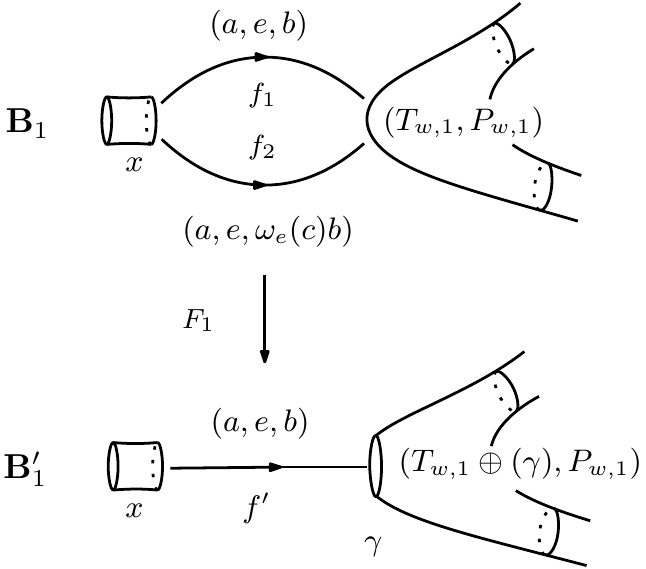}
\caption{The fold along $f^{-1}$ is bad.}
\label{fig:commIIIAbadperipheral}
\end{figure} 
$$ b^{-1}\omega_e(A_e)b\cap B_{w,1}'=b^{-1}\omega_e(A_e)b\cap \langle T_{w, 1}\oplus  (\gamma)\oplus  P_{w, 1}\rangle=\langle \gamma\rangle$$ 
This shows that $(f')^{-1}$ violates condition (F2) of folded $\mathbb A$-graphs.  Let $F_1'$ denote the fold of type IIA  that is based on $(f')^{-1}\in EB_1'$.  We claim that $F_1'$ is tame and  elementary.  
In fact, we only need to show that $(f')^{-1}$ cannot be folded with any edge in $Star(w, B_1')\cap E(\mathcal B_1')$. But this follows as $(\gamma)\oplus  P_{w, 1}$ is part of a minimal generating set of  $B_{w, 1}'$.
 As the vertex group at $x$ is non-trivial we are in case FoldIIA(El.1)(1), thus $F_1'$ is bad.

\noindent{\it Case 2.}  $x$ is non-peripheral (and hence $w$ is peripheral). This case is illustrated in Figure~\ref{fig:commIIIAbad1}.  The hypotheses that $F_1$ is good implies that  $B_{w}=1$. In particular, $B_g=1$ for all $g\in Star(w, B)$. By definition,  $F_1$ replaces   $B_w$ by   $B_{w,1}'=\langle \gamma\rangle\leq A_{[w]}.$   On the other hand, as  $F_2$ is bad,   one of the following occurs:
\begin{enumerate} 
\item[(1)]  $x$ is  of orbifold type.

\item[(2)]  $x$  is exceptional  and  $\mathcal P_{x, 2}'=(T_{x,2} \oplus (\gamma'), P_{x,2})= (T_{x,1} \oplus (\gamma'), P_{x,1}) $ cannot be non-critical of simple type.
\end{enumerate}

Let $F_1'$ be the  fold of type IIA based on $(f')^{-1}\in EB_1'$.  We will show that $F_1'$ is   bad tame elementary.  The fact  that $w$ is peripheral combined with  $Star(w, B_1')\cap E(\mathcal B_1')=\emptyset$ imply that $F_1'$ is tame elementary since $(f')^{-1}$ satisfies conditions (II.1)-(II.3) and (El.0). It remains to show that $F_1'$ is bad. There are two cases:

\noindent{\it Subcase a.}  $x$  is of orbifold type.   In this case   $F_1'$ is bad as we are in case FoldIIA(El.0)(1).

\noindent{\it Subcase b.}  $x$ is exceptional (and therefore $\mathcal P_{x,1}'=\mathcal P_{x,1}$ is non-critical  of simple type). By definition, $F_1'$ replaces  $\mathcal P_{x,1}'=\mathcal P_{x,1}=(T_{x,1},  P_{x,1})$  by 
 $\mathcal P_{x,1}'':=(T_{x,1},(\gamma')\oplus  P_{x,1}).$ 
Observe that $\mathbf B_1''$ cannot be an almost orbifold covering with a good marking  since $w$ is a free boundary vertex in $\mathbf B_1''$.  Therefore $F_1'$ is bad unless $\mathcal P_{x, 1}''$ is non-critical of simple type.  But  $\mathcal P_{x, 1}''$ cannot be non-critical of simple type since then   $\mathcal P_{x,1}'=(T_{x,1}\oplus(\gamma'),  P_{x,1})$  would also be   non-critical of simple  which contradicts (2).
\begin{figure}[h]
\centering
\includegraphics[scale=1]{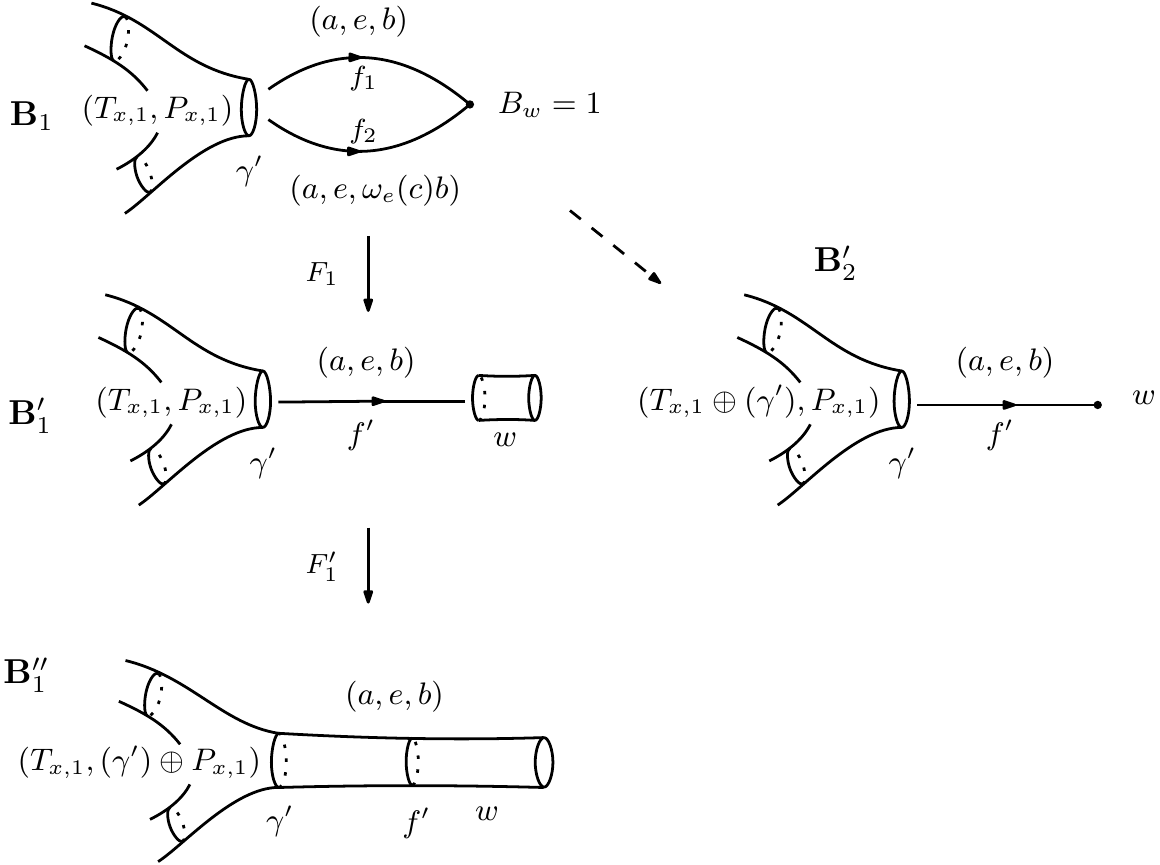}
\caption{The fold $F_1'$ is bad.}
\label{fig:commIIIAbad1}
\end{figure}

\smallskip

\noindent \textbf{(D4)}  We may assume without loss of generality that  $F_2$ is bad, so that the peripheral vertex  $z:=\omega(g)$  has non-trivial group. This case is illustrated in Figure~\ref{fig:commnone4}.

Let $F_2'$ be  the fold  of type IIA based on the edge $g$ of $\mathcal B_1'$. Observe that  $g$ cannot be folded with any edge in $Star(x, B_1')\cap E(\mathcal B_1')$ because  $\gamma_g'$ is part of a minimal generating set of $B_{x, 1}'=B_{x, 1}=B_{x, 2}$.  Moreover, $F_2'$  is elementary because it satisfies  condition (El.2).  
  The badness of $F_2'$  follows from the  fact that the vertex $z$ has non-trivial group in $\mathcal B_1'$ since $B_{z, 1}'=B_{z, 1}=B_{z,2}$.  
\begin{figure}[h!]
\centering
\includegraphics[scale=1]{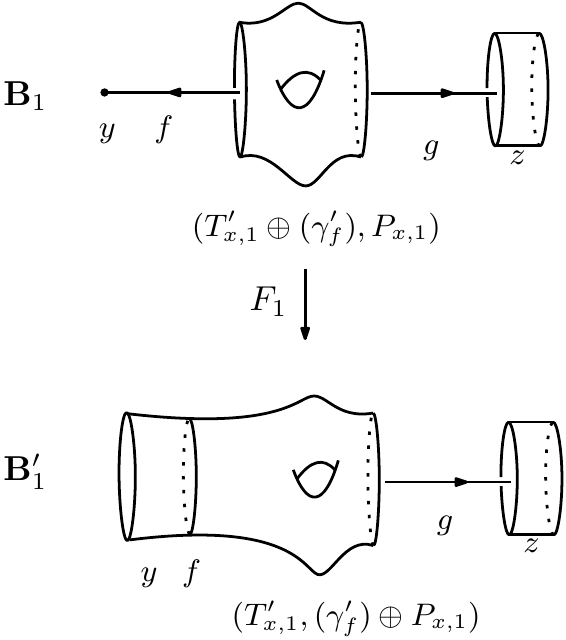}
\caption{The folds $F_2$ and $F_2'$  along   $g$ are bad.}
\label{fig:commnone4}
\end{figure}

\smallskip

\noindent \textbf{(D5)} $\mathbf B:=\mathbf B_1=\mathbf B_2$ and     $F_1$ and $F_2$ are of type IIA   with $f\neq g$ such that  $f$ and $ g$ can be folded in $\mathcal B:=\mathcal B_1=\mathcal B_2$ by an elementary fold.  Note that  $F_i$  is not tame after  $F_j$ is applied. 

We may assume that $F_1$ is good and $F_2$ is bad.  Let $F$ denote  the fold that identifies $f$ and $g$ in $\mathbf B_1$. Note that since $F_1$ and $F_2$ are tame elementary     the labels of $f$ and $g$ are not affected   and therefore $F$ is tame elementary.  Let $\gamma_f'$ (resp.~$\gamma_g'$) be the element added by $F_1$ (resp.~$F_2$) to $B_{y,1}$ (resp.~$B_{z,2}$).

\noindent{\emph{Case 1}.}  $z =\omega(g)$ is of orbifold type (either peripheral or non-peripheral). Since $F_1$ is good it follows that  $y=\omega(f)\neq \omega(g)=z$. In particular, $F$ is of type IA.   Note that  $B_{y,1}'\neq 1$ since $F_1$ adds $\gamma_f'$ to $B_{y,1}$. Therefore $F$ is bad as it identifies the vertex of orbifold type  $z$  with the vertex $y$ that has non-trivial group.

\smallskip

\noindent{\emph{Case 2.}}  $z =\omega(g)$ is exceptional (and hence $x$ is peripheral of orbifold type).  In this case $F_1$ and $F_2$ satisfy the elementary  condition (El.0).  Note also that  $y$ is non-peripheral as $[y]=[z]\in VA$. Thus  $y$  is exceptional   because $F_1$ is good. 

We give a complete argument for   the case  $y\neq z$. The case $y=z$ is analogous. By definition $F_1$ replaces $\mathcal P_{y, 1}=(T_{y,1}, P_{y,1})$ by  
 $\mathcal P_{y, 1}'=(T_{y,1}, (\gamma_f')\oplus P_{y,1})$  and $F_2$ replaces $\mathcal P_{z, 1}=(T_{z,1}, P_{z,1})$ by 
 $\mathcal P_{z, 2}'=(T_{z,1}, (\gamma_g')\oplus  P_{z,1})$ 
As $F_2$ is bad it follows that $\mathcal P_{z,2}'$ cannot be non-critical of simple type.  

On the other hand, $F$ replaces the partitioned tuples  $\mathcal P_{y, 1}'$ and $\mathcal P_{z,1}'=\mathcal P_{z,1}=\mathcal P_{z,2}$ by the partitioned tuple
 $\mathcal P: =(T_{y,1}\oplus T_{z,1}, (\gamma_f')\oplus P_{y,1}\oplus P_{z,1}).$ Since $F$ is elementary it follows that $\gamma_g'=\gamma_f'$.  The badness of $F$ now follows from Corollary~\ref{lemma:sumtuples}.

 \smallskip

\noindent \textbf{Generic case.} Thus the conclusion of Lemmas~\ref{lemma5cases} and \ref{lemma5cases1} holds. Thus   we may assume that  $\mathbf B:=\mathbf B_1=\mathbf B_2$ and that $F_i$ stays tame after $F_j$ is applied.  Assume  that $F_1$ is good. We will show that  $\mathbf B_1'$ admits a bad tame elementary fold that is  induced by $F_2$.

\smallskip

\noindent{\emph{Case 1.}} $F_2$ is of type IA/IIIA. As $F_2$ is tame elementary we can assume that: 
\begin{enumerate}
\item[(i)] there is $k\in \{1, 2\}$ such that   $B_{g_k}=1$.

\item[(ii)] if $F_2$ is of type  IA, then   $g_1$ and $g_2$ are labeled $(c, \bar e, d)$. 

\item[(ii)'] if $F_2$ is of type IIIA then $g_1$ and $g_2$ are labeled $(c, \bar e, d_1)$ and $(c, \bar e , d_2)$.
\end{enumerate}
 
\begin{claim} $F_1$ maps the   subgraph $g_1\cup g_2$ of $ B$    isomorphically   into $B_1'$.
\end{claim}
\begin{proof}
Note that  $g_1\cup g_2$ is mapped isomorphically into $B_1'$ unless one of the following holds: 
\begin{enumerate}
\item[(a)] $F_1$ and $F_2$ are of type IIIA and $\{f_1, f_2\}=\{g_1^{\pm 1}, g_2^{\pm 1}\}$.

\item[(b)] $F_1$ and $F_2$ are of type IA and $\{\omega(f_1), \omega(f_2)\}=\{\omega(g_1), \omega(g_2)\}$. 
\end{enumerate}
Item (a) does not occurs because configurations \textbf{(1)} and \textbf{(3)} are excluded while item (b) does not occur because  configurations \textbf{(1)} and \textbf{(2)} are excluded. Therefore $g_1\cup g_2$ is mapped isomorphically into $B_1'$.
\end{proof}

We will denote the image of  $g_1$ and $g_2$ (resp.~$z_i=\alpha(g_i) \in VB$ for $i=1,2$)  under $F_1$ by $g_1'$  and $g_2'$ (resp.~$z_i'$ for $i=1,2$).  If $F_2$ is of type IIIA then $z:=z_1=z_2$ and $z':=z_1'=z_2'$. The previous claim therefore says that $g_1\cup g_2$ is isomorphic to $g_1'\cup g_2'$.

\begin{claim}
$B_{g_1',1}'=1$ or $B_{g_2',1}'=1$ (in $\mathcal B_1'$).
\end{claim}
\begin{proof}
Suppose on the contrary  $B_{g_1',1}'\neq 1\neq B_{g_2',1}'$. Thus $F_1$ affects the edge group  of $g_k\in EB$ (which is trivial). Observe that  this occurs only when one of the following holds:
\begin{enumerate}
\item[(1)] $F_1$ is of type IA/IIIA, $B_{f_i}\neq 1$,  and $f_j=g_k^{\pm 1}$ for $i\neq j\in \{1, 2\}$. 

\item[(2)] $F_1 $ is of type IIA and $f=g_k^{\pm 1}$.  
\end{enumerate}
Both cases are ruled out by Lemma~\ref{lemma:ruledoutcases}.
\end{proof}
 
Next observe that the fold  that identifies  $g_1'$ and $g_2'$  in $\mathcal B_1'$, which we denote by $F_2'$, is elementary  unless the following holds:
\begin{enumerate}
\item[(a)] $F_1$ is of type IIIA   
 and  replaces the edges $f_1$ and $f_2$   labeled $(a, e, b_1)$ and $(a, e, b_2)$  respectively by a single edge $f'$ labeled $(a, e, b_1)$.
 
\item[(b)] $f_2= g_j^\varepsilon$ for some $j\in \{1, 2\}$ and some $\varepsilon\in \{\pm 1\}$.
\end{enumerate}

To overcome this issue we can  assume  that the label of $f'$  is $(a, e, b_2)$ so that $F_2'$ is elementary. Observe that this  makes no difference in our argument since  instead of adding the element $b_1^{-1}b_2$ to the vertex group $B_y$ we add $b_2^{-1}b_1$  which clearly  gives equivalent marked $\mathbb A$-graphs. In particular,  we can assume that $g_i'$ and $g_i$  have the same label.
 
\smallskip

Therefore $F_2'$ is tame elementary. We  can finally show that $F_2'$ is bad. We assume that $F_2$ is of type IA. The situations in which $F_2$ is of type IIIA is entirely analogous.

If  one of the vertices $z_1$ or $z_2$ is of orbifold type, say $z_1$,  then the badness of $F_2$ means that $B_{z_2}\neq 1$. Since the image of a vertex of orbifold type under a marking preserving fold is still a vertex of orbifold type,  $z_1'$ is of orbifold type in $\mathbf B_1'$.  By the definition of folds,     $B_{z_i}$ is a subgroup of $B_{z_i',1}'$. Thus  $B_{z_2',1}'\neq 1$. Therefore,  $F_2'$ is bad as it identifies $z_1'$ and $z_2'$.

Assume now that  $z_1$ and $z_2$ are not of orbifold type.  The badness of $F_2$ implies that  $z_1$ and $z_2$ are exceptional. In fact,  if $z_1$ (and hence $z_2$) were peripheral then the assumption that they are not of orbifold type implies that $B_{z_1}=B_{z_2}=1$. But then $F_2$ would be good.  Thus $z_1$ and $z_2$ are non-peripheral and therefore exceptional.   

By definition, $F_2$ replaces  $\mathcal P_{z_1}=(T_{z_1}, P_{z_1})$ and  $\mathcal P_{z_2}=(T_{z_2}, P_{z_2})$  by 
$$\mathcal P_{\bar{z}, 2}':=(T_{z_1}\oplus T_{z_2}, P_{z_1}\oplus P_{z_2}).$$
where $\bar{z}:=F_2(z_1)=F_2(z_2)\in VB_2'$.  As $F_2$ is bad  $\mathcal P_{\bar{z}, 2}'$  cannot be non-critical of simple type. 

On the other hand, observe that $\mathcal P_{z_i}=(T_{z_i}, P_{z_i})$ is a partitioned subtuple of  $\mathcal P_{z_i', 1}'=(T_{z_i', 1}', P_{z_i', 1}')$ and by definition,  $F_2'$ replaces $\mathcal P_{z_1', 1}'$ and $\mathcal P_{z_2', 1}'$ by 
$$\mathcal P':=(T_{z_1', 1}'\oplus T_{z_2', 1}', P_{z_1', 1}'\oplus P_{z_2', 1}').$$ 
The badness of $F_2'$ follows from Lemma~\ref{lemma:sumtuples}  as $\mathcal P_{z, 2}'$ is a partitioned subtuple of  $\mathcal P'$.

\medskip
 
\noindent{\emph{Case 2.}}  $F_2$ is of type IIA. We denote the image   of   $g \in EB$  (resp.~$w=\alpha(g)\in VB$ and $z=\omega(g)\in VB$) under  $F_1$   by  $g'\in EB_1'$  (resp.~$w'\in VB_1'$ and $z'\in VB_1'$).   

\begin{claim} 
$F_1$ does not change the group of $g$, i.e. $B_{g',1}'=1$.
\end{claim}
\begin{proof} Indeed,   $B_{g',1}'= 1$ unless one of the following holds:  
\begin{enumerate}
\item[(C1)] $F_1$ is of type IA/IIIA such that $B_{f_i}\neq 1$ and $f_j=g^{\pm1}$ for $i\neq j$.

\item[(C2)] $F_1$  is of type IIA based and $f= g^{\pm 1}$.
 \end{enumerate}  
Both configurations are ruled out by  Lemma~\ref{lemma:ruledoutcases}. Therefore $B_{g',1}'=1$  as claimed. 
\end{proof}

Therefore, as  $B_w$ is a subgroup of $B_{w', 1}'$, a fold of type IIA along $g'$ can be applied to $\mathcal B_1'$. Denote this fold by $F_2'$.

\begin{claim}
$F_2'$ is tame elementary.
\end{claim}
\begin{proof}
Assume first that $B_{w', 1}'=B_w\le A_{[w]}$, i.e.~$F_1$ does not affect the vertex $w$.  Then  $F_2'$ is tame elementary unless one of the following holds:
\begin{enumerate}
\item[(a)] $g'$ can be folded with some $h'\in Star(w', B_1')\cap E(\mathcal B_1')$, i.e.~$g'$ is not normalized in $\mathbf B_1'$.

\item[(b)] $F_2$ satisfies condition (El.1) and $F_1$  changes the label of $g$.    
\end{enumerate}
 
Note that if (a) occurs then  $F_1$ is of type IIA  based on $h$.  But then $h$ can be folded with $g$ in $\mathcal B$ what does not occur because we assumed that  \textbf{(D5)} does not occur. 
If (b) occurs then we argue as in the previous  paragraph.
This shows that  $F_2'$ is tame elementary  when  $F_1$ does not affect the vertex $w$. 

\smallskip

Assume now that $B_w\neq B_{w', 1}'$. First observe that $w$ is exceptional since good folds do not  affect  vertices  of orbifold type.  

We will give a complete argument for the case $F_1$ is of type IIIA  with $\omega(f_1)=\omega(f_2)=w=\alpha(g)$.  In this case $F_1$ replaces $B_w$ by 
 $B_{w',1}'=\langle B_w,  b_1^{-1}b_2 \rangle\le A_{[w]}.$   

As $F_1$ is good  it follows from   Proposition~\ref{lem:critical1} that   $B_w$ is of infinite index in $A_{[w]}$, and hence  $F_2$ satisfies the elementary condition (El.1). Thus,   $\mathcal P_w$ has the form $(T_w'\oplus (\gamma_g'), P_w')$ such that 
$$B_w=\langle T_w'\rangle\ast \langle \gamma_g'\rangle\ast \langle \gamma_1'\rangle \ast \ldots \ast \langle \gamma_r'\rangle$$
 where $P_w=(\gamma_1', \ldots, \gamma_r')$  and $\gamma_g'$ corresponds to the fold $F_2$, i.e.~
 $c\alpha_{\bar{e}}(A_{\bar{e}})c^{-1}\cap B_w=\langle \gamma_g'\rangle.$ 
By definition  $F_1$ replaces $\mathcal P_w$ by 
 $\mathcal P_{w',1}'=(T_w'\oplus (\gamma_g', b_1^{-1}b_2), P_w').$  Since $F_1$ is good it follows that $\mathcal P_{w', 1}'$ is either non-critical of simple type or of almost orbifold covering type.  But  Proposition~\ref{prop:02}  rules out the second alternative. Thus $\mathcal P_{w', 1}'$ is non-critical of simple type. 
Therefore,   as $\gamma_g'$ is part of a minimal generating set of $B_{w,1}'$,  it follows that   $g$ cannot be folded with any edge with  non-trivial group starting at $h'$. This shows that $F_2'$ is tame and  satisfies the elementary condition (El.1).
\end{proof}

We now show that $F_2'$ is bad. If $z'$ is  of orbifold type it is clear. Assume that $z$ is exceptional.    The badness of $F_2$ means that the partitioned tuple $(T_z, (\gamma')\oplus  P_z)$ is not non-critical of simple  type. Since $(T_z, P_z)$ is a partitioned sub-tuple of $\mathcal P_{z', 1}'$ and  since $F_2'$ simply adds a peripheral element to $\mathcal P_{z', 1}'$,   it follows from Corollary~\ref{lemma:sumtuples} that  the resulting partitioned tuple is critical. Therefore $F_2'$ is bad.

%--------------------------------------------------------------

\subsection{Proof of Proposition~\ref{lemma:fork1}}
The proof of this proposition will follow along the  lines of the proof of Proposition~\ref{lemma:fork2}. Thus we consider the ``non-normalizable''  configurations  \textbf{(D1)}-\textbf{(D5)} and then we consider the  generic case  where we can assume that $\mathbf B=\mathbf B_1=\mathbf B_2$ and  each of the folds    stays tame after   the other is applied.

Assume that $F_1$ and $F_2$ are good and  that the marked $\mathbb A$-graphs $core(\mathbf B_1')$ and $core(\mathbf B_2')$ are not equivalent. Hence, by Lemma~\ref{lemma:equivcore}, $\mathbf B_1'$ and $\mathbf B_2'$ are also  not equivalent.
$$\begin{tikzcd}
   & \mathbf B_1    \arrow[dr, "F_2" ] \approx \mathbf B_2   \arrow[dl, "F_1"'] &  \\
 \mathbf B_1' &  & \mathbf B_2'  
\end{tikzcd}$$ 
It follows from  Lemma~\ref{lemma:corerevealsbad} that  to   establish Proposition~\ref{lemma:fork1}  it suffices  to verify, in all possible scenarios, that one of the following occurs:
\begin{enumerate}
\item[(A)] There exist  marked $\mathbb A$-graphs $\mathbf B_1''$ and $\mathbf B_2''$   such that $\mathbf B_i''$ is equivalent to   $\mathbf B_i'$ and $\mathbf B_i''$  admits a  bad tame elementary fold for $i\in\{1,2\}$.

\item[(B)]  There exist marked $\mathbb A$-graphs $\mathbf B_1''$ and $\mathbf B_2''$ such that $\mathbf B_i''$ is  equivalent to $\mathbf B_i'$  for $i\in\{1,2\}$ and that  $\mathbf B_1''$ and $\mathbf B_2''$ admit good tame elementary  folds  that yield equivalent marked $\mathbb A$-graphs.  
\end{enumerate}

%------------------------------------------------------------------------

\smallskip

\noindent \textbf{(D1)} As $\mathbf B_1'$ and $\mathbf B_2'$ are not equivalent,   Lemma~\ref{lemma_edges1} implies that the following holds: 
\begin{enumerate}
\item[(i)] $F_1$ (and therefore $F_2$) is   of type IA/IIIA.

\item[(ii)] $B_{f_1}=B_{f_2}=1$.

\item[(iii)] $o_{f_1}^{\mathcal B_2}= g\alpha_1  o_{f_1}^{\mathcal B_1}  \alpha_e(c)$ and $o_{f_2}^{\mathcal B_2}= g\alpha_2  o_{f_2}^{\mathcal B_1}  \alpha_e(d)$ with  $g\in A_{[x]}$,  $\alpha_1\neq \alpha_2\in B_{x,1}$ and $c\neq d\in A_e$ where $e:=[f_1]=[f_2]$.  In particular,   $$B_{x,1}\cap o_{f_1}^{\mathcal B_1}   \alpha_e( A_e ) (o_{f_1}^{\mathcal B_1})^{-1}\neq 1  \ \ \text{ and  }  \ \ B_{x,2}\cap o_{f_2}^{\mathcal B_2}   \alpha_e( A_e ) (o_{f_2}^{\mathcal B_2})^{-1}\neq 1.$$
Moreover, $$t_{f_1}^{\mathcal B_2}=\omega_e(c)t_{f_1}^{\mathcal B_1}\beta_1h_1^{-1} \  \ \text{ and } \  \ t_{f_2}^{\mathcal B_2}=\omega_e(d)t_{f_2}^{\mathcal B_1}\beta_2h_2^{-1}$$  
with  $\beta_i\in B_{\omega(f_i),1}$ and  $h_1, h_2\in A_{[\omega(f_1)]}=A_{[\omega(f_2)]}$ such that $h_1=h_2$ if $\omega(f_1)=\omega(f_2)$.   
\end{enumerate}

Before proceeding with the argument we fix some notation. Observe that the  graphs $B_1'$ and $B_2'$ underlying $\mathbf B_1'$ and $\mathbf B_2'$ coincide.  We  denote $B':=B_1'=B_2'$. Moreover, for all $f'\in EB'$ we have $B_{f',1}'=B_{f',2}'$  which implies that   $E:=E(\mathcal B_1')=E(\mathcal B_2').$   The image of $x=\alpha(f_1)=\alpha(f_2)$ under $F_i$ will be denoted by $x'$. 

We give a complete argument in the case $F_1$ and  $F_2$  are of type IIIA. The case of folds of type IA  is similar and  left to the reader. Denote the label of $f_1$ and $f_2$  in $\mathcal B_1$ by $(a, e, b_1)$ and $(a, e, b_2)$ respectively.    After aplyig auxilairy moves of type A0 based on $x$ and $y=\omega(f_1)=\omega(f_2)$ to $\mathbf B_2$  with conjugating element $g$ and $h_1=h_2$ respectively followed by auxiliary moves of type A2 based on $f_1^{-1}$ and $f_2^{-1}$ with elements $\beta_1$ and $\beta_2$ respectively (which does not affect the equivalence class of $\mathbf B_2'$)  we can assume that  the label of $f_1$ and $f_2$ in $\mathcal B_2$ are given by 
 $(\alpha_1 a \alpha_e(c^{-1}), e, \omega_e(c)b_1)$ and  $(\alpha_2 a \alpha_e(d^{-1}), e, \omega_e(d)b_2)$ 
respectively.  Since $F_1$ and $F_2$ are good, $b_1\neq b_2$ and $\omega_e(c)b_1\neq \omega_e(d)b_2$. We distinguish two cases according to the type of $x$.

\noindent{\emph{Case 1.}} The vertex $x$ is peripheral. Note that 
$$B_x:=B_{x,1}=B_{x,2}=B_{x', 1}'=B_{x', 2}'.$$ 
As $F_1$ and $F_2$ are good,  $y$ is an exceptional vertex. Hence the partitioned tuples $\mathcal P_{y,1}$ and $\mathcal P_{y, 2}$ are non-critical of simple type.  As   $B_{f_1}=B_{f_2}=1$, we can assume that $(T, P):=\mathcal P_{y, 1}=\mathcal P_{y, 2}$ as the  elementary moves  needed to turn $\mathbf B_1$ into $\mathbf  B_2$  do not affect $\mathcal P_{y,1}$.

We now look at the marked $\mathbb A$-graphs $\mathbf B_1'$ and $\mathbf B_2'$. Item (iii) combined with $B_x=B_{x', i}'$ implies that a fold of type IIA based on  $f'\in EB'$    can be applied to $\mathcal B_i'$. We denote this fold by  $F_i'$.  Moreover, since all edges  in $Star(x', B')\setminus\{f'\}$  have the same label in $\mathcal B_1'$ and in $\mathcal B_2'$   it follows that  $f'$ can be folded with some edge in $Star(x, B_1')\cap E(\mathcal B_1')$ if, and only if, $f'$ can be folded with some edge in $Star(x, B_2')\cap E(\mathcal B_2')$.

\noindent{\emph{Sub-case a.}} $f'$ cannot be folded with an edge in $Star(x, B_1')\cap E(\mathcal B_1')$.  Since $x$ is peripheral it follows that  $F_1'$ and $F_2'$  are tame.  

Let $\mathbf B_i''$ be the marked  $\mathbb A$-graph that is obtained from $\mathbf B_i'$ by the fold $F_i'$. Therefore the partitioned tuple at $y'=\omega(f')$  in $\mathbf B_1''$  is  equal to  
 $$\mathcal P_{y', 1}''= (T_y \oplus( b_1^{-1}b_2), (\gamma)\oplus  P_y)$$ 
and   in $\mathbf B_2''$ is equal to 
$$\mathcal P_{y', 2}'':=(T_y\oplus (b_1^{-1}\omega_e(c^{-1}d)b_2, (\gamma)\oplus  P_y )$$ 
where $\gamma:=b_1^{-1}\omega_e(a_e^z)b_1$ and  $z\in \mathbb Z$  such that  
 $\langle a_e^z\rangle =\alpha_e^{-1}(a^{-1}B_xa)=\alpha_e^{-1}(B_x).$ It follows from condition (4) of marked $\mathbb A$-graphs that $b_1^{-1}\omega_e(c^{-1}d)b_1\in \langle \gamma \rangle$  which implies that    
$\mathcal P_{y',1}'' $ and $\mathcal P_{y',2}'' $ 
are equivalent.   \cite[Lemma~4.17]{Dut} implies that  $\mathbf B_1''$ and $\mathbf B_2''$  are equivalent.  Therefore $F_1'$ is bad if, and only if, $F_2'$ is bad.  This shows that (A) occurs if  $F_1'$  and $F_2'$ are bad and that (C) occurs  if $F_1'$ and $F_2'$ are good. 

\smallskip

\noindent{\emph{Sub-case b.}}  there is $h'\in Star(x', B')\cap E$ that can be folded with  $f'$.  An argument similar to that of sub-case (a) shows   $\mathbf B_1'$ is equivalent to $\mathbf B_2'$ if $\omega(h')=y'$.    Thus $\omega(h')$ and $y$ must be distinct. 

Let $F_i''$ denote the fold that identifies $f'$ and $h'$ and let $u''$ denote the common image of the vertices $y'$ and $\omega(h')$. As the group of $f'$ is trivial, $F_1''$ and $F_2''$ are tame.  Moreover, by Lemma~\ref{lemma_edges1}, we may assume that $F_1''$ and  $F_2''$ are elementary since the group of  $B_{h',1}'\neq 1\neq B_{h', 2}'$.

If   $\omega(h')$ is of orbifold type,  then $F_1''$ and $F_2''$ are bad since $B_{y',1}'\neq 1\neq B_{y',2}'$. Thus  assume that $\omega(h')$ is exceptional. Observe that we may without loos of generality  assume that 
$(T,P):=\mathcal P_{ \omega(h'), 1}'=\mathcal P_{\omega(h'), 2}'.$
Hence the partitioned tuple at $u'' $  in $\mathbf B_1''$ is equal to
 $$\mathcal P_{u'' , 1}''=(T\oplus T_y\oplus (b_1^{-1}b_2), (\gamma)\oplus   P\oplus P_y )$$
while   the partitioned tuple at $u$  in  $\mathbf B_2''$ is equal to
 $$\mathcal P_{u'' ,2}''=(T\oplus T_y\oplus (b_1^{-1}\omega_e(c^{-1}d)b_2), (\gamma)\oplus   P\oplus P_y).$$
 It is not hard to see that these partitioned tuples are equivalent which implies that $\mathcal B_1''$ is equivalent to $\mathcal B_2''$.  This shows that  $F_1''$ is good if, and only if, $F_2''$ is good.

\smallskip

\noindent{\emph{Case 2.}} $x=\alpha(f_1)=\alpha(f_2)$ is non-peripheral.  In this case  the group of  $y=\omega(f_1)=\omega(f_2)$  in $\mathcal B_i$ ($i=1,2$) is trivial  since $F_1$ and $F_2$ are good. In particular,   the group  of   edges staring at $y$  is also  trivial. 
 
By definition,  $F_1$  and $F_2$ replace  the trivial group   $B_y=B_{y,1}=B_{y,2} $ by the non-trivial groups  
 $B_{y,1}'=\langle b_1^{-1}b_2\rangle $  and $ B_{y,2}'=\langle b_1^{-1}\omega_e(c^{-1}d)b_2\rangle$  
respectively.  Thus a fold of type IIA based on $(f')^{-1}$ can be applied to $\mathbf B_1'$ and to $\mathbf B_2'$.  On the other hand,   (iii) implies that 
$$B_{x',1}' \cap a\alpha_e(A_e)a^{-1} \neq 1 \ \ \text{ and } \ \  B_{x',2}' \cap  \alpha_1  a \alpha_e(c^{-1})\alpha_e(A_e)\alpha_e(c)a^{-1}\alpha_1^{-1} \neq 1.$$ 
Lemma~\ref{lemma:IIAbad} therefore implies that   $F_1'$ and $F_2'$ are bad. 

%-------------------------------------------------------------------
 
\smallskip

\noindent \textbf{(D2)}  It follows from Corollary~\ref{corollary_edges2} that $\mathbf B_1$ can be replaced by an equivalent marked $\mathbb A$-graph $\bar{\mathbf B}_1$  such that the following hold:
\begin{enumerate}
\item[(i)] the labels of $f_1$ and $f_2$  do not change     and the labels of $g_1$ and $g_2$ are of type $(c, \bar{e}, d_1)$ and $(c, \bar{e}, d_2)$, see Figure~\ref{fig:3optionscase(i)}.

\item[(ii)] the marked  $\mathbb A$-graph $\bar{\mathbf B}_1'$ that is obtained from $\bar{\mathbf B}_1$ by the fold $\bar{F}_1$ that identifies $f_1$ and $f_2$ is equivalent to $\mathbf B_1'$. 

\item[(iii)] if $\bar{\mathbf B}_2$ denotes the marked $\mathbb A$-graph that is obtained from $\bar{\mathbf B}_1$ by an A0 move that makes the labels of $g_1$ and $g_2$ coincide, then  the marked  $\mathbb A$-graph that is obtained from $\bar{\mathbf B}_2$ by the fold $\bar{F}_2$ that identifies $g_1$ and $g_2$  is equivalent to $\mathbf B_2'$. 
\end{enumerate} 
\begin{figure}[h]
\centering
\includegraphics[scale=1]{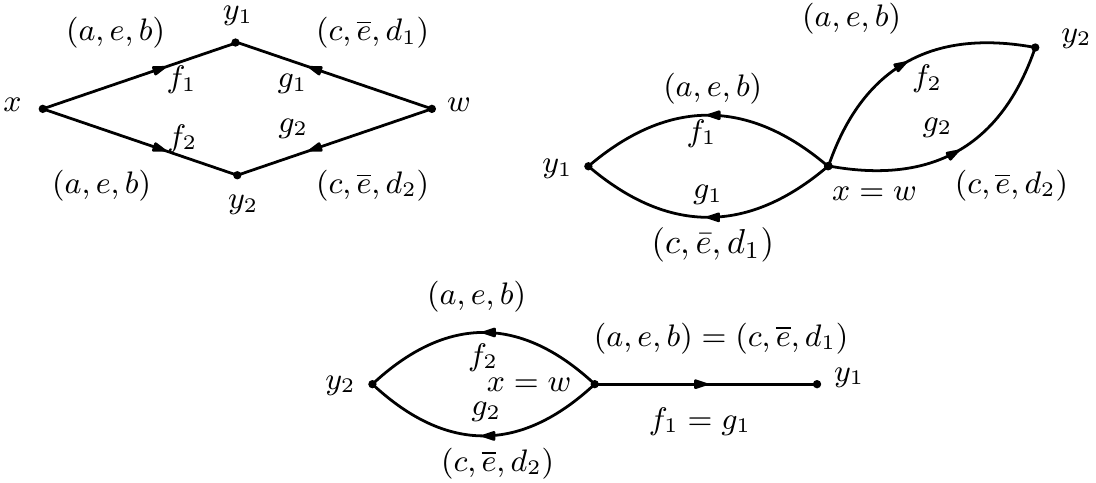}
\caption{The marked $\mathbb A$-graph $\bar{\mathcal B}_1$ in the three configurations that can occur).}
\label{fig:3optionscase(i)}
\end{figure}

Therefore we can assume that  $\mathbf B_1=\bar{\mathbf B}_1$ (resp. $\mathbf B_2=\bar{\mathbf B}_2$) and $\bar{F}_1=F_1$ (resp. $F_2=\bar F_2$).

We first consider the case that $y_1$ and $y_2$ are peripheral. As $F_1$ and $F_2$ are good,  at most one one of them has non-trivial vertex group. Thus there are two cases: If $B_{y_1, 1}=B_{y_2, 1}=1$,  then  we are in case (B) since the marked $\mathbb A$-graph that is obtained from $\mathbf B_1'$ by folding the edges $g_1$ and $g_2$ is equivalent to the marked $\mathbb A$-graph that is obtained from $\mathbf B_2'$ by folding the edges $f_1$ and $f_2$. If either $B_{y_1, 1}\neq 1$ or $B_{y_2,1}\neq 1$, then we are in case (A).  Indeed,  the fold that identifies the edges $g_1$ and $g_2$ (resp. $f_1$ and $f_2$) in $\mathbf B_1'$ (resp. in $\mathbf B_2'$) is bad.

\smallskip

We now consider the case that $y_1$ and $y_2$ are non-peripheral. If one of the vertices $y_1$ or $y_2$ is of orbifold type, then  the fact that $F_1$ and $F_2$ are good  implies that the other  vertex has trivial group.   In this case the same argument as in the  previous paragraph shows that (A) holds.

\smallskip 

Finally assume that $y_1$ and $y_2$ are exceptional vertices, see Figure~\ref{fig:commIA}. Note that the graph does not always look exactly as in the figure, as the other two configurations of Figure~\ref{fig:3optionscase(i)} might also occur. However, in these situations exactly the same arguments can be applied.
\begin{figure}[h!]
\centering
\includegraphics[scale=1]{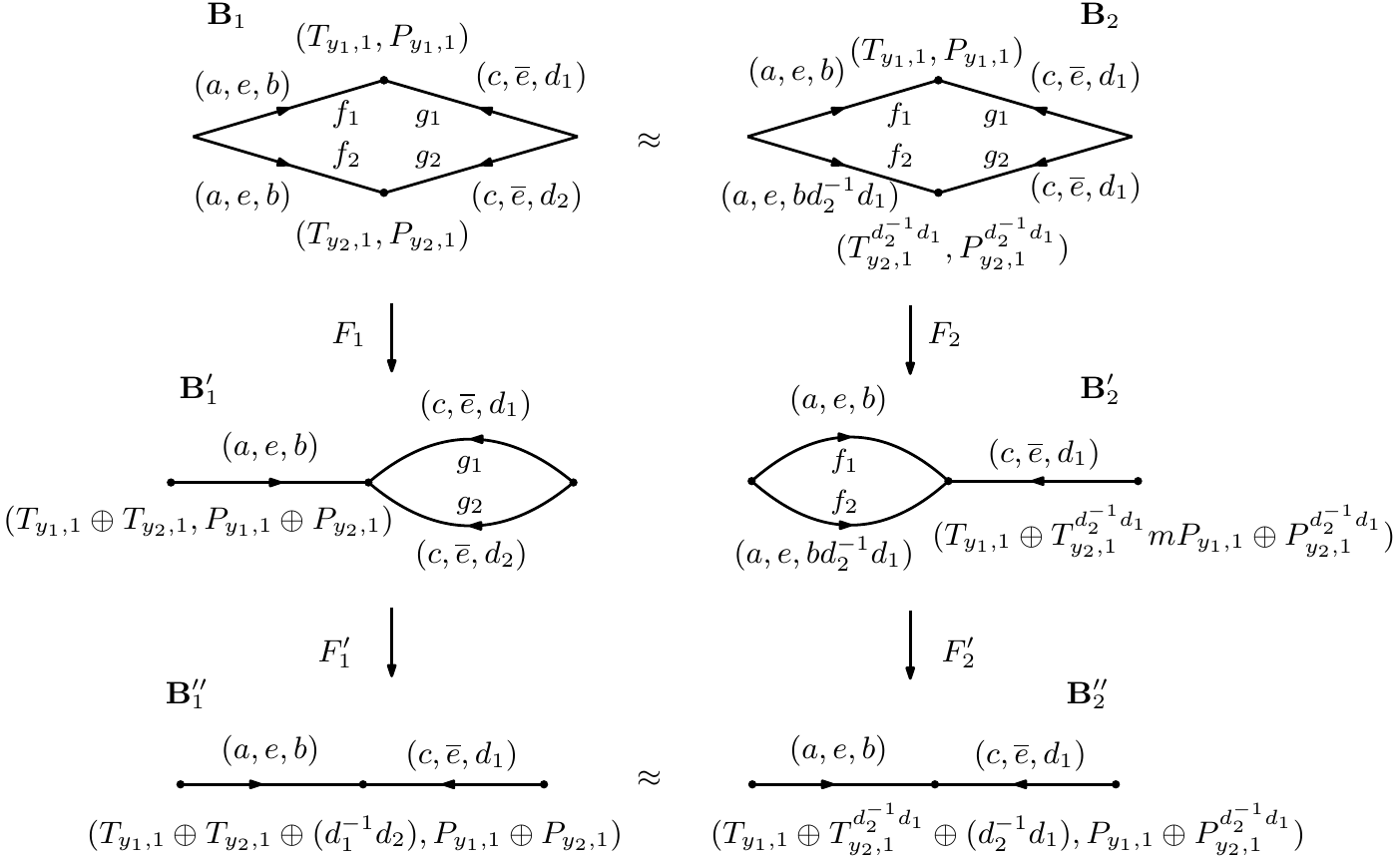}
\caption{ ${F}_1'$ and $F_2'$  produce equivalent marked $\mathbb A$-graphs.}
\label{fig:commIA}
\end{figure}

The loop $g_1\cup g_2$ lies in the core of  the $\mathbb A$-graph $ \mathcal B_1'$ underlying $\mathbf B_1'$.  As  $F_2$ is tame,  the group of at least  one of the edges $g_1$ and $g_2$ is trivial in $\mathcal B_1$, and therefore also in $ {\mathcal B}_1'$. Putting these facts together we conclude that  $core(\mathbf B_1')$ cannot be a special almost orbifold covering with a good marking. Therefore, as $F_1$ is good,  the partitioned tuple  
 $$\mathcal P_{y, 1}':=(T_{y_1, 1}\oplus  T_{y_2, 1} , P_{y_1, 1}\oplus P_{y_2, 1})$$    
is non-critical  of simple type, where $y$ denotes the common image of $y_1$ and $y_2$ in $\mathcal B_1'$. The same argument shows that the partitioned tuple  $$\mathcal P_{y, 2}':=(T_{y_1, 2}\oplus  T_{y_2, 2} , P_{y_1, 2}\oplus P_{y_2, 2})=(T_{y_1, 1}\oplus T_{y_2, 1}^{d_2^{-1}d_1}, P_{y_1, 1}\oplus P_{y_2, 1}^{d_2^{-1}d_1})$$ is of simple type.

Let $F_1'$ denote  the tame elementary fold of type IIIA that identifies the edges $g_1$ and $g_2$ in ${\mathbf B}_1'$ and let ${\mathbf B}_1''$  denote the resulting marked  $\mathbb A$-graph. Similarly,  let  $F_2'$ denote  the  tame elementary fold of type IIIA that identifies the edges $f_1$ and $f_2$  in  $\mathbf B_2'$ and let  $\mathbf B_2''$  denote the resulting  marked $\mathbb A$-graph, see Figure~\ref{fig:commIA}.  Note that  $\mathbf B_1''$ and $\mathbf B_2''$ only differ at  the vertex $y:=\omega(f)=\omega(g)$ where the partitioned tuples  are given by
$$
\mathcal P_{y,1}'':=(T_{y_1, 1}\oplus T_{y_2, 1}\oplus (d_1^{-1}d_2),  P_{y_1, 1}\oplus P_{y_2, 1} )
$$
and 
$$
\mathcal P_{y,2}'':=(T_{y_1, 1}\oplus T_{y_2, 1}^{d_1^{-1}d_2}\oplus (d_1^{-1}d_2), P_{y_1,1}\oplus P_{y_2, 1}^{d_2^{-1}d_1}), 
$$
respectively. One easily verifies that $\mathcal P_{y,1}''$ and $\mathcal P_{y,2} ''$ are equivalent.  Therefore, by \cite[Lemma~3.16]{Dut},  ${\mathbf B}_1''$ and ${\mathbf B}_2''$ are equivalent. 

The equivalence between $\mathbf B_1''$ and $\mathbf B_2''$   implies that the fold  $F_1'$ is good if, and only if, the fold $F_2'$ is good. Therefore we are in case (A) if $F_1'$ is good and we are in case (B) if $F_1'$ is bad.

%------------------------------------------------------------------

\smallskip

\noindent \textbf{(D3)} $F_1$ and $F_2$ are folds of type IIIA  with $f_1=g_1^{-1}$ and $f_2=g_2^{-1}$.  We may assume without loss of generality that  $x=\alpha(f_1)=\alpha(f_2)$ is peripheral (and hence $w=\omega(f_1)=\omega(f_2)$ is non-peripheral). 

It follows from Corollary~\ref{corollary_edges2} that $\mathbf B_1$ can be replaced by an equivalent marked $\mathbb A$-graph $\bar{\mathbf B}_1$ such that the following hold:
\begin{enumerate}
\item[(i)] the label of $f_1$ is $(a, e, b)$  and the label of  $f_2$ is $(a, e, \omega_e(c)b)$ for some non-trivial element  $c\in A_e$.

\item[(ii)] the marked $\mathbb A$-graph that is obtained from $\bar{\mathbf B}_1$ by the fold $\bar{F}_1$ that identifies $f_1$ and $f_2$ is equivalent to $\mathbf B_1'$. 

\item[(iii)] if $\bar{\mathbf B}_2$ denotes the marked $\mathbb A$-graph that is obtained from  $\bar{\mathbf B}_1$ by an A1 move that makes the label  of $f_2$ equal to $(a\alpha_e(c), e, b)$, then  the marked  $\mathbb A$-graph that is obtained from $\bar{\mathbf B}_2$ by the fold $\bar{F}_2$ that identifies $f_1^{-1}$ and $g_2^{-1}$  is equivalent to $\mathbf B_2'$. 
\end{enumerate}
Therefore  we can assume that $\mathbf B_1=\bar{\mathbf B}_1$ (resp. $\mathbf B_2=\bar{\mathbf B}_2$)   and that $F_1=\bar{F}_1$ (resp. $F_2=\bar{F}_2$), see Fig.~\ref{fig:commIIIA}. 
Denote  $\gamma:= b^{-1}\omega_e(c)b\in A_{[w]}$ and $ \gamma':=a^{-1}\alpha_e(c)a=\alpha_e(c)\in A_{[x]}.$  
\begin{figure}[h]
\centering
\includegraphics[scale=1]{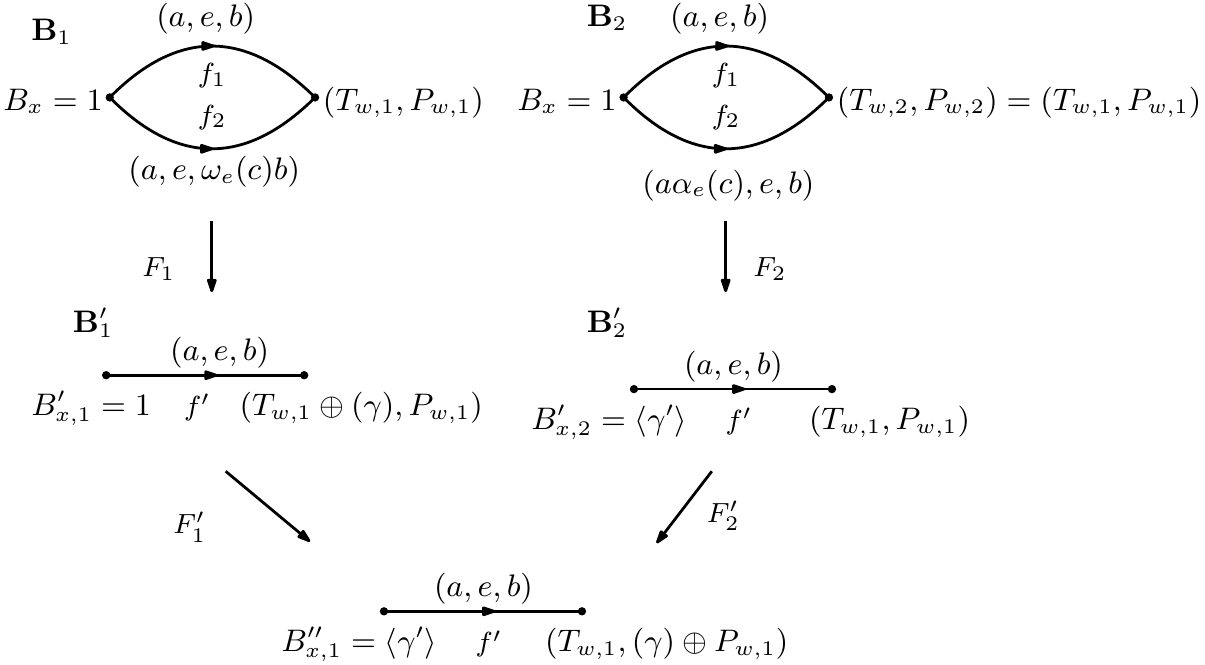}
\caption{$F_1$ and $F_2$ are of type IIIA and $F_1'$ and $F_2'$ are of type IIA.}
\label{fig:commIIIA}
\end{figure}

Since $F_1$ is good, $w$ is exceptional  and 
$\mathcal P_{w, 1}=(T_{w, 1},  P_{w, 1})$ is non-critical of simple type. On the other hand, since $F_2$ is good,  $B_x=1$. Thus $B_g=1$ for all $g\in Star(x, B)$. In particular, $B_{f_1}=B_{f_2}=1$.

We first look at $\mathbf B_1'$. It follows from  Proposition~\ref{prop:02} that  $(T_{w, 1}\oplus (\gamma),  P_{w, 1})$ cannot be of almost orbifold covering type. Thus, as $F_1$ is good,  
$$\mathcal P_{w, 1}'=(T_{w, 1}\oplus (\gamma),  P_{w, 1})$$  is non-critical of simple type.  Observe that the group of $(f')^{-1}$ is trivial (in $\mathcal B_1'$)   and $(f')^{-1}$ violates condition (F2) of folded $\mathbb A$-graphs since the peripheral element  $\gamma $ lies in  $B_{w, 1}'$. 

Let $F_1'$ denote  the type IIA fold based on $(f')^{-1}$. As $\gamma$ is part of a minimal generating tuple of $B_{y, 1}'$ it follows that
 $b^{-1}\omega_e(A_e)b \cap  B_{y,1}'=\langle \gamma\rangle$  and that $(f')^{-1}$ cannot be folded with any edge in $Star(w, B_1')\cap E(\mathcal B_1')$. Thus $F_1'$ is tame elementary.

We now look at $\mathbf B_2'$. Let   $F_2'$ be the type IIA fold based on the edge $f'$.  Note that $B_{h,2}'=1$ for all $h\in Star(w, B_2')$  since $B_h=1$ for all $h\in Star(w, B)$. Thus $F_2'$  is  tame elementary.

It is not hard to see that the marked $\mathbb A$-graph that is obtained from $\mathbf B_1'$ by the fold  $F_1'$ coincides with the marked $\mathbb A$-graph that is obtained from $\mathbf B_2'$ by the fold $F_2'$. Hence $F_1'$ is good if,  and only if,  $F_2'$  is good.  Therefore (A) holds if $F_1'$ and $F_2'$ are bad and  (B) occurs otherwise.

\begin{remark} Observe that $F_2'$ (and hence $F_1'$) is good exactly when   $(T_{w, 1}, (\gamma)\oplus  P_{w, 1}  )$ is non-critical and  of simple type. This follows from the fact that the resulting marked $\mathbb A$-graph cannot be a special almost orbifold cover with a good marking as the vertex $x$ becomes a boundary vertex of a degenerate sub-$\mathbb A$-graph of orbifold type   such that  there is  only one edge  incident at $x$ that has non trivial group, namely the edge $f$. 
\end{remark}

%------------------------------------------------------------------
 
\noindent \textbf{(D4)} As $F_1$ and $F_2$ are good it follows that $B_y=B_z=1$ where  $y=\omega(f)$ and $z=\omega(g)$.  This case is illustrated in Figure~\ref{fig:commnone4ss} in the case that $y=z$.

Let $F_1'$ (resp. $F_2'$) be the fold  of type IIA based on   $g $ (resp. $f)$  of $\mathbf B_1'$ (resp. $\mathbf B_2'$).  Observe that  both folds are tame because the  corresponding peripheral elements  $\gamma_f'$ and $\gamma_g'$ are  part of a minimal generating set of 
 $B_{x, 1}'=B_{x, 2}'=B_{x, 1}=B_{x, 2}\leq A_{[x]}.$  Moreover, both are elementary   because they satisfy  condition (El.2). It is not hard to see that $F_1'$ and $F_2'$ are good (resp.~bad) if and only if $y\neq z$ (resp.~$y=z$), see Fig.~\ref{fig:commnone4}. 
\begin{figure}[h!]
\centering
\includegraphics[scale=1]{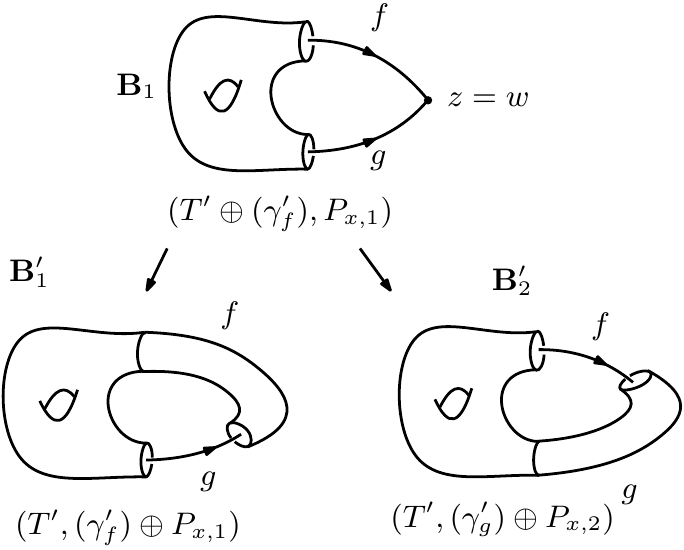}
\caption{$F_1'$ and $F_2'$ are bad tame elementary folds of type   IIA.}
\label{fig:commnone4ss}
\end{figure}

%-------------------------------------------------------------------

\smallskip
 
\noindent \textbf{(D5)} $\mathbf B:=\mathbf B_1=\mathbf B_2$ and     $F_1$ and $F_2$ are of type IIA   with $f\neq g$ such that  $f$ and $ g$ can be folded in $\mathcal B:=\mathcal B_1=\mathcal B_2$ by an elementary fold.  Note that  $F_i$  is not tame after  $F_j$ is applied.  Let $F_i'$ denote the tame elementary fold that identifies $f$ and $g$  in $\mathcal B_i'$.     
Since $F_1$ and $F_2$ are good,   $\omega(f)$ and $\omega(g)$ are  either both peripheral  with trivial group or both  exceptional.   We will give a complete argument for the case they are exceptional and distinct so that $F_i'$ are of type IA.   

The fact that $f$ and $g$ have the same label  implies that the peripheral element $\gamma_{f^{-1}}'$   added  by $F_1$ to the partitioned tuple $\mathcal P_{\omega(f)} $ coincides with the  peripheral element $\gamma_{g^{-1}}'$  added  by $F_2$ to the partitioned tuple $\mathcal P_{\omega(g)}$.   Therefore the marked $\mathbb A$-graph that is obtained from $\mathbf B_1'$ by   $F_1'$  coincides with the marked $\mathbb A$-graph that is obtained from $\mathbf B_2'$ by   $F_2'$.  This  implies that $F_1'$ is good if and  only if $F_2'$ is good. The hole argument is illustrated in Figure~\ref{fig:commnone5s}.
\begin{figure}[h!]
\centering
\includegraphics[scale=1]{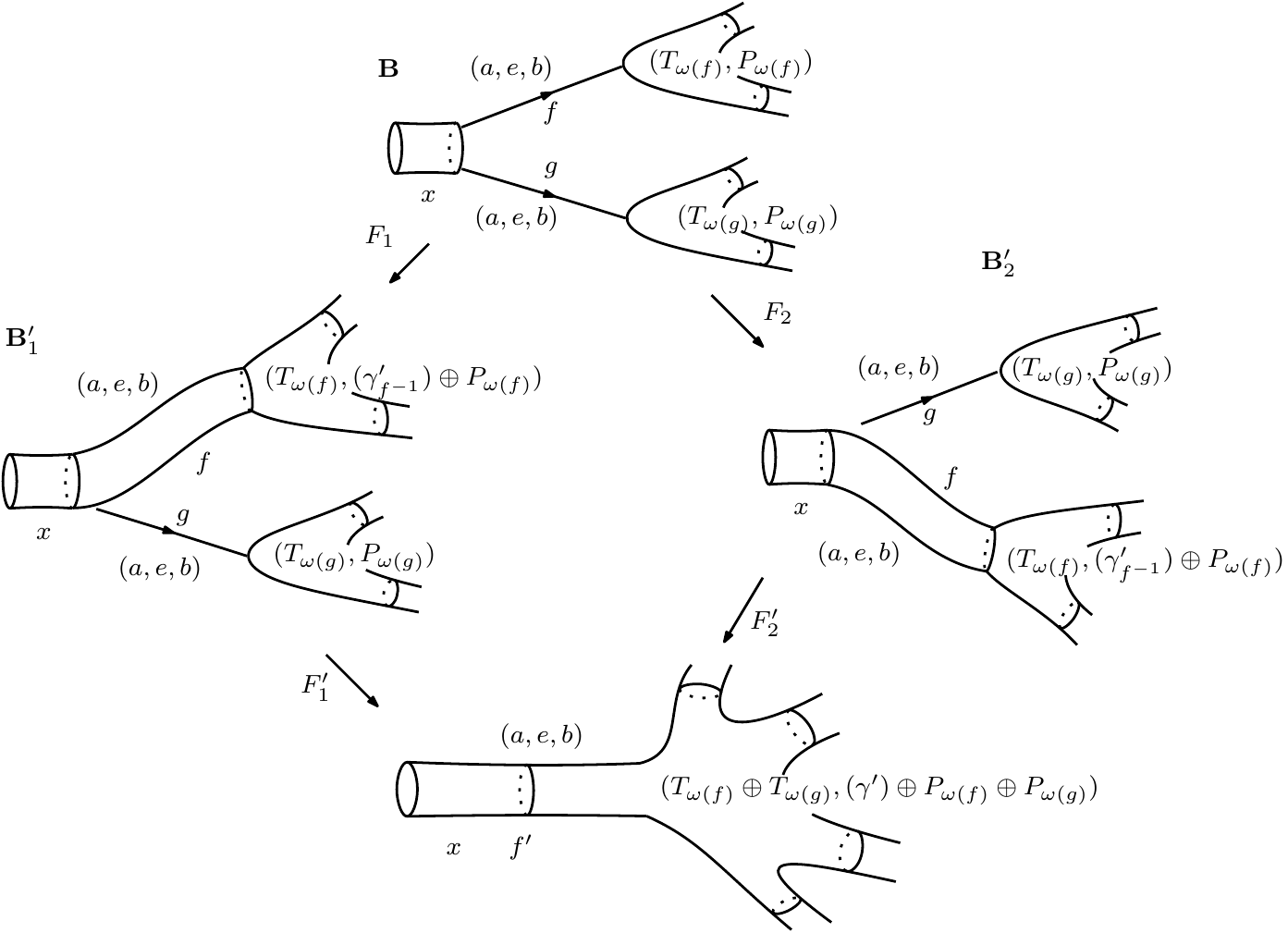}
\caption{$\gamma'=\gamma_{f^{-1}}'=\gamma_{g^{-1}}'$.}
\label{fig:commnone5s}
\end{figure}

%-------------------------------------------------------------------

\smallskip

\noindent \textbf{Generic case.}    Assume that the conclusions of Lemmas~\ref{lemma5cases} and \ref{lemma5cases1}  hold. 
Since $F_1$ and $F_2$ are good, the argument used in the   generic case when  one of the folds is bad   shows that  a fold $F_2'$ (resp. $F_1'$)  that is essentially the same as $F_2$ (resp. $F_1$) can be applied to $\mathbf B_1' $ (resp. $\mathbf B_2'$) and $F_1'$ and $F_2'$ yield the same $\mathbb A$-graph $\mathcal B''$, i.e. $F_1$ and $F_2$ commute.  

Therefore to complete the proof it suffices to  show that $F_1'$  is good if and only if $F_2'$ is good.  This follows by inspecting the various cases.  We will give a sample argument for the case $F_1$ and $F_2$ are of type IIA, and so $F_1'$ and $F_2'$ are also of type IIA.   If $\omega(f)\neq \omega(g)$ then claim is trivial. Thus assume that $y =\omega(f)=\omega(g)=z$.  As $F_1$ and $F_2$ are good, $y$ is either exceptional or peripheral with trivial group.  If $x$ is exceptional then the claim is also trivial.   If $x$ is peripheral, then $F_1$ and  $F_2$  turn  $x$ into a vertex of orbifold type.   Consequently $F_1'$ and $F_2'$  do not induce markings  on $\mathcal B'' $ and therefore are bad.

%--------------------------------------------------------------------

\section{Horizontal Heegaard splittings}\label{section_heegaard}

Throughout this section we assume familiarity with horizontal Heegaard splittings of Seifert manifolds as introduced by Moriah and Schultens \cite{MS}.

Let $M$ be an orientable Seifert 3-manifold over the base orbifold $\mathcal O$, let $\pi:M\to\mathcal O$ be the canoncial map. Given a horizontal Heegaard splitting we obtain the following:
\begin{enumerate}
\item  a surface $\Sigma\subset M$ with a single boundary component, the regular neighborhood of this surface can be thought of as one of the handlebodies of the splitting.
\item An open disk $D\subset\mathcal O$ containing at most one cone point such that $\pi(\Sigma)=\mathcal O\setminus D$ and the map $\pi|_\Sigma:\Sigma\to\mathcal O\setminus D$ is an orbifold cover. 
\end{enumerate}

It follows in particular that the map $\pi|_\Sigma:\Sigma\to\mathcal O$ is an almost orbifold cover. Note that this is an amost orbifold cover of a very special type as $\pi^{-1}(D)$ contains no disks but consists of a single 1-sphere.

\begin{theorem}[Theorem~\ref{thm03}] Let $M$ be a non-small orientable Seifert 3-manifold with orientable base space and $T$ be a tuple corresponding to a horizontal Heegaard splitting. Then  $T$ is irreducible iff the Heegaard splitting is irreducible.
\end{theorem} 

\begin{proof} As reducible Heegard splitting always define reducible generating tuples it suffices to show that the generating tuples corresponding to irreducible Heegaard splitting are irreducible. Thus we may assume that the horizontal Heegaard splitting under consideration is irreducible.

Let $\Sigma$ and $D$ be as above and let $T:=(x_1,\ldots ,x_n)$ a basis of the free group $\pi_1(\Sigma)$, thus $$i_*(T)=(i_*(x_1),\ldots ,i_*(x_n))$$ is a generating set of $\pi_1(M)$ if $i:\Sigma\to M$ is the inclusion map. We need to show that $i_*(T)$ is irreducible. To do so it suffices to show that $$\pi_*\circ i_*(T)=(\pi|_{\Sigma})_*(T)$$ is irreducible.

This claim follows from Theorem~\ref{thm02} once we have verified that the almost orbifold covering $\pi|_{\Sigma}:\Sigma\to \mathcal O$ is special and that $T$ is rigid. 

The rigidity of $T$ is trivial as a simple homology argument shows that the element of surfact group corresponding to the single boundary component of a surface is never part of a basis.

Let now $m=1$ if $D$ contains no cone point, otherwise let $m$ be the order of the single cone point contained in $D$. Let $m_1,\ldots ,m_k$ be the orders of the cone points in $\Sigma\setminus D$ and $d=lcm(m_1,\ldots ,m_k)$. It follows from the construction of horizontal Heegaard splittings that $d$ is the degree of the almost orbifold covering $\pi|_{\Sigma}$, see chapter 4 of \cite{Se}. The irreducibility of the Heegaard splitting morever implies that $m>d$ unless $\mathcal O$ is a surface, a case when the claim is trivial. This follows from \cite{Se}.

To see that $\pi|_{\Sigma}$ is special it remains to show that $d$ does not divide $m$. If $d$ divides $m$ the $\pi|_{\Sigma}$ factors through a finite sheeted cover of $\mathcal O$ which contradicts the assumtion that $(\pi|_{\Sigma})_*(T)$ is a generating tuple of $\pi_1(\mathcal O)$. Thus $d$ divides $m$ and the Theorem is proven.
\end{proof}

Note that the proof actually implies that  Heegaard splittings are irreducible. Thus we reprove Theorem~8.1 of \cite{Se} in the case of non-small Seifert manifolds.

%----------------------------------------------------------

\bibliographystyle{amsplain}

\end{document}